\documentclass[reqno]{amsart}

\usepackage[utf8]{inputenc}
\usepackage{amssymb, amsmath, amsthm, color, enumerate, mathabx, mathrsfs, mathtools}
\usepackage{bbm}
\usepackage{thmtools}

\usepackage{pdfpages}

\usepackage{tcolorbox}
\usepackage{esint} 
\usepackage{marginnote}

\usepackage{array}

\usepackage{upgreek}

\usepackage{circuitikz}

\usepackage{amsbsy}

\usepackage{hyperref}

\numberwithin{equation}{section}
\theoremstyle{plain}
\newtheorem{theorem}{Theorem}[section]
\newtheorem{lemma}[theorem]{Lemma}
\newtheorem{proposition}[theorem]{Proposition}
\newtheorem{corollary}[theorem]{Corollary}
\newtheorem{conjecture}[theorem]{Conjecture}

\newtheorem*{claim}{Claim}
\theoremstyle{definition}
\newtheorem{definition}[theorem]{Definition}

\newtheorem{remark}[theorem]{Remark}

\RenewCommandCopy{\theHtheorem}{\thetheorem}
\RenewCommandCopy{\theHequation}{\theequation}
\RenewCommandCopy{\theHlemma}{\thelemma}
\RenewCommandCopy{\theHdefinition}{\thedefinition}
\RenewCommandCopy{\theHproposition}{\theproposition}

\newcommand{\C}{\mathbb{C}}

\newcommand{\N}{\mathbb{N}}

\newcommand{\R}{\mathbb{R}}

\newcommand{\T}{\mathbb{T}}

\newcommand{\Z}{\mathbb{Z}}

\newcommand{\cE}{\mathcal{E}}

\newcommand{\cI}{\mathcal{I}}

\newcommand{\cK}{\mathcal{K}}
\newcommand{\cL}{\mathcal{L}}

\newcommand{\cS}{\mathcal{S}}

\newcommand{\fkh}{\mathfrak{h}}

\newcommand{\fkG}{\mathfrak{G}}

\newcommand{\fkN}{\mathfrak{N}}

\newcommand{\fkS}{\mathfrak{S}}

\newcommand{\bdC}{\mathbf{C}}

\newcommand{\bdM}{\mathbf{M}}

\newcommand{\bdell}{\boldsymbol{\ell}}
\newcommand{\bl}{{\bdell}}
\newcommand{\tbdell}{\boldsymbol{\tilde{\ell}}}

\newcommand{\bdone}{\mathbf{1}}

\newcommand{\bdnu}{\boldsymbol{\nu}}

\newcommand{\bdrho}{\boldsymbol{\rho}}

\newcommand{\ud}{\mathrm{d}}

\newcommand{\ceil}[1]{\lceil #1 \rceil }
\newcommand{\supp}[1]{\mathrm{supp}_{#1}\,}

\newcommand{\dist}{\mathrm{dist}}

\newcommand{\Dil}{{\operatorname{Dil}}}

\newcommand{\overtheta}{\overline{\Theta}\vphantom{\Theta}}

\usepackage[normalem]{ulem}

\usepackage{tikz}
\usepackage{circuitikz}
\usepackage{pgfplots}
\pgfplotsset{compat=1.18}

\begin{document}

\title[Space curve restriction for orthonormal functions]{Space curve Fourier restriction for orthonormal functions}

\begin{abstract}
We prove a variant of the $L^2$ Fourier restriction inequality for non-degenerate curves in the setting of orthonormal systems. To achieve this, we develop a Hilbert space approach to studying the $L^2$ theory for Fourier restriction for curves. Our arguments also rely on novel $L^p$-improving bounds for associated convolution operators, which take averages over neighbourhoods of cones generated by the curve. The $L^p$-improving bound are established using a variant of the \textit{high-low} method, in particular using induction-on-scale. 
\end{abstract}

\author[J-C. Cuenin]{ Jean-Claude Cuenin }
\address{Jean-Claude Cuenin: Department of Mathematical Sciences, Loughborough University, Loughborough, Leicestershire, LE11 3TU, United Kingdom}
\email{J.Cuenin@lboro.ac.uk}

\author[J. Denson]{Jacob Denson}
\address{Jacob Denson: School of Mathematics and Maxwell Institute for Mathematical Sciences, James Clerk Maxwell Building, The King’s Buildings, Peter Guthrie Tait
Road, Edinburgh, EH9 3FD, UK}
\email{Jacob.Denson@ed.ac.uk}

\author[J. Hickman]{ Jonathan Hickman }
\address{Jonathan Hickman: School of Mathematics and Maxwell Institute for Mathematical Sciences, James Clerk Maxwell Building, The King's Buildings, Peter Guthrie Tait Road, Edinburgh, EH9 3FD, UK.}
\email{jonathan.hickman@ed.ac.uk}

\maketitle




\section{Introduction}




\subsection{Main result}\label{subsec: main result} We say a $C^\infty$ function $\gamma \colon [-2,2] \to \R^n$ is (a parametrisation of) a \textit{smooth, non-degenerate curve} if
\begin{equation*}
    |\gamma^{(1)}(s) \wedge \cdots \wedge \gamma^{(n)}(s)| \neq 0 \qquad \textrm{for all $s \in [-2,2]$.}
\end{equation*}
Fix $\alpha \in C_c^\infty(\R)$ non-negative and not identically zero with $\supp{} \alpha \subseteq [-1,1]$, and define the \textit{Fourier extension operator}
\begin{equation*}
    \cE_\gamma g(x) \coloneq 
    \int_{\R} e^{2 \pi i x \cdot \gamma(t)} g(t) \alpha(t)\,\ud t \qquad \textrm{for all $g \in L^1(\R) + L^\infty(\R)$.} 
\end{equation*}
In this article, we study Fourier restriction estimates for non-degenerate curves in the setting of orthonormal systems. Defining the exponents
\begin{equation*}
    q_n \coloneq n(n+1) \qquad \textrm{and} \qquad  r_n(q) \coloneq \frac{2q}{n(n+1)} \quad \textrm{for $1 \leq q \leq \infty$,}
\end{equation*}
our main result is the following square function estimate.

\begin{theorem} \label{thm: main}
    Let $n \geq 3$ and $\gamma \colon [-2,2] \to \R^n$ be a smooth, non-degenerate curve. For $q > q_n$ and $1 \leq r < r_n(q)$, the inequality\footnote{Given quantities $A,B \geq 0$, and a family of parameters $\theta$, we write $A \lesssim_\theta B$ or $B \gtrsim_{\theta} A$ to indicate that $A \leq C_{\theta} B$ for a constant $C_{\theta} > 0$ depending only on the dimension $n$ and the parameters $\theta$. We write $A \sim_{\theta} B$ to indicate $A \lesssim_{\theta} B$ and $B \lesssim_{\theta} A$. We let $O_{\theta}(B)$ denote a constant $A$ such that $A \lesssim_\theta B$, though different occurences of $O_\theta(B)$ will denote different constants.}
    \begin{equation}\label{eq: orthonormal Fourier rest}
        \left\|\sum_{j \in J}\nu_j |\cE_\gamma g_j|^2 \right\|_{L^{q/2}(\R^n)} \lesssim_{\gamma,\alpha, q, r} \left( \sum_{j \in J} |\nu_j|^r \right)^{1/r}
    \end{equation}
    holds whenever $(g_j)_{j \in J}$ are orthonormal functions in $L^2(\R)$ and $(\nu_j)_{j \in J}$ are complex coefficients. Here the indexing set $J$ can be taken to be finite or countable.
\end{theorem}

For $q \geq q_n$, classical $L^2$ Fourier restriction for curves, together with the triangle inequality, implies that \eqref{eq: orthonormal Fourier rest} holds for $r = 1$ when $(g_j)_{j \in J}$ are arbitrary $L^2$ normalised functions, not necessarily orthonormal. Without additional assumptions, this bound is sharp, in the sense that the $\ell^1$ norm cannot be replaced with an $\ell^r$ norm for any $r > 1$. For $q > q_n$, Theorem~\ref{thm: main} implies such bounds can be improved when the functions $(g_j)_{j \in J}$ are assumed to be orthonormal. In particular, in \eqref{eq: orthonormal Fourier rest} we have a smaller $\ell^r$ norm for $1 < r < r_n(q)$ on the right-hand side. This shows that the extension operator preserves orthogonality, in a quantitative $L^q$ sense.

In \S\ref{sec: nec}, we prove that if \eqref{eq: orthonormal Fourier rest} holds for some exponents $q$ and $r$, then necessarily $q \geq q_n$ and $r \leq q/2n$. Consequently, Theorem~\ref{thm: main} is optimal up to endpoints when $n = 3$. For $n \geq 4$, stronger estimates may hold. In lieu of evidence to the contrary, we conjecture this to be the case.

\begin{conjecture}\label{conj: CDH} 
Let $n \geq 2$ and $\gamma \colon [-2,2] \to \R^n$ be a smooth, non-degenerate curve. For $q \geq q_n$ and $1 \leq r \leq q/2n$, the inequality
\begin{equation}\label{eq: ON Fourier rest 2}
    \left\|\sum_{j \in J}\nu_j |\cE_{\gamma} g_j|^2 \right\|_{L^{q/2}(\R^n)} \lesssim_{\gamma, \alpha} \left( \sum_{j \in J} |\nu_j|^r \right)^{1/r}
\end{equation}
holds whenever $(g_j)_{j \in J}$ are orthonormal functions in $L^2(\R)$ and $(\nu_j)_{j \in J}$ are complex coefficients. Here the indexing set $J$ can be taken to be finite or countable.
\end{conjecture}

Conjecture~\ref{conj: CDH} is known for $n=2$ due to Frank--Sabin~\cite{FS2017}. For $n \geq 3$, Theorem~\ref{thm: main} verifies the bound \eqref{eq: ON Fourier rest 2} for the more restrictive range $q > q_n$ and $r < r_n(q)$. In particular,  when $n = 3$ this proves Conjecture~\ref{conj: CDH} up to endpoints. For $n \geq 4$, there is a significant gap between Theorem~\ref{thm: main} and Conjecture~\ref{conj: CDH}: see Figure~\ref{fig: type}. 

\begin{figure}
    \centering
    \includegraphics[width=0.6\linewidth]{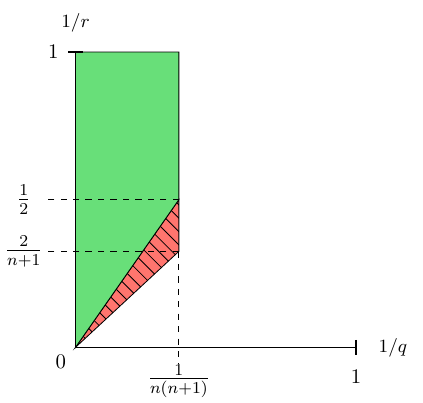}
    \caption{Type diagram for estimates considered. The shaded region depicts the bounds of Theorem~\ref{thm: main}, and the hatched region the bounds conjectured in Conjecture~\ref{conj: CDH}. The main difficulty in extending our results is obtaining bounds with $2 < r < \infty$ when $n \geq 4$ without recourse to interpolation between $r = 2$ and $r = \infty$.}
    \label{fig: type}
\end{figure}

The main difficulty in proving further estimates essentially arises from the fact that the critical exponent
\begin{equation*}
    r_n \coloneq \frac{q_n}{2n} = \frac{n+1}{2}
\end{equation*}
in Conjecture~\ref{conj: CDH} satisfies $r_n > 2$ for $n \geq 4$. As discussed in \S\ref{subsec: methodology} below, our approach follows a duality argument of Frank--Sabin \cite{FS2017} which recasts the problem in terms of Schatten norm estimates. Our arguments are based on proving Hilbert--Schmidt norm bounds in the dual setting. For $n \geq 4$, proving \eqref{eq: ON Fourier rest 2} for $2q/n(n+1) < r \leq q/2n$, once dualised, requires Schatten $r'$-norm bounds for some $1 < r' < 2$. These bounds do not appear to follow via interpolation between Hilbert--Schmidt and trace norm bounds. This is challenging, since obtaining Schatten $r'$-norm bounds for $1 < r' < 2$ directly, without the use of interpolation, typically requires genuinely spectral techniques. 



\subsection{History and motivation for the problem} The estimates considered here are part of a line of research originating in the work of Lieb and Thirring \cite{LiebThirring1975,LiebThirring1976}. These ideas first appeared in the context of Fourier restriction / Strichartz estimates through the work of Frank--Lewin--Lieb--Seiringer \cite{MR3254332} and by Frank--Sabin \cite{FS2017}.

Frank--Sabin \cite{FS2017} established the orthonormal systems analogue of the Stein--Tomas theorem for compact hypersurfaces of nonvanishing Gaussian curvature. They obtained sharp estimates and established several applications of their results to PDE and spectral theory. Furthermore, Frank--Sabin \cite{FS2017} introduced a number of general tools for proving estimates involving orthonormal systems, including a duality principle which reformulates the problem in the setting of Schatten spaces. 

The initial works \cite{MR3254332,FS2017} inspired a large body of literature on Strichartz estimates for orthonormal systems, including \cite{MR3250975,MR3985036,MR4125984,MR4068269,MR4202486,MR4833971,MR4725147}. Furthermore, variable coefficient versions of the Stein--Tomas theorem have been studied in this context in, for example, \cite{MR3682666,MR4656689,cuenin2025spectralclusterboundsorthonormal,jian2025sharpspectralclusterrestrictionbounds}. However, as far as the authors are aware, hitherto orthonormal Fourier restriction has not been considered for space curves. 

Fourier restriction theory for curves has developed in parallel to that for hypersurfaces, although the former is much better understood than the latter in higher dimensions. The full Fourier restriction conjecture for non-degenerate curves was established by Drury~\cite{Drury1985}, building upon earlier work of Prestini~\cite{Prestini1978} and Christ~\cite{Christ1985}. More recently, the theory has seen a resurgence in interest, due to its intimate connection with certain problems in number theory, including Vinogradov's mean value theorem \cite{BDG2016} and estimates for the Riemann zeta function along the critical line \cite{Bourgain2017}.

One motivation for Theorem~\ref{thm: main} arises from the method of proof. The $L^2$ Fourier restriction theory for hypersurfaces relies on Hilbert space techniques, in particular the $T^*T$ method developed by Stein~\cite{Stein1984} and Tomas~\cite{Tomas1975}. By contrast, to the best of our knowledge, Hilbert space techniques have not been effectively exploited in the study of Fourier restriction for curves (see \cite{CFdO} for partial results). Instead, primarily multilinear techniques~\cite{Prestini1978, Christ1985, Drury1985} are used, which yield estimates beyond $L^2$ and are therefore fundamentally different from the Stein--Tomas approach. However, the restriction problem for orthonormal functions is inherently $L^2$-based. Consequently, proving Theorem~\ref{thm: main} requires developing a Hilbert space approach to Fourier restriction for curves. Theorem~\ref{thm: main} implies optimal $L^2 \to L^q$ bounds for $\cE_\gamma$ away from the endpoint, by considering the special case of a single function rather than a collection of orthonormal functions, and thus also provides such a Hilbert space approach to the classical $L^2$ restriction theorem for curves. We explore this topic further in a companion paper~\cite{CDH}.

A second motivation for studying Theorem~\ref{thm: main} is that the bounds for Fourier restriction to curves exhibit qualitatively new behaviour relative to Fourier restriction to curved hypersurfaces. This introduces new considerations, related to the discussion at the end of \S\ref{subsec: main result}, which are not present in the study of orthonormal systems for Fourier restriction to curved hypersurfaces. Much of our arguments centre around Hilbert--Schmidt bounds for certain auxiliary operators, arising from the duality argument used to prove Theorem~\ref{thm: main}. In the work of Frank--Sabin~\cite{FS2017}, Hilbert--Schmidt bounds also play a role, but follow from elementary arguments. By contrast, here the Hilbert--Schmidt bounds are deeper and are obtained as a consequence of certain $L^p$-improving inequalities for convolution operators. The $L^p$-improving estimates are novel, and their proof constitutes the main part of our argument. We discuss this further in the next subsection.




\subsection{Methodology}\label{subsec: methodology} The proof of Theorem~\ref{thm: main} follows the general framework introduced by Frank--Sabin~\cite{FS2017}. Using the duality principle from \cite{FS2017}, we reformulate Theorem~\ref{thm: main} in terms of Schatten norm bounds for the operator $\cE_\gamma (\cE_\gamma)^*$. The main novelty of the paper is our analysis of the Schatten bounds. There are three main features of our argument:\medskip

\noindent \textit{1. Reduction to $L^p$-improving estimates}. Our argument reduces Theorem~\ref{thm: main} to the study of optimal Hilbert--Schmidt norm estimates associated with $\cE_\gamma (\cE_\gamma)^*$. The operator $\cE_\gamma (\cE_\gamma)^*$ is given by convolution against $\cK_{\gamma} \coloneq \widecheck{\sigma}_{\gamma}$, where $\sigma_{\gamma}$ is a smooth measure on the underlying curve $\gamma$ and $\smash{\widecheck{\sigma}_{\gamma}}$ denotes its inverse Fourier transform. In Lemma~\ref{lem: Lp improv red}, we show the desired Hilbert--Schmidt bounds arise from $L^p \to L^{p'}$ estimates for
\begin{equation*}
    T_{\gamma} f \coloneq |\cK_{\gamma}|^2 \ast f.
\end{equation*} 
This reduction is elementary, yet the relevance of $L^p$-improving estimates is a new and interesting feature of the problem compared with the hypersurface case \cite{FS2017}.\medskip

\noindent \textit{2. $L^p$-improving estimates: decomposition of the kernel}. It remains to bound the $L^p \to L^{p'}$ operator norm for $T_{\gamma}$. At its heart, these bounds rely on a rather intricate analysis that exploits the \emph{direction-sensitive} Fourier decay of the measure $\sigma_\gamma$.

Because $\gamma$ is non-degenerate, van der Corput's Lemma implies the uniform bound 
\begin{equation}\label{eq: slow decay}
    |\widecheck{\sigma}_{\gamma}(x)| \lesssim (1+|x|)^{-1/n} \qquad \textrm{for all $x \in \R^n$.}
\end{equation} 
The estimate is sharp, as there is a cone of points $x \in \R^n$ where the reverse inequality holds for $|x|$ large. Yet naive application of this uniform bound does not lead to sharp results. This is because the cone of points where \eqref{eq: slow decay} is sharp is only two-dimensional. For most $x \in \R^n$, the significantly stronger decay estimate 
\begin{equation*}
    |\widecheck{\sigma}_\gamma(x)| \lesssim (1 + |x|)^{-1/2}
\end{equation*}
holds. More generally, for each $2 \leq m \leq n-1$ there is an $n+1-m$ dimensional cone $\Gamma_m \subseteq \R^n$ such that
\begin{equation*}
   |\widecheck{\sigma}_\gamma(x)| \lesssim (1 + |x|)^{-1/m} 
\end{equation*}
holds away from $\Gamma_m$.
 
To obtain more effective bounds, we carefully decompose $\R^n$ into different regions associated to the cones $\Gamma_2,\dots,\Gamma_{n-1}$, and thereby according to the decay of $\widecheck{\sigma}_{\gamma}$. In particular, we use a pointwise variant of a decomposition argument developed in \cite{PS2007, KLO2023, GovindanSheri2024} and related works to study averaging operators associated to non-degenerate curves. \medskip

\noindent \textit{3. $L^p$-improving estimates: a high-low argument.} The argument alluded to above induces a decomposition of the kernel $\smash{\cK_\gamma}$ into pieces supported on various anisotropic neighbourhoods of the cones discussed above. We thus reduce matters to studying $L^p \to L^{p'}$ bounds for averaging operators associated to these neighbourhoods.

Here we exploit the curvature of the cones to obtain an $L^p$-improving estimate in the spirit of Christ~\cite{Christ1998}. The work of Christ~\cite{Christ1998} introduced powerful combinatorial methods for establishing $L^p$-improving estimates. However, owing to the anisotropy of our setup, theses methods do not easily adapt to our problem. Instead, we rely on Fourier analysis. This is possible partly because we are only interested in bounds on the line of duality: that is, we only require $L^p \to L^{p'}$ bounds rather than more general $L^p \to L^q$ bounds. Such bounds are more amenable to Fourier analytic methods, as they may be obtained by interpolating between $L^2 \to L^2$ estimates (proved using Fourier analysis) and $L^1 \to L^\infty$ estimates (proved by direct analysis of the kernel).

Our proof employs a variant of the \emph{high-low method} (see, for instance, \cite{GWZ2020}), which has seen many recent applications in the harmonic analysis literature. The method involves breaking the operator into high and low frequency components. The high frequency component can be estimated directly using Plancherel's theorem and certain disjointness properties of the kernel's frequency support. The low frequency component can be estimated either trivially or via an appeal to induction-on-scale. We refer to \S\ref{sec: Lp improv III} for the full details.




\subsection*{Acknowledgements} The first author is supported by New Investigator Award EP/X011488/1. The second and third authors are supported by New Investigator Award UKRI097.

\subsection*{Tool and Computational Resource Disclosure} Very minor use of AI tools were made to help address \LaTeX\ formatting issues.



\section{Spectral theory and reduction to \texorpdfstring{$L^p$-improving}{improving} estimates} 

To prove \eqref{eq: orthonormal Fourier rest}, we use duality to convert the inequality into certain \emph{Schatten space} bounds. Thus, we begin by recalling relevant results and definitions from spectral theory.




\subsection{Schatten norms} Recall that for two Hilbert spaces $\fkh_1$ and $\fkh_2$, the \emph{singular values} of a compact operator $A: \fkh_1 \to \fkh_2$ are the eigenvalues of the non-negative operator $(AA^*)^{1/2}$. Since $A$ is compact, the eigenvalues of $(AA^*)^{1/2}$ form a countable set of non-negative numbers accumulating only at the origin. For $m \in \N$, we let $s_m(A)$ denote the $m$th largest singular value of $A$, repeated up to multiplicity.

For $A$ as above and $1\leq  p < \infty$, we define the $\fkS^p$ \emph{Schatten norm}
\begin{equation*}
  \| A \|_{\fkS^p(\fkh_1, \fkh_2)} \coloneq \Bigg( \sum_{m = 1}^\infty s_m(A)^p \Bigg)^{1/p}.  
\end{equation*}
We also let
\begin{equation*}
 \| A \|_{\fkS^{\infty}(\fkh_1, \fkh_2)} \coloneq \max_{m \in \N} s_m(A) = s_1(A) = \| A \|_{\fkh_1 \to \fkh_2},    
\end{equation*}
where the last norm denotes the operator norm of $A$.
For $1 \leq p \leq \infty$, we then let $\fkS^p(\fkh_1,\fkh_2)$ denote the space of all compact operators $A \colon \fkh_1 \to \fkh_2$ for which $\|A\|_{\fkS^{p}(\fkh_1,\fkh_2)} < \infty$. For a Hilbert space $\fkh$, we also define $\fkS^{p}(\fkh) \coloneq \fkS^{p}(\fkh,\fkh)$.




\subsection{Hilbert--Schmidt integral operators}\label{subsec: HS} For $m$, $n \in \N$, we recall that the Schatten space $\fkS^2(L^2(\R^m),L^2(\R^n))$ is precisely the space of Hilbert--Schmidt operators. That is, a bounded operator $T \colon L^2(\R^m) \to L^2(\R^n)$ is an element of $\fkS^2(L^2(\R^m),L^2(\R^n))$ if and only if there exists a function $\cK \in L^2(\R^n \times \R^m)$ which is an \emph{integral kernel} for $T$, in the sense that
\begin{equation*}
    Tf(x) = \int_{\R^m} \cK(x,y) f(y)\, \ud y \qquad \textrm{for all $f \in L^2(\R^m)$.}
\end{equation*}
Here the integral is well defined for almost every $x \in \R^n$ by Fubini's theorem, uniquely specifying an element of $L^2(\R^n)$ by H\"{o}lder's inequality. The function $\cK$ is necessarily unique, and 
\begin{equation*}
  \| T \|_{\fkS^2(L^2(\R^m),L^2(\R^n))} = \| \cK \|_{L^2(\R^n \times \R^m)}.  
\end{equation*}
In particular, one may calculate the $\fkS^2(L^2(\R^m), L^2(\R^n))$ norm of an operator by producing an integral kernel for the operator and taking its $L^2(\R^n \times \R^m)$ norm. See \cite[pp. 108-111]{GK_book} or \cite[Thm. 3.8.5]{MR3364494} for details.




\subsection{Duality and interpolation}\label{subsec: duality and interpolation} Suppose $2 \leq q \leq \infty$, and define $p \in [1,\infty]$ by the H\"older relationship $1/(2p) + 1/q = 1/2$. Any $W \in L^{2p}(\R^n)$ then defines a bounded operator from $L^q(\R^n)$ to $L^2(\R^n)$ acting by pointwise multiplication, with operator norm $\| W \|_{L^{2p}(\R^n)}$. Abusing notation, we also denote this operator by $W$. Thus, it follows that
\begin{equation*}
   \|\overline{W}\|_{L^2(\R^n) \to L^{q'}(\R^n)} = \|W\|_{L^q(\R^n) \to L^2(\R^n)} = \|W\|_{L^{2p}(\R^n)} \quad \textrm{for $W \in L^{2p}(\R^n)$.}
\end{equation*}
We now recall a duality principle of Frank--Sabin~\cite{FS2017}.

\begin{proposition}[{\cite[Lemma 3]{FS2017}}]\label{prop: duality principle}
    Let $1 \leq p$, $q$, $r \leq \infty$ satisfy $1/(2p) + 1/q = 1/2$ and $\fkh$ be a complex separable Hilbert space. For a bounded operator $A \colon \fkh \to L^q(\R^n)$ and $C > 0$, the following are equivalent:
    \begin{enumerate}[\normalfont(a)]
        \item For any orthonormal elements $(u_j)_{j\in J}$ of $\fkh$ and any complex sequence $(\nu_j)_{j\in J}$, we have
        \begin{equation*}
            \Bigg\| \sum_{j \in J}\nu_j |Au_j|^2 \Bigg\|_{L^{q/2}(\R^n)} \leq C \Bigg( \sum_{j \in J} |\nu_j|^r \Bigg)^{1/r}.  
        \end{equation*}

        \item For any $W \in L^{2p}(\R^n)$, we have
        \begin{equation*}
            \| W  AA^*  \overline{W}\|_{\fkS^{r'}(L^2(\R^n))}\leq C \|W\|_{L^{2p}(\R^n)}^2.
        \end{equation*}
    \end{enumerate}
\end{proposition}

In this paper, we apply Proposition~\ref{prop: duality principle} with $\fkh = L^2(\R)$ and with $A = \cE_\gamma$. Recall that $q_n \coloneq n(n+1)$ is the exponent in Theorem~\ref{thm: main}. Define the operator $U_\gamma$ and exponent $p_n$ by
%
%
\begin{equation}\label{eq: U gamma}
    U_\gamma \coloneq \cE_\gamma (\cE_\gamma)^*  \qquad \textrm{and} \qquad p_n' \coloneq \frac{n(n+1)}{2}.
\end{equation}
We note then that
\[ \frac{1}{2p_n} + \frac{1}{q_n} = \frac{1}{2}. \]
Proposition~\ref{prop: duality principle} may be combined with an interpolation argument to reduce the proof of Theorem~\ref{thm: main} to the following theorem.

\begin{theorem}\label{thm: endpoint}
    For $n \geq 3$ and $1 \leq p < p_n$, the inequality
    \begin{equation}\label{eq: endpoint}
        \| W_1 U_\gamma W_2 \|_{\fkS^2(L^2(\R^n))} \lesssim_{\gamma, \alpha} \| W _1\|_{L^{2p}(\R^n)}\| W_2 \|_{L^{2p}(\R^n)}
    \end{equation}
    holds whenever $W_1$, $W_2 \in L^{2p}(\R^n)$.
\end{theorem}

\begin{proof}[Proof (of Theorem~\ref{thm: main}, assuming Theorem~\ref{thm: endpoint})] Let $W_1$, $W_2 \in L^2(\R^n)$. Since 
\begin{equation*}
   W_1 U_\gamma W_2 = (W_1 \cE_\gamma) (\overline{W_2}\cE_\gamma)^*, 
\end{equation*}
a variant of H\"older's inequality in Schatten spaces (see \cite[Theorem 2.8]{MR2154153}) implies
\begin{equation*}
    \| W_1 U_\gamma W_2 \|_{\fkS^1(L^2(\R^n))} \leq \| W_1 \cE_\gamma \|_{\fkS^2(L^2(\R),L^2(\R^n))} \| \overline{W_2} \cE_\gamma \|_{\fkS^2(L^2(\R),L^2(\R^n))}.
\end{equation*}
For $W \in L^2(\R^n)$, it is a simple matter to verify $W \cE_{\gamma}$ has integral kernel 
\begin{equation*}
    \cK(x,t) \coloneq W(x) \alpha(t) \quad \textrm{satisfying} \quad \|\cK\|_{L^2(\R^n \times \R)} \lesssim_{\alpha} \|W\|_{L^2(\R^n)}. 
\end{equation*}
Thus, by the discussion in \S\ref{subsec: HS}, we have
\begin{equation} \label{eq: Schatten interp 1}
    \| W_1 U_\gamma W_2 \|_{\fkS^1(L^2(\R^n))} \lesssim_{\alpha} \| W_1 \|_{L^2(\R^n)} \| W_2 \|_{L^2(\R^n)}.
\end{equation}
We may then interpolate\footnote{More precisely, we apply the bilinear complex interpolation method to $(W_1,W_2)\mapsto W_1 U_\gamma W_2$ (see, for instance, \cite[Theorem 4.4.1]{MR482275}). For this, we use the interpolant identity $(\mathfrak{S}^{p_1},\mathfrak{S}^{p_2})_{\theta}=\mathfrak{S}^{p_{\theta}}$, which holds when $1/p_\theta = (1-\theta)/p_1 + \theta/p_2$ and $1\leq p_1,p_2\leq\infty$ (see \cite[p. 137, Theorem 13.1]{GK_book} or \cite[Theorem 2.9]{MR2154153}).} between \eqref{eq: endpoint} and \eqref{eq: Schatten interp 1} to conclude that for $q > q_n$, $1 \leq r < r_n(q)$, and $1/(2p) + 1/q = 1/2$, we have
\begin{equation*}
     \| W_1 U_\gamma W_2 \|_{\fkS^{r'}(L^2(\R^n))} \lesssim_{\gamma, \alpha} \| W_1 \|_{L^{2p}(\R^n)}\| W_2 \|_{L^{2p}(\R^n)}.
\end{equation*}
By the duality principle from Proposition~\ref{prop: duality principle}, this implies \eqref{eq: orthonormal Fourier rest} for the same range of exponents as in Theorem~\ref{thm: main}, which completes the proof.
\end{proof}

\subsection{Reduction to \texorpdfstring{$L^p$}{}-improving} 

Having reduced Theorem~\ref{thm: main} to Theorem~\ref{thm: endpoint}, we further reduce the problem to bounding certain convolution operators with nonnegative kernels.

\begin{lemma}\label{lem: Lp improv red}
 Let $\cK \colon \R^n \to \C$ be measurable and define $T f \coloneq |\cK|^2 * f$ for all measurable functions $f \geq 0$. Assume that, for some $1 \leq p \leq 2$ and $C > 0$, and all non-negative $f \in L^p(\R^n)$, the
  $L^p$-improving bound
  \begin{equation}\label{eq: Lp improv red a}
    \| T f \|_{L^{p'}(\R^n)} \leq C \, \| f \|_{L^p(\R^n)}
  \end{equation}
  holds. Then for all
  $W_1$, $W_2 \in L^{2p}(\R^n)$ the function
  \begin{equation*}
    K(x,y) \coloneq W_1(x) \, \cK(x-y) \, W_2(y)
  \end{equation*}
  belongs to $L^2(\R^n \times \R^n)$, with
  \begin{equation}\label{eq: Lp improv red b}
    \| K \|_{L^2(\R^n \times \R^n)}
      \leq \sqrt{C} \, \| W_1 \|_{L^{2p}(\R^n)} \| W_2 \|_{L^{2p}(\R^n)}.
  \end{equation}
  Equivalently, the integral operator with kernel $K$, denoted by $W_1 U W_2$, is Hilbert--Schmidt, and
  \[ \| W_1 U W_2 \|_{\fkS^2(L^2(\R^n)}  \leq \sqrt{C} \, \| W_1 \|_{L^{2p}(\R^n)} \| W_2 \|_{L^{2p}(\R^n)}. \]
\end{lemma}

\begin{proof} Observe that
\begin{align*}
\| K \|_{L^2(\R^n \times \R^n)}^2
        &= \int_{\R^n} \int_{\R^n} |W_1(x)|^2 |\cK(x-y)|^2 |W_2(y)|^2\, \ud x\, \ud y \\
        &= \int_{\R^n} |W_1(x)|^2\, T\big(|W_2|^2\big)(x)\,\ud x.
\end{align*}
Applying H\"{o}lder's inequality and then the hypothesised bound \eqref{eq: Lp improv red a}, we deduce
\begin{align*}
   \| K \|_{L^2(\R^n \times \R^n)}^2 
    &\leq \| |W_1|^2 \|_{L^{p}(\R^n)} \big\| T \big( |W_2|^2 \big) \big\|_{L^{p'}(\R^n)} \\
    &\leq C\, \| W_1 \|_{L^{2p}(\R^n)}^2 \| W_2 \|_{L^{2p}(\R^n)}^2.
\end{align*}
Taking square roots yields \eqref{eq: Lp improv red b}. Finally, by the discussion in \S\ref{subsec: HS}, the operator $W_1 U W_2$ is Hilbert--Schmidt and $\| W_1 U W_2 \|_{\fkS^2(L^2(\R^n))} = \| K \|_{L^2(\R^n \times \R^n)}$.
\end{proof}

We apply Lemma~\ref{lem: Lp improv red} in the special case of the operator $U_{\gamma} = \cE_{\gamma} (\cE_{\gamma})^*$ introduced in \eqref{eq: U gamma}. A simple computation shows that
\begin{equation}\label{eq: kernel}
U_{\gamma}f = \cK_{\gamma} \ast f \qquad \textrm{where} \qquad    \cK_{\gamma}(x) \coloneq \int_{\R} e^{2 \pi i x \cdot \gamma(t)} |\alpha(t)|^2\,\ud t. 
\end{equation}
Consistent with the discussion in \S\ref{subsec: methodology}, we may interpret $\cK_{\gamma}$ as the inverse Fourier transform $\widecheck{\sigma}_{\gamma}$ of the pushforward measure $\sigma_\gamma \coloneq \gamma_*(|\alpha|^2\,\ud t)$. Lemma~\ref{lem: Lp improv red} reduces matters to proving $L^p$-improving bounds for the operator $T_{\gamma} f \coloneq |\cK_{\gamma}|^2 \ast f$.

For the purposes of the proof, it is convenient to embed $\cK_{\gamma}$ above in a larger class of kernels. To this end, we introduce the following symbol class.

\begin{definition}\label{dfn: intial symb} For $A \geq 1$, let $\cS(A)$ denote the class of all\footnote{Here $C^{\infty}_tL^{\infty}_x(\R \times \R^n)$ denotes the space of bounded, measurable functions $a \colon \R \times \R^n$ which are smooth in the first variable.} $a \in C^{\infty}_tL^{\infty}_x(\R \times \R^n)$ non-negative which satisfy
\begin{equation*}
  \supp{} a \subseteq [-1, 1] \times \R^n \quad \textrm{and} \quad  \sup_{(t,x) \in \R \times \R^n}|\partial_t^j a(t, x)| \leq A \qquad \textrm{for $0 \leq j \leq n$.}
\end{equation*}
\end{definition}

\begin{definition} 
For $\gamma \colon [-2,2] \to \R^n$ a smooth, non-degenerate curve and $a \in \cS(A)$ for $A \geq 1$, we let $T_{\gamma}[a]$ be the operator given by convolution with $|\cK_{\gamma}[a]|^2$, where
\begin{equation*}
     \cK_{\gamma}[a](x) \coloneq \int_{\R} e^{2 \pi i \gamma(t) \cdot x} a(t, x) \, \ud t.
\end{equation*}
\end{definition}

In view of Lemma~\ref{lem: Lp improv red}, Theorem~\ref{thm: endpoint} is a consequence of the following theorem.

\begin{restatable}{theorem}{symbolconvolutionbound} \label{thm: symbolconvolutionbound}
     Let $n \geq 3$ and $\gamma \colon [-2,2] \to \R^n$ be a smooth, non-degenerate curve and $a \in \cS(A)$ for $A \geq 1$. For $1 \leq p < p_n$ and non-negative $f \in L^p(\R^n)$, we have
     \begin{equation*}
         \|T_{\gamma}[a] f \|_{L^{p'}(\R^n)} \lesssim_{\gamma, A, p} \| f \|_{L^p(\R^n)}. 
     \end{equation*}
     Thus, $T_\gamma[a]$ uniquely extends to a bounded linear operator from $L^p(\R^n)$ to $L^{p'}(\R^n)$.
\end{restatable}

We postpone the proof of Theorem~\ref{thm: symbolconvolutionbound} to the following sections. Assuming it holds, it is a simple matter to deduce Theorem~\ref{thm: endpoint} holds and therefore conclude the proof of Theorem~\ref{thm: main}.

\begin{proof}[Proof (of Theorem~\ref{thm: endpoint}, assuming Theorem~\ref{thm: symbolconvolutionbound})] Defining $a \colon \R \times \R^n \to [0, \infty)$ by $a(t,x) \coloneq |\alpha(t)|^2$ for all $(t,x) \in \R \times \R^n$, it follows that $a \in \cS(A)$ for some $A = O_{\alpha}(1)$. Furthermore, $\cK_{\gamma}[a] = \cK_{\gamma}$, where the latter is defined as in \eqref{eq: kernel}. Thus, $T_{\gamma}[a] = T_{\gamma}$ and therefore Theorem~\ref{thm: symbolconvolutionbound} guarantees $T_\gamma$ satisfies the hypothesis \eqref{eq: Lp improv red a} of Lemma~\ref{lem: Lp improv red} for some $C = O_{\gamma, \alpha}(1)$ and all $1 \leq p < p_n$. Lemma~\ref{lem: Lp improv red} now implies the desired bound \eqref{eq: endpoint} holds for all $1 \leq p < p_n$.    
\end{proof}

\begin{remark} In the proof of the proposition, we decompose the function $a$ into pieces supported in various regions of $\R \times \R^n$. It is therefore natural to consider the more general context where $a$ is allowed to depend on the $x$ variable, rather than just considering the particular case used in the proof of Theorem~\ref{thm: endpoint}.
\end{remark}

We have reduced Theorem~\ref{thm: main} to the $L^p$-improving estimate in Theorem~\ref{thm: symbolconvolutionbound}. The use of $L^p$-improving bounds to control $\| W_1 U_{\gamma} W_2 \|_{\fkS^2(L^2(\R^n))}$ is apparently a new phenomenon, particular to the study of restriction estimates for curves. For the Fourier extension operator $\cE_{\Sigma}$ associated to a hypersurface $\Sigma$ of nonvanishing Gaussian curvature, as treated by Frank--Sabin~\cite{FS2017}, the analogous Hilbert--Schmidt norm can be effectively estimated simply by pointwise bounding the kernel of $\cE_\Sigma (\cE_{\Sigma})^*$ using the Fourier decay of surface measures on $\Sigma$. This simple argument does not effectively generalise to the setting of space curves, owing to the anisotropic nature of the Fourier decay in this setting (see the discussion in \S\ref{subsec: methodology} above). 

The proof of Theorem~\ref{thm: symbolconvolutionbound} forms the bulk of the paper and is presented in \S\S\ref{sec: Lp improv I}-\ref{sec: Lp improv III} below. As an indication of the complexity of the problem, we note that Young's convolution inequality does not capture the optimal $L^p$-improving estimates for $T_\gamma$.\footnote{This is in contrast to the proof of \cite[Theorem 9]{FS2017}, where a weak-type version of Young's inequality (the Hardy-Littlewood-Sobolev inequality) is used to prove Schatten versions of Strichartz estimates.} Indeed, classical work of Arkhipov--Chubarikov--Karatsuba \cite{ACK_book, ACK1979} implies that the kernel $|\cK_{\gamma}|^2$ of $T_\gamma$ belongs to $L^{p'/2}(\R^n)$ if and only if $p' > (n^2 + n + 2)/2$, and thus Young's inequality can only guarantee $L^p(\R^n) \to L^{p'}(\R^n)$ boundedness in this range. This is much more restrictive than the optimal range $p' \geq p_n' = (n^2 + n)/2$, which we obtain up to the endpoint. Thus, we must develop a more nuanced approach to analyse $T_\gamma$.




\section{Proof of \texorpdfstring{$L^p$-improving}{Lp improving} I: Technical preliminaries}\label{sec: Lp improv I}

We now turn to the proof of Theorem~\ref{thm: symbolconvolutionbound}, which will occupy this section and the following two sections. Here we discuss various technical preliminaries.




\subsection{Quantitative non-degeneracy} It is helpful to quantify the non-degeneracy condition imposed on the curve $\gamma$. 

\begin{definition}\label{dfn: G class} For $B \geq 1$ and $1 \leq N \leq n$, we let $\fkG(B; N)$ denote the set of all smooth, non-degenerate curves $\gamma \colon [-2,2] \to \R^n$ such that
\begin{equation} \label{eq: G class 1}
   \max_{1 \leq j \leq 2n} \| |\gamma^{(j)}| \|_{L^\infty([-2,2])} \leq B \quad \textrm{and} \quad
    \inf_{s \in [-1,1]}|\gamma^{(1)}(s) \wedge \cdots \wedge \gamma^{(N)}(s)| \geq B^{-1}.
\end{equation}
For $n = N$, we write $\fkG(B) \coloneq \fkG(B; n)$. 
\end{definition}

Clearly, given any smooth, non-degenerate curve $\gamma \colon [-2,2] \to \R^n$, there exists some $B = O_{\gamma}(1)$ such that $\gamma \in \fkG(B)$. 

For $N = n$, the wedge product in \eqref{eq: G class 1} can be interpreted as the determinant of the $n \times n$ matrix
\begin{equation}\label{eq: M matrix}
    \bdM_{\gamma}(s) \coloneq \begin{pmatrix}
\gamma^{(1)}(s) & \cdots & \gamma^{(n)}(s)
    \end{pmatrix},
\end{equation}
whose columns are given by the first $n$ derivatives of the curve $\gamma$. Given $\gamma \in \fkG(B)$ for $B \geq 1$, by \eqref{eq: G class 1}, for each $s \in [-1,1]$ it follows that $\bdM_{\gamma}(s)$ and its transpose $\bdM_{\gamma}(s)^{\top}$ are invertible. Moreover, by combining Cramer's rule with \eqref{eq: G class 1}, the inverse transpose $\bdM_{\gamma}(s)^{-\top}$ has operator norm $B^{O(1)}$. This implies that
    \begin{equation} \label{eq: G class 4}
        |x| = |\bdM_{\gamma}(s)^{-\top}\bdM_{\gamma}(s)^{\top}x| \leq B^{O(1)} \max_{1 \leq j \leq n} |\gamma^{(j)}(s) \cdot x| 
    \end{equation}
    for $x \in \R^n$, $s \in [-1,1]$.
%




\subsection{Plates} Here we introduce a family of parallelepipeds associated to a smooth, non-degenerate curve, which we use in the kernel decomposition in \S\ref{sec: Lp improv II} below.

\begin{definition}\label{dfn: plate} Let $\gamma \colon [-2,2] \to \R^n$ be a smooth, non-degenerate curve. For $s \in [-2,2]$, $\bdrho = (\rho_1,\dots,\rho_n) \in (0, \infty)^n$ and $C \geq 1$, define the parallelepiped
\begin{equation*}
    \Theta(s,\bdrho, C) \coloneq \big\{ x \in \R^n: |\gamma^{(j)}(s) \cdot x| \leq C \rho_j, \, 1 \leq j \leq n\big\}.
\end{equation*}
Furthermore, for integers $3 \leq M \leq N \leq n$, define
\begin{equation*}
  \Theta_{M,N}(s, \bdrho, C) \coloneq \Theta(s,\bdrho, C)  \cap \big\{ x \in \R^n: |\gamma^{(j)}(s) \cdot x| \geq (2C)^{-1}\rho_j, \, M \leq j \leq N\big\}.  
\end{equation*} 
We refer to such sets \emph{plates}, and $s$ and $\bdrho$ as the \emph{centre} and \emph{principal radii} of the given plate, respectively.
\end{definition}

In later applications, we always have $C = O_\gamma(1)$, so implicit constants depending on this parameter should be viewed as unproblematic. Here we record a separation property of the plates, which is later exploited in \S\ref{sec: Lp improv III}.

\begin{lemma} \label{lem: plate sep}
    For $B$, $C \geq 1$, there exists $\varepsilon_{B,C} > 0$ with $\varepsilon_{B,C} \gtrsim_B C^{-2}$ such that for $0 < \varepsilon \leq \varepsilon_{B,C}$ the following holds. Suppose $3 \leq M \leq N \leq n$ and $s_1$, $s_2 \in [-1,1]$, $\bdrho = (\rho_1,\dots,\rho_n) \in (0, \infty)^n$ satisfy:
    \begin{enumerate}[\normalfont(i)]
        \item $\rho_1 \leq \rho_2 \leq \cdots \leq \rho_N = \rho_{N+1} = \cdots = \rho_n$;
        \item $\varepsilon^{-1}\rho_M^{-1} \rho_{M-1} \leq |s_2 - s_1| \leq \varepsilon$;
        \item For $M \leq i \leq N-1$, either $|s_2 - s_1| \geq \varepsilon^{-1} \rho_{i+1}^{-1} \rho_i$ or $|s_2 - s_1| \leq \varepsilon \rho_{i+1}^{-1} \rho_i$.
    \end{enumerate}
    If $\gamma \in \fkG(B)$, then
    \begin{equation*}
         d \big( \Theta_{M,N}(s_1,\bdrho,C), \Theta_{M,N}(s_2,\bdrho,C) \big) \gtrsim_B C^{-1} \rho_M |s_2 - s_1|.
    \end{equation*}
\end{lemma}

Here, given non-empty sets $E$, $F \subseteq \R^n$,
\[ d(E,F) \coloneq \inf\{ |x-y| : x \in E,\, y \in F \}. \]

\begin{proof}[Proof (of Lemma~\ref{lem: plate sep})]
       Fix $x_j \in \Theta_{M,N}(s_j, \bdrho, C)$ for $j = 1$ and $j = 2$. Our goal is to show that, provided $\varepsilon > 0$ is chosen sufficiently small, depending only on $n$, $A$ and $B$, there exists some $M \leq m \leq N$ such that
    \begin{equation}\label{eq: plate sep 1}
        |s_2 - s_1| \geq \varepsilon^{-1} \rho_m^{-1}\rho_{m-1} 
    \end{equation}
and
\begin{equation}\label{eq: plate sep 2}
        \big| \big(\gamma^{(m-1)}(s_2) - \gamma^{(m-1)}(s_1)\big) \cdot x_1 \big| \geq (4C)^{-1} \rho_m |s_2 - s_1|.
\end{equation}
Once this is established, we may argue as follows. By the Cauchy--Schwarz and triangle inequalities,
    \begin{align*}
        |\gamma^{(m-1)}(s_2)||x_2 - x_1| &\geq |\gamma^{(m-1)}(s_2)\cdot(x_2 - x_1)| \\
        &\geq \big|\big(\gamma^{(m-1)}(s_2) - \gamma^{(m-1)}(s_1)\big)\cdot x_1| - \sum_{j=1}^2|\gamma^{(m-1)}(s_j)\cdot x_j|. 
    \end{align*}
    We apply \eqref{eq: plate sep 2} to lower bound the first term on the right hand side, and use the upper bounds from Definition~\ref{dfn: plate} that arise because $x_j \in \Theta_{M,N}(s_j,\bdrho,C)$ to upper bound the second term on the right hand side. From this, we deduce that
    \begin{align*}
        |\gamma^{(m-1)}(s_2)||x_2 - x_1| &\geq (4C)^{-1}\rho_m |s_1 - s_2| - 2C \rho_{m-1}\\
        &\geq \big((4C)^{-1} - 2C\varepsilon\big) \rho_m |s_1 - s_2|,
    \end{align*}
    where the second inequality follows from using \eqref{eq: plate sep 1} to upper bound $2C \rho_{m-1}$. Provided $\varepsilon > 0$ is sufficiently small, it follows from \eqref{eq: G class 1} that
    \begin{equation*}
        |x_2 - x_1| \geq (8C)^{-1} B^{-1} \rho_m|s_1 - s_2| \geq (8C)^{-1} B^{-1} \rho_M|s_1 - s_2|.
    \end{equation*}
    Taking infima of this equation over all choices of $x_j \in \Theta_{M,N}(s_j, \bdrho, C)$ for $j = 1$ and $j = 2$ gives the required lower bound on $d(\Theta_{M,N}(s_1, \bdrho, C),\Theta_{M,N}(s_2, \bdrho, C))$.

  We now turn to the proof of the existence of such an $m$. Combining \eqref{eq: G class 4} with Definition~\ref{dfn: plate} and hypothesis~(i), we find that
    \begin{equation*}
        |x_1| \leq B^{O(1)} \max_{1 \leq j \leq n} |\gamma^{(j)}(s_1) \cdot x_1| \leq C B^{O(1)} \rho_N.
    \end{equation*}
    Recalling \eqref{eq: G class 1}, this implies that
    \begin{equation}\label{eq: plate sep 3}
        \sup_{s \in [-1,1]} |\gamma^{(j)}(s) \cdot x_1| \leq \| |\gamma^{(j)}| \|_{L^\infty([-1,1])}\, |x_1| \leq C B^{O(1)} \rho_N
    \end{equation}  
    for $1 \leq j \leq 2n$.

    The proof now divides into two cases.\medskip

\noindent \textit{Case 1: $|s_2 - s_1| \geq \varepsilon^{-1} \rho_N^{-1} \rho_{N-1}$.} In this case, we take $m = N$. Then \eqref{eq: plate sep 1} holds by the hypothesis of this case. Given $s \in [-1,1]$ lying between $s_1$ and $s_2$, by the mean value theorem,
\begin{equation*}
    |\gamma^{(N)}(s) \cdot x_1| \geq |\gamma^{(N)}(s_1) \cdot x_1| - \sup_{t \in [-1,1]}|\gamma^{(N+1)}(t) \cdot x_1| |s - s_1|. 
\end{equation*}
Apply the lower bounds on $|\gamma^{(N)}(s_1) \cdot x_1|$ implied by Definition~\ref{dfn: plate}, because $x_1 \in \Theta_{M,N}(s_1,\rho,C)$, to lower bound the first term on the right-hand side. Apply \eqref{eq: plate sep 3} and the bound $|s_1 - s_2| \leq \varepsilon$ from hypothesis~(ii) to upper bound the second term. This yields that
\begin{equation*}
    |\gamma^{(N)}(s) \cdot x_1| \geq \big((2C)^{-1} - C B^{O(1)} \varepsilon\big)\rho_N \geq (4C)^{-1} \rho_N, 
\end{equation*}
 where the last inequality holds provided $\varepsilon > 0$ is suitably small. Applying the mean value theorem, we thus conclude that
    \begin{equation*}
     \big| [\gamma^{(N-1)}(s_1) - \gamma^{(N-1)}(s_2)] \cdot x_1 \big| \geq (4C)^{-1} \rho_N |s_1 - s_2|,   
    \end{equation*}
    which is the required bound \eqref{eq: plate sep 2} for $m = N$. This concludes the proof of this case.\medskip

    \noindent \textit{Case 2: $|s_2 - s_1| < \varepsilon^{-1} \rho_N^{-1} \rho_{N-1}$.} By hypothesis~(ii), we must have $M < N$. Let $M \leq m \leq N-1$ be the least value satisfying 
    \begin{equation}\label{eq: plate sep 4}
       |s_2 - s_1| < \varepsilon^{-1} \rho_{i+1}^{-1} \rho_i \qquad \textrm{for all $m \leq i \leq N-1$,}  
    \end{equation}
    noting that such an $m$ exists by the hypothesis of this case. 
    By the minimality assumption, and hypothesis (ii) in the case $m = M$, we deduce that \eqref{eq: plate sep 1} holds for this value of $m$. It thus remains to show \eqref{eq: plate sep 2}. 

    By hypothesis (iii), we can upgrade \eqref{eq: plate sep 4} to $|s_2 - s_1| \leq \varepsilon \rho_{i+1}^{-1} \rho_i$ for all $m \leq i \leq N-1$. We may recursively apply this bound to deduce that
    \begin{equation*}
      \rho_j|s_2 - s_1|^{j-m} \leq \varepsilon^{j-m} \rho_m \quad \textrm{for all $m \leq j \leq N$.}
    \end{equation*}
    Since our hypotheses ensures $\rho_N = \cdots = \rho_n$ and $|s_2 - s_1| \leq \varepsilon$, we in fact have
    \begin{equation}\label{eq: plate sep 5}
      \rho_j|s_2 - s_1|^{j-m} \leq \varepsilon^{j-m} \rho_m
    \end{equation}
    for all $m \leq j \leq n$.

    By a Taylor expansion of $\gamma^{(m-1)}$, we conclude that
    \begin{multline}\label{eq: plate sep 6}
       \big|\big(\gamma^{(m-1)}(s_2) - \gamma^{(m-1)}(s_1)\big) \cdot x_1\big| \geq |\gamma^{(m)}(s_1) \cdot x_1||s_1 - s_2| \\ - \sum_{j= m + 1}^{n-1} |\gamma^{(j)}(s_1)\cdot x_1||s_2 - s_1|^{j-m +1} - \sup_{t \in [-1,1]}|\gamma^{(n)}(t)\cdot x_1||s_2 - s_1|^{n-m+1}.
    \end{multline}
    We use Definition~\ref{dfn: plate}, through the inclusion $x_1 \in \Theta_{M,N}(s_1,\bdrho,C)$, to lower bound the principal term arising from the Taylor polynomial and upper bound the higher order terms. We further use \eqref{eq: plate sep 3} to upper bound the remainder term. From this, we see the right-hand side of \eqref{eq: plate sep 6} is lower bounded by
    \begin{equation*}
        \Big( (2C)^{-1}\rho_m  - C\sum_{j = m+1}^{n-1} \rho_j |s_2 - s_1|^{j-m} - CB^{O(1)} \rho_n |s_2 - s_1|^{n-m} \Big)|s_2 - s_1|.
    \end{equation*}
    Using \eqref{eq: plate sep 5}, this can be further lower bounded by
    \begin{equation*}
        \big( (2C)^{-1} - nC\varepsilon - CB^{O(1)}\varepsilon \big)\rho_m|s_2 - s_1|  \geq (4C)^{-1}\rho_m|s_2 - s_1|,
    \end{equation*}
    where the last inequality holds provided $\varepsilon > 0$ is suitably small. Thus, 
    \begin{equation*}
        \big| \big(\gamma^{(m-1)}(s_2) - \gamma^{(m-1)}(s_1)\big) \cdot x_1 \big| \geq (4C)^{-1} \rho_m |s_2 - s_1|,
    \end{equation*}
which is the required bound \eqref{eq: plate sep 2}. This concludes the proof.
\end{proof}

Later, we shall apply Lemma~\ref{lem: plate sep} pairwise to a whole family of plates, with centres lying in some set $X \subseteq [-1,1]$. The following technical lemma shall allow us to essentially dispense with hypothesis (iii).

\begin{lemma} \label{lem: force sep}
    For $K \in \N$ and $0 < \delta \leq 1$, there exists some $A \in \N$ with $A \lesssim_K \delta^{-2K}$ such that the following holds. Let $\varepsilon_1, \dots, \varepsilon_K > 0$ and $X \subseteq [-1,1]$. There exists a decomposition $X = \bigcup_{\alpha = 1}^A X_{\alpha}$ satisfying the following property. For all $1 \leq \alpha \leq A$, if $s_1$, $s_2 \in X_\alpha$, then for all $1 \leq i \leq K$ either $|s_1 - s_2| \leq \delta \varepsilon_i$ or $|s_1 - s_2| \geq \delta^{-1} \varepsilon_i$.
\end{lemma}
\begin{proof}
    By applying the lemma recursively, it suffices to consider the case $K = 1$, with a single parameter $\varepsilon \coloneq \varepsilon_1 >0$. We shall also assume $2\delta^{-1}\varepsilon > 1$; if this is not the case, then the result follows from a simpler variant of the following argument. 
    
    Partition $[-1,1]$ as a unions of disjoint intervals $\{I^{\beta} : 1 \leq \beta \leq B \}$,  for some $B \in \N$, which satisfy 
    \begin{equation*}
        |I^\beta| = 2 \delta^{-1} \varepsilon \quad \textrm{for $1 \leq \beta \leq B-1$} \quad \textrm{and} \quad \inf I^{\beta_1} < \inf I^{\beta_2}, \quad 1 \leq \beta_1 < \beta_2 \leq B.
    \end{equation*}
    For each $1 \leq \beta \leq B$, further partition $I^{\beta}$ as a union of disjoint intervals $\{I^{\beta}_{\alpha} : 1 \leq \alpha \leq A_{\beta}\}$, for some $A_{\beta} \in \N$, which satisfy 
    \begin{equation*}
        |I^{\beta}_{\alpha}| = \delta \varepsilon \quad \textrm{for $1 \leq \beta \leq A_{\beta}-1$} \quad \textrm{and} \quad \inf I^{\beta}_{ \alpha_1} < \inf I^{\beta}_{\alpha_2}, \quad 1 \leq \alpha_1 < \alpha_2 \leq A_{\beta}.
    \end{equation*}
    Note that $A \coloneq \max\{A_{\beta} : 1 \leq \beta \leq B\} \lesssim \delta^{-2}$. Now define 
    \begin{equation*}
      X_{\alpha} \coloneq X \cap \bigcup_{\substack{1 \leq \beta \leq B \\ \alpha \leq A_{\beta}}}I^{\beta}_{\alpha} \qquad \textrm{for $1 \leq \alpha \leq A$.}  
    \end{equation*}
    It is a straightforward matter to verify these sets have the desired property. 
\end{proof}




\subsection{Dual plates} We also introduce a family of \emph{dual plates}. Let $\gamma \colon [-2,2] \to \R^n$ be a smooth, non-degenerate curve. Consider the dual frame $v_1, \dots, v_n \colon [-2,2] \to \R^n$ to the frame $\gamma^{(1)},\dots,\gamma^{(n)} \colon [-2,2] \to \R^n$; that is, for each $s \in [-2,2]$ and $1 \leq i \leq n$, the vector $v_i(s)$ is defined so that 
\begin{equation}\label{eq: dual frame dfn}
v_i(s) \cdot \gamma^{(j)}(s) = \delta_{ij} \qquad \textrm{for all $1 \leq j \leq n$.}
\end{equation}
Explicitly, $v_i(s) = \bdM_{\gamma}(s)^{-\top} e_i$ where $\bdM_{\gamma}(s)$ is the matrix in \eqref{eq: M matrix} and $e_i$ is the $i$th standard basis vector, for all $s \in [-2,2]$ and $1 \leq i \leq n$.

\begin{definition}\label{dfn: dual plate} Let $\gamma \colon [-2,2] \to \R^n$ be a smooth, non-degenerate curve. For $s \in [-2,2]$ and $\bdrho = (\rho_1, \dots, \rho_n) \in (0,\infty)^n$,  define the \emph{dual plate}
    \begin{equation*}
        \Theta^{\circ}(s,\bdrho) \coloneq \Big\{ \xi \in \widehat{\R}^n: |v_j(s) \cdot \xi| \leq \rho_j^{-1}\ \text{for}\ 1 \leq j \leq n \Big\}.
    \end{equation*}
\end{definition}

The set $\Theta^{\circ}(s,\bdrho)$ is essentially the polar dual of the plate $\Theta(s,\bdrho,C)$. Here we record an overlap property of the dual plates, which is later exploited in \S\ref{sec: Lp improv III}.

\begin{lemma} \label{lem: dual overlap} For $B,\lambda \geq 1$, there exists $\varepsilon_B > 0$ such that for all $0<\varepsilon \leq \varepsilon_B$ the following holds. Let $\bdrho = (\rho_1,\dots,\rho_n) \in (0, \infty)^n$ with $0 < \rho_1 \leq \rho_2 \leq \cdots \leq \rho_n$. Suppose $\xi \in \smash{\widehat{\R}^n}$ satisfies $|\xi| \geq \varepsilon^{-1} \lambda \rho_2^{-1}$. If $I \subseteq [-1,1]$ is an interval of length $|I| \leq \varepsilon$ and $\gamma \in \fkG(B)$, then the set
 \begin{equation*}
     \big\{ s \in I : \xi \in \lambda \Theta^{\circ}(s,\bdrho) \big\}
 \end{equation*}
is contained in an interval of length $O( B^{O(1)} \lambda \rho_2^{-1} |\xi|^{-1} )$.
\end{lemma}

Before proving Lemma~\ref{lem: dual overlap}, we make some basic observations concerning the dual basis introduced above. Let $\gamma \in \fkG(B)$ for some $B \geq 1$. By the formula $v_i(s) = \bdM_{\gamma}(s)^{-\top} e_i$ and \eqref{eq: G class 1}, we have 
\begin{equation} \label{eq: dual G 1}
    |\xi| = |\bdM_{\gamma}(s) \bdM_{\gamma}(s)^{-1} \xi| \leq B^{O(1)} \max_{1 \leq i \leq n} |v_i(s) \cdot \xi|
\end{equation}
for all $\xi \in \widehat{\R}^n$ and $s \in [-1,1]$. On the other hand, since $\gamma$ is a smooth non-degenerate curve, the functions $v_i \colon [-2,2] \to \R^n$ are smooth and, by \eqref{eq: G class 1}, satisfy
\begin{equation} \label{eq: dual G 2}
    \||v_i^{(j)}|\|_{L^{\infty}([-1,1])} \leq B^{O(1)} \quad\text{for $1 \leq i$, $j \leq n$,}
\end{equation}
where $v_i^{(j)}$ denotes the function formed by entry-wise differentiation of $v_i$.

\begin{proof}[Proof (of Lemma~\ref{lem: dual overlap})] Suppose $s_1,s_2 \in I$, and $\xi \in \lambda \Theta^{\circ}(s_1,\bdrho) \cap \lambda \Theta^{\circ}(s_2,\bdrho)$. Our goal is to show 
\begin{equation}\label{eq: dual overlap 1}
   |s_1 - s_2| \lesssim_B \lambda \rho_2^{-1} |\xi|^{-1},
\end{equation}
which implies the conclusion of the lemma.

By using the bounds in Definition~\ref{dfn: dual plate} which follow because $s_1 \in \lambda \Theta^\circ(s_2,\bdrho)$, and the monotonicity hypothesis on $\bdrho$, we have
\begin{equation}\label{eq: dual overlap 2}
  |v_j(s_1) \cdot \xi| \leq \lambda \rho_j^{-1} \leq \lambda \rho_2^{-1} \qquad \textrm{for $2 \leq j \leq n$.}
\end{equation}
By \eqref{eq: dual G 1} and the hypothesis $|\xi| \geq \varepsilon^{-1} \lambda \rho_2^{-1}$, this gives
\begin{equation*}
    |\xi| \leq B^{O(1)} \max\big\{ |v_1(s_1) \cdot \xi|, \lambda \rho_2^{-1}\big\} \leq B^{O(1)} \max\big\{ |v_1(s_1) \cdot \xi|, \varepsilon |\xi| \big\} .
\end{equation*}
Thus we conclude that
\begin{equation} \label{eq: dual overlap 3}
    |v_1(s_1) \cdot \xi| \geq B^{-O(1)} |\xi|,
\end{equation}
provided that $0 < \varepsilon \leq (1/2) B^{-O(1)}$.
   
    Differentiating the relation $v_i \cdot \gamma^{(j)} = \delta_{ij}$ from \eqref{eq: dual frame dfn}, we find that 
    \begin{equation}\label{eq: dual overlap 4}
    v_2'(t) = - v_1(t) - \big(v_2(t) \cdot \gamma^{(n+1)}(t)\big) v_n(t) \qquad \textrm{for all $t \in [-2,2]$.}   
    \end{equation}
    Note \eqref{eq: dual G 2} and \eqref{eq: G class 1} imply that
    \begin{equation} \label{eq: dual overlap 5}
        |v_2(t) \cdot \gamma^{(n+1)}(t)| \leq |v_2(t)| |\gamma^{(n+1)}(t)| \leq B^{O(1)} \qquad \textrm{for all $t \in [-1,1]$.}
    \end{equation}
    By \eqref{eq: dual overlap 4} and the triangle inequality,
    \begin{equation*}
        |v_2'(s_1) \cdot \xi|\geq |v_1(s_1) \cdot \xi| - |v_2(s_1) \cdot \gamma^{(n+1)}(s_1)||v_n(s_1) \cdot \xi|.
    \end{equation*}
    We use \eqref{eq: dual overlap 3} to lower bound the first term on the right-hand side. And we use \eqref{eq: dual overlap 5}, the $j = n$ case of \eqref{eq: dual overlap 2} and the hypothesis $|\xi| \geq \varepsilon^{-1} \lambda \rho_2^{-1}$ to upper bound the second term. Thus,
    \begin{align}
    \nonumber
        |v_2'(s_1) \cdot \xi| &\geq B^{-O(1)}|\xi| - B^{O(1)}\lambda \rho_2^{-1} \\
        \label{eq: dual overlap 6}
        &\geq \big(B^{-O(1)} - \varepsilon B^{O(1)}\big)|\xi| \geq (2B)^{-O(1)}|\xi|,
    \end{align}
    where the last inequality holds provided $\varepsilon$ is suitably small.

    On the other hand, by \eqref{eq: dual G 2}, we have
   \begin{equation}\label{eq: dual overlap 7}
       |v_2''(t) \cdot \xi| \leq |v_2''(t)| |\xi| \leq B^{O(1)} |\xi| \qquad \textrm{for all $t \in [-1,1]$.}
   \end{equation}
    Thus, by Taylor expansion of the function $v_2$, 
    \begin{align*}
        \big| \big(v_2(s_2) - v_2(s_1)\big) \cdot \xi \big| \geq \big((2B)^{-O(1)}  - B^{O(1)}  |s_1 - s_2|\big) |\xi| |s_1 - s_2|,
    \end{align*}
    where we have used \eqref{eq: dual overlap 6} to lower bound the linear term and \eqref{eq: dual overlap 7} to bound the remainder. Since $|s_1 - s_2| \leq |I| \leq \varepsilon$, it follows that
    \begin{equation}\label{eq: dual overlap 8}
       \sum_{j=1}^2 |v_2(s_j) \cdot \xi| \geq \big| \big(v_2(s_2) - v_2(s_1)\big) \cdot \xi \big| \gtrsim B^{-O(1)} |\xi| |s_1 - s_2|,
    \end{equation}
    provided $\varepsilon > 0$ is suitably small. Finally, since $\xi \in \lambda \Theta^{\circ}(s_1,\bdrho) \cap \lambda \Theta^{\circ}(s_2,\bdrho)$, we may apply Definition~\ref{dfn: dual plate} to upper bound the left-hand side of \eqref{eq: dual overlap 8} by $2 \lambda \rho_2^{-1}$, and thereby deduce \eqref{eq: dual overlap 1}, as required. 
\end{proof}




\section{Proof of \texorpdfstring{$L^p$-improving}{Lp improving} II: Kernel decomposition}\label{sec: Lp improv II}

Given $a \in \cS(A)$ for $A \geq 1$, to study the kernel $\cK_{\gamma}[a]$, we perform a rather intricate decomposition of $a$.




\subsection{Initial decomposition} We first perform a routine dyadic decomposition, which partially motivates the following symbol class.

\begin{definition} \label{dfn: symb depth N}
 For $\gamma \colon [-2,2] \to \R^n$ a smooth, non-degenerate curve, $A \geq 1$, $1 \leq N \leq n$ and $\lambda \geq 1$, let $\cS_{\gamma}(A; N; \lambda)$ denote the class of all $a \in \cS(A)$ satisfying:\vspace{2pt}
    \begin{enumerate}[\normalfont(i)]
        \item $\supp{} a \subseteq [    -1, 1] \times \{ x \in \R^n : (2A)^{-1}\lambda \leq |x| \leq 2A\lambda \}$;\vspace{2pt}
    
        \item for all $(s; x) \in \supp{} a$, $\displaystyle \max_{1 \leq j \leq N} | \gamma^{(j)}(s) \cdot x | \geq A^{-1} |x|$.
 \end{enumerate}
\end{definition}

Primarily, we are interested in the case $N = n$, in which case property (ii) follows automatically for some $A = O_{\gamma}(1)$ by the non-degeneracy hypothesis. However, it is useful to work with general $N$, in anticipation of inductive arguments used below.

\begin{lemma}\label{lem: initial decomp} Let $\gamma \colon [-2,2] \to \R^n$ be a smooth, non-degenerate curve and $a \in \cS(A)$ for $A \geq 1$. Then $a = \sum_{k=0}^{\infty} a^k$ where: 
\begin{enumerate}[\normalfont(a)]
    \item $a^0 \in C^{\infty}_tL^{\infty}_x(\R \times \R^n)$ satisfies
    \begin{equation*}
        \|\cK_{\gamma}[a^0]\|_{L^{\infty}(\R^n)} \lesssim_A 1 \qquad \textrm{and} \qquad \supp{} a^0 \subseteq [-1,1] \times \{x \in \R^n : |x| \leq 2\};
    \end{equation*}
    \item $a^k \in \cS_{\gamma}(O(A); n; 2^k)$ for all $k \in \N$. 
\end{enumerate}
\end{lemma}

\begin{proof}
    This follows by a standard dyadic decomposition.
\end{proof}




\subsection{Preliminaries for the main decomposition} Given a dyadically localised symbol $a^k \in \cS_{\gamma}(A; N; 2^k)$, we perform a further decomposition to isolate decay properties of the oscillatory kernel $\cK_{\gamma}[a^k]$. The description of this decomposition is somewhat involved, and requires a number of preliminary definitions. 

We first introduce sets used to index the pieces of the decomposition. 

\begin{definition}\label{dfn: indices 1} For $1 \leq N \leq n$ and $k \in \N$, define $L_N(N;k) \coloneq \ceil{k/N}$. For $1 \leq M \leq N \leq n$ and $\bdell = (\ell_M, \dots, \ell_N) \in \N_0^{N - M + 1}$, further define
\begin{equation*}
  L_m\bigl(N;k;\ell_{m+1},\dots,\ell_N\bigr)
\coloneq
\Big\lceil
\frac{1}{m}
\Big(
k - \sum_{j=m+1}^N j \ell_j
\Big)
\Big\rceil \qquad \textrm{for $M \leq m \leq N-1$.}
\end{equation*}
\end{definition}

\begin{definition}[Indexing set]\label{dfn: indices 2} For $1 \leq M < N \leq n$ and $k \in \N$, define:
\begin{enumerate}[\normalfont(a)]
    \item $\cL(N, N; k) \coloneq \{L_N(N;k)\}$.\vspace{2pt}
    \item Let $\cL(M, N; k)$ denote the set of all $\bdell = (\ell_M, \dots, \ell_N) \in \N_0^{N - M + 1}$ which satisfy:\vspace{2pt}
\begin{enumerate}[\normalfont(i)]
    \item $0 \leq \ell_N < L_N(N;k)$;
    \item $0 \leq \ell_m <
L_m\bigl(N;k;\ell_{m+1},\dots,\ell_N\bigr)$ for all $M+1 \leq m \leq N$;
\item $\ell_M =
L_M\bigl(N;k;\ell_{M+1},\dots,\ell_N\bigr)$.
\end{enumerate}
\item Furthermore, define
\begin{equation}\label{eq: indicies}
    \cL(N; k) \coloneq \bigcup_{M=1}^N \cL(M, N; k) \qquad \textrm{and} \qquad  \cL(k) \coloneq \cL(n; k)
\end{equation}
\end{enumerate}

\end{definition}

Given $a^k \in \cS_{\gamma}(A; N; 2^k)$, we shall write $a^k = \sum_{\bdell \in \cL(N; k)} a^{k, \bdell}$ where each $a^{k, \bdell}$ satisfies certain kernel estimates and support properties. The support properties are described in terms of the plates introduced in Definition~\ref{dfn: plate}. For this, we need to specify the relevant centres and principal radii.

\begin{remark}
    We note that
\[ \cL(M,N;k) = \left\{ \bdell \in \N_0^{N-M+1} : k \leq \sum_{m = M}^N m \ell_m < k + M\ \text{and}\ \sum_{m = M+1}^N m \ell_m < k \right\}. \]
%
%
%
%
%
%
%
%
%
%
%
%
%
%
The recursive description of $\cL(M,N;k)$ in Definition \ref{dfn: indices 2} arises naturally from the recursive construction by which the functions $a^{k,\bdell}$ are defined.
\end{remark}

\begin{definition}[Plate families]\label{dfn: plate family} For $1 \leq M \leq N \leq n$, $C \geq 1$, and $\bdell \in \N_0^{N-M+1}$ with $\bdell = (\ell_M, \dots, \ell_N)$, define the following:
\begin{enumerate}[\normalfont(a)]
    \item Given $\bdnu = (\nu_M, \dots, \nu_N) \in \Z^{N-M+1}$, let 
    \begin{equation}\label{eq: s nu ell centres}
        s_{\bdnu}^{\bdell} \coloneq  \sum_{j = M}^N 2^{-\sum_{i = j}^N \ell_i} \nu_j 
    \end{equation}
    and define
    \begin{equation*}
    \fkN(\bdell) \coloneq  \left\{ (\nu_M, \dots, \nu_N) \in \Z^{N-M+1} : |\nu_j| \leq 2^{\ell_j}, \,M \leq j \leq N\,; \left|s_{\bdnu}^{\bdell}\right| \leq 2\right\}.
\end{equation*}
    \item Let $\bdrho(\bdell)  = (\rho_1(\bdell), \dots, \rho_n(\bdell)) \in (0, 1]^n$ where
        \begin{equation}\label{eq: ej exp}
 \rho_j(\bdell) \coloneq 2^{-D_j(\bdell)}\ \ \textrm{for}\ \ D_j(\bdell) \coloneq \sum_{i = \max\{j+1, M\}}^N (i - j) \ell_i,\ \ 1 \leq j \leq N-1,
    \end{equation}
    and $D_j(\bdell) \coloneq 0$ for $N \leq j \leq n$.
    \item Given $\gamma \colon [-2,2] \to \R^n$ a smooth, non-degenerate curve, let 
\begin{equation}\label{eq: theta nu ell}
\Theta_{\bdnu}^{\bdell} \coloneq \Theta_{M,N}(s_{\bdnu}^{\bdell}, \bdrho(\bdell), C)
\end{equation}
for all $\bdnu = (\nu_M, \dots, \nu_N) \in \fkN(\bdell)$.
\item Finally, for $k \in \N$ let
\begin{equation*}
    \Lambda_{\gamma}(\bdell; C) \coloneq \bigcup_{\bdnu \in \fkN(\bdell)} \Theta_{\bdnu}^{\bdell} \qquad \textrm{and} \qquad \Lambda_{\gamma}(k, \bdell; C) \coloneq 2^k \Lambda_{\gamma}(\bdell; C).
\end{equation*}
\end{enumerate}
\end{definition}
Although the definitions of $\Theta_{\bdnu}^{\bdell}$ depends on $C$, since this is a technical parameter we suppress it in the notation. 




\subsection{The main decomposition} With the above definitions, the main decomposition reads as follows.

\begin{proposition}[Kernel decomposition]\label{prop: fine decomp} For $A$, $B \geq 1$ and $1 \leq N \leq n$, there exist $C  \geq 1$ and $K \in \N$ such that the following holds.  Let $\gamma \in \fkG(B; N)$ and $a^k \in \cS_{\gamma}(A; N; 2^k)$ for some $k \in \N$ with $k \geq K$. We may write
\begin{equation*}
    a^k = \sum_{\bdell \in \cL(N; k)} a^{k, \bdell}
\end{equation*}
for non-negative symbols $a^{k, \bdell} \in C^{\infty}_tL^{\infty}_x(\R \times \R^n)$ which, for all $\bdell \in \cL(N; k)$, satisfy:
\begin{enumerate}[\normalfont(i)]
    \item Kernel estimate: 
    \begin{equation}\label{eq: decomp K}
        \|\cK_{\gamma}[a^{k, \bdell}]\|_{L^{\infty}(\R^n)} \lesssim_{A, B} 2^{-\|\bdell\|};
    \end{equation}
 \item Support condition: 
 \begin{equation}\label{eq: decomp S}
   \supp{} a^{k, \bdell} \subseteq [-1,1] \times \Lambda_{\gamma}(k; \bdell; C).  
 \end{equation}
\end{enumerate}
\end{proposition}

Here, for $\bdell = (\ell_M, \dots, \ell_N) \in \N_0^{N-M+1}$, we write $\|\bdell\| \coloneq \ell_M + \cdots + \ell_N$. Recall that the class $\fkG(B; N)$ was defined in Definition~\ref{dfn: G class}. In practice, we are only interested in the case $N = n$; however, for the purpose of proving Proposition~\ref{prop: fine decomp}, it is useful to work with the more general formulation above. 

Proposition~\ref{prop: fine decomp} is a pointwise variant of a decomposition argument developed in \cite{PS2007, KLO2023, GovindanSheri2024} and related works. The proof essentially follows by unpacking the inductive argument used in \cite{KLO2023}. A key ingredient is a scaling symmetry of the class of smooth, non-degenerate curves. We describe this scaling property, and other prerequisites, in the next subsection, before presenting the proof of Proposition~\ref{prop: fine decomp} in \S\ref{subsec: decomp proof} below.




\subsection{Further properties of non-degenerate curves}\label{subsec: rescaling} We briefly describe some additional properties of non-degenerate curves, used in the proof of Proposition~\ref{prop: fine decomp}.\medskip

\noindent \textit{Rescaling.} Let $\gamma \colon [-2,2] \to \R^n$ be a smooth, non-degenerate curve. Given $2 \leq N \leq n$ and $s \in [-2,2]$, let
\begin{equation*}
   V_N(s) \coloneq \mathrm{span}\big\{\gamma^{(1)}(s),\dots, \gamma^{(N)}(s)\big\}. 
\end{equation*}
The non-degeneracy hypothesis ensures that $\dim V_N(s) = N$. For $\delta \in (0,1)$, we define an anisotropic dilation $\Dil_{\gamma}(N; s, \delta) \colon \R^n \to \R^n$ by specifying that
\begin{equation}\label{eq: scaling map 1}
\Dil_{\gamma}(N; s, \delta) \colon \gamma^{(j)}(s) \mapsto \delta^j \gamma^{(j)}(s)  \qquad \textrm{for $1 \leq j \leq N$}
\end{equation}
and
\begin{equation*}
\Dil_{\gamma}(N; s, \delta) \colon v \mapsto \delta^N v \qquad \textrm{ for $v \in  V_N(s)^{\perp}$}.
\end{equation*} 
It follows immediately from the definition that
\begin{equation}\label{eq: scaling map 2}
\delta^N|x|
\lesssim_B
|\Dil_{\gamma}(N; s, \delta)x|
\lesssim_B
\delta|x| \qquad \textrm{for all $x \in \R^n$.}
\end{equation}
As in \cite{KLO2023}, we use this dilation to define a family of rescalings of $\gamma$.

\begin{definition} Assuming $[s-2\delta, s+2\delta] \subseteq [-2,2]$, define the rescaled curve $\gamma_N^{s,\delta} \colon [-2,2] \to \R^n$ by
\begin{equation}\label{eq: rescaled curve}
    \gamma_N^{s,\delta}(t) \coloneq \Dil_{\gamma}(N; s, \delta)^{-1} \big(\gamma(s + \delta t) - \gamma(s)\big) \qquad \textrm{for $t \in [-2,2]$}.
\end{equation}
\end{definition}

The following lemma verifies non-degeneracy assumptions for the rescaled curve. 

\begin{lemma}\label{lem: rescale} 
Let  $\gamma \in \fkG(B; N)$ for some $B \geq 1$ and $2 \leq N \leq n$. Then
\begin{equation*}
    \gamma_N^{s, \delta} \in \fkG(O_B(1); N-1) \qquad \textrm{for all $s \in [-1,1]$ and $0 < \delta \leq 1$.}
\end{equation*}
\end{lemma}
\begin{proof} See \cite[p. 10]{KLO2023} or \cite[Lemma 12]{GovindanSheri2024}. The result is stated in \cite[Lemma 12]{GovindanSheri2024} for small values of $\delta$, but it is clear it also holds trivially for the remaining range.
\end{proof}

\noindent \textit{Non-degeneracy hierarchy.} For $B \geq 1$ and $2 \leq N \leq n$, we have 
\begin{equation}\label{eq: heirarchy}
\fkG(B; N) \subseteq \fkG(B^2; N-1).    
\end{equation}
 Indeed, using \eqref{eq: G class 1}, this follows from the inequality
\begin{align*}
    B^{-2} &\leq B^{-1} |\gamma^{(1)}(s)\wedge\ldots\wedge \gamma^{(N)}(s)|\\
    &\leq B^{-1} |\gamma^{(N)}(s)| |\gamma^{(1)}(s)\wedge\ldots\wedge \gamma^{(N-1)}(s)|\\
    &\leq |\gamma^{(1)}(s)\wedge\ldots\wedge \gamma^{(N-1)}(s)|,
\end{align*}
and that $\||\gamma^{(j)}| \|_{L^\infty([-2, 2])} \leq B \leq B^2$ for $1 \leq j \leq 2n$.




\subsection{Proof of the main decomposition}\label{subsec: decomp proof} We now turn to the proof of Proposition~\ref{prop: fine decomp}, which follows by modifying ideas from \cite{KLO2023}.

\begin{proof}[Proof (of Proposition~\ref{prop: fine decomp})] We prove the proposition by induction on $N$.\medskip




\noindent \textbf{Base case:} Let $N = 1$ and  $A$, $B \geq 1$ and fix constants $C_1 \geq 1$ and $K_1 \in \N$, chosen sufficiently large depending on $n$, $A$ and $B$ so as to satisfy the forthcoming requirements of the proof. Suppose $\gamma \in \fkG(B; 1)$ and $a^k \in \cS_{\gamma}(A; N; 2^k)$ for some $k \in \N$ with $k \geq K_1$. From the definitions,
\begin{equation*}
    \cL(1; k) = \cL(1, 1;k) = \{L_1(1;k)\} = \{k\}.
\end{equation*}
Thus, in this case, we have a trivial decomposition given by $a^{k, k} \coloneq a^k$. It remains to check \eqref{eq: decomp K} and \eqref{eq: decomp S} hold.\smallskip

\noindent \textit{Kernel estimate.} Since $a^{k,k}\in \cS_{\gamma}(A; N; 2^k)$, by applying property (ii) and then property (i) of Definition~\ref{dfn: symb depth N}, we see that 
\begin{equation}\label{eq: base case 1}
    |\gamma'(s) \cdot x| \geq A^{-1}|x| \geq A^{-2} 2^{k-1} \qquad \textrm{for all $(s; x) \in \supp{} a^{k,k}$.}
\end{equation}
Thus, by applying integration-by-parts in $s$ once to the integral defining the kernel,
\begin{equation*}
    |\cK_{\gamma}[a^{k,k}](x)| \lesssim_{A, B} 2^{-k},
\end{equation*}
where we have used the derivative bounds from Definition~\ref{dfn: intial symb} and \eqref{eq: G class 1}. This is the desired kernel estimate~\eqref{eq: decomp K}.\smallskip

\noindent \textit{Support condition.} Let $(s; x) \in \supp{} a^{k,k}$, so that $s \in [-1,1]$ and there exists some $\nu \in \fkN(k)$ such that $|s_{\nu}^k - s| < 2^{-k} \leq 2^{-K_1}$. By property (i) of Definition~\ref{dfn: symb depth N} and \eqref{eq: G class 1}, provided $C_1 \geq 1$ is chosen sufficiently large, depending only on $A$ and $B$, we have
\begin{equation*}
   |\gamma'(s_{\nu}^k)\cdot x| \leq |\gamma'(s_{\nu}^k)||x| \leq 2A B 2^k \leq C_1 2^k.
\end{equation*}
On the other hand, to lower bound $|\gamma'(s_{\nu}^k)\cdot x|$, we employ the mean-value theorem, using \eqref{eq: base case 1} to lower bound $|\gamma'(s) \cdot x|$, using \eqref{eq: G class 1} to upper bound $|\gamma''(t) \cdot x|$, and using property (i) of Definition~\ref{dfn: symb depth N} to lower bound $|x|$. We conclude that
\begin{align*}
    |\gamma'(s_{\nu}^k)\cdot x| &\geq  |\gamma'(s)\cdot x| - \sup_{t \in [-2,2]} |\gamma''(t) \cdot x||s_{\nu}^k - s| \\
    &\geq \big(A^{-1} - 2^{-K_1}B\big)|x| \geq (2A)^{-1}(2A)^{-1}2^k \geq C_1^{-1}2^k,
\end{align*}
where the final inequalities in this bound hold provided $K_1 \in \N$ and $C_1 \geq 1$ are chosen sufficiently large. Thus, using the notation of \eqref{eq: theta nu ell}, we have $x \in \Theta_{\nu}^k \subseteq \Lambda_{\gamma}(k, k; C_1)$. This establishes the support condition \eqref{eq: decomp S}. 

\medskip




\noindent \textbf{Inductive step:} Let $2 \leq N \leq n$ and assume the result holds in the case $N-1$. Let $A$, $B \geq 1$ and fix constants $A' \geq A$ and $B' \geq B$, chosen sufficiently large depending on $n$, $A$ and $B$ so as to satisfy the forthcoming requirements of the proof. Let $C_{N-1} \geq 1$ and $K_{N-1} \in \N$ satisfy the conclusion of the proposition with $C = C_{N-1}$ and $K = K_{N-1}$ whenever $\tilde{\gamma} \in \fkG(B';N-1)$ and $\tilde{a}^k \in \cS_{\tilde{\gamma}}(A'; N-1; 2^k)$ for some $k \in \N$ with $k \geq K_{N-1}$. 

Let $K_{N-1} \in \N$ and suppose $\gamma \in \fkG(B; N)$ and $a^k \in \cS_{\gamma}(A; N; 2^k)$ for some $k \in \N$ with $k \geq K_N \geq 1$. Our goal is to show that, provided $C_N \geq 1$ and $K_N \in \N$ are sufficiently large depending only on $n$, $A$ and $B$, the conclusion of the proposition holds for $C = C_N$ and this choice of $\gamma$ and $a^k$.\medskip




\noindent \textit{Step 1: Isolating the singularity.} Fix $\eta \in C^{\infty}_c(\R)$ non-negative, radially decreasing and such that $\eta(r) = 1$ if $|r| \leq 1$ and $\eta(r) = 0$ if $|r| \geq 2$. Let $\varepsilon = \varepsilon_{A,B} > 0$ be chosen sufficiently small, depending only on $n$, $A$ and $B$, so as to satisfy the forthcoming requirements of the proof. Decompose
\begin{equation}\label{eq: decomp 1 1}
    a^k = a_{\leq N-1}^k + a_N^k \qquad \textrm{where} \qquad a_{\leq N-1}^k \coloneq a^k(1-H_N^k)
\end{equation}
for $H_N^k \colon \R \times \R^n \rightarrow \R$ given by
\begin{equation*}
H_N^k(s; x) \coloneq \prod_{j = 1}^{N-1} \eta\Big(\frac{\gamma^{(j)}(s) \cdot x}{\varepsilon 2^k}\Big).
\end{equation*} 
If $(s; x) \in \supp{}(1 - H_N^k)$, then there must exist some $1 \leq m \leq N-1$ such that $|\gamma^{(m)}(s) \cdot x| \geq \varepsilon 2^k$. Furthermore, 
\begin{equation*}
    |\partial_s^j H_N^k(s; x)| \lesssim_{A,B} 1 \qquad \textrm{ for $(s; x) \in \supp{} a^k$ and $0 \leq j \leq n$.}
\end{equation*}
Thus, provided $A' \geq 1$ is chosen sufficiently large, $a_{\leq N-1}^k \in \cS_{\gamma}(A'; N-1; 2^k)$ and, by \eqref{eq: heirarchy}, provided $B' \geq B^2$, we have $\gamma \in \fkG(B'; N-1)$. Furthermore, provided we choose $K_N \geq K_{N-1}$, we may apply the induction hypothesis to deduce that
\begin{equation}\label{eq: decomp 1 2}
    a_{\leq N-1}^k = \sum_{\bdell' \in \cL(N-1; k)} a^{k, \bdell'}
\end{equation}
for non-negative symbols $a^{k, \bdell'} \in C^{\infty}_t L^{\infty}_x(\R \times \R^n)$ satisfying
\begin{equation*}
  \|\cK_{\gamma}[a^{k, \bdell'}]\|_{L^{\infty}(\R^n)} \lesssim_{A, B} 2^{-\|\bdell'\|} \qquad \textrm{and} \qquad 
    \supp{} a^{k, \bdell'} \subseteq \Lambda_{\gamma}(k, \bdell'; C_{N-1})
\end{equation*}
for all $\bdell' \in \cL(N-1; k)$. 

For $1 \leq M \leq N-1$, Definition~\ref{dfn: indices 1} and Definition~\ref{dfn: indices 2} imply that
\begin{equation*}
    \big\{\bdell \in \cL(M, N; k) : \ell_N = 0 \big\} = \big\{ (\bdell', 0) \in \N_0^{N-M+1} : \bdell' \in \cL(M, N-1; k) \big\}.
\end{equation*}
Given $\bdell = (\bdell', 0) \in \N_0^{N- M +1}$ with $\bdell' \in \cL(N-1; k)$, we define $a^{k, \bdell} \coloneq a^{k, \bdell'}$ so that the decomposition \eqref{eq: decomp 1 2} may be rewritten as
\begin{equation}\label{eq: decomp 1 3}
    a_{\leq N-1}^k(s; x) = \sum_{M=1}^{N-1}\sum_{\substack{\bdell \in \cL(M, N; k) \\ \ell_N = 0}} a^{k, \bdell}(s; x).
\end{equation}
Given $\bdell = (\bdell', 0) \in \cL(M, N; k)$, note that $\|\bdell'\| = \|\bdell\|$ and, recalling \eqref{eq: ej exp}, that $D_j(\bdell') = D_j(\bdell)$ for $1 \leq j \leq n$. From this, we see that the above $a^{k, \bdell}$ automatically satisfy the desired kernel estimate \eqref{eq: decomp K}. On the other hand, if $\bdnu' \in \fkN(\bdell')$, provided we choose $C_N \geq C_{N-1}$, we have 
\begin{equation*}
    s_{\bdnu'}^{\bdell'} = s_{\bdnu}^{\bdell} \quad \textrm{and} \quad \Theta_{\bdnu'}^{\bdell'} = \Theta_{\bdnu}^{\bdell} \qquad \textrm{for $\bdnu = (\bdnu', 0)$ and $\bdell = (\bdell', 0)$.}
\end{equation*}
Thus, the above $a^{k, \bdell}$ also automatically satisfy the support condition \eqref{eq: decomp S}.\medskip




\noindent \textit{Step 2: Decomposing with respect to the singularity.} We now turn to the remaining part $a_N^k$ of the symbol. For $(s,x) \in \supp{} a_N^k$, by the definition \eqref{eq: decomp 1 1} and property (i) of Definition~\ref{dfn: symb depth N},
\begin{equation} \label{eq: decomp 2 1}
    |\gamma^{(j)}(s) \cdot x| \leq \varepsilon 2^{k+1} \leq 4A\varepsilon |x| \quad\text{for $1 \leq j \leq N-1$}.
\end{equation}
Recall from property (ii) of Definition~\ref{dfn: symb depth N} that 
\begin{equation*}
    \max_{1 \leq j \leq N} |\gamma^{(j)}(s) \cdot x| \geq A^{-1}|x|.
\end{equation*}
Thus, provided $\varepsilon > 0$ is chosen sufficiently small, depending only on $A$, we have
\begin{equation} \label{eq: decomp 2 2}
    A^{-1} |x| \leq |\gamma^{(N)}(s) \cdot x| \leq B |x|,
\end{equation}
where the upper bound holds by \eqref{eq: G class 1}.

Cover $[-2,2] \times \R^n$ by $O_{A,B}(1)$ sets of the form $I \times \Gamma$, where $I$ is an interval of length $(10A^2) \varepsilon$ and $\Gamma$ is a cone of aperture $\varepsilon A^{-1}$. It follows from \eqref{eq: decomp 2 2} and the derivative upper bounds of Definition \ref{dfn: G class} that if $(s;x) \in \supp{}(a^k_N) \cap (I \times \Gamma)$ for some $I$ and $\Gamma$, then $\gamma^{(N)}(s) \cdot x \neq 0$ for all $(s,x) \in 2I \times \Gamma$. By applying a partition of unity to the function $a^k$, possibly increasing the constant $A$ in the process, though only in a way depending on $A$ and $B$, we will assume without loss of generality that $a^k_N$ is supported on a single set $I \times \Gamma$.


Let $\varepsilon_0 \coloneq (4A)^2 \varepsilon$, so that, provided $\varepsilon > 0$ is chosen sufficiently small, $s \pm \varepsilon_0 \in [-2, 2]$. By Taylor's theorem, there exists some $s_{\star} \in [-2, 2]$ such that 
\begin{equation} \label{eq:e1o2ej1oie1j2elabel}
  \gamma^{(N-1)}(s \pm \varepsilon_0) \cdot x = \pm (\gamma^{(N)}(s)\cdot x)\varepsilon_0 +   \gamma^{(N-1)}(s)\cdot x + (\gamma^{(N+1)}(s_{\star})\cdot x) \frac{\varepsilon_0^2}{2}.  
\end{equation}
By \eqref{eq: decomp 2 1} and \eqref{eq: G class 1}, we have
\begin{equation*}
    \Big|\gamma^{(N-1)}(s)\cdot x + (\gamma^{(N+1)}(s_{\star})\cdot x) \frac{\varepsilon_0^2}{2}\Big| \leq \big((4A)^{-1} + 8 A^2 B \varepsilon\big) |x|\varepsilon_0 \leq (2A)^{-1}|x|\varepsilon_0,
\end{equation*}
where the last inequality holds provided $\varepsilon > 0$ is again chosen sufficiently small, depending only on $A$ and $B$. Comparing this with \eqref{eq: decomp 2 2}, we find
\[ \Big|\gamma^{(N-1)}(s)\cdot x + (\gamma^{(N+1)}(s_{\star})\cdot x) \frac{\varepsilon_0^2}{2}\Big| \leq \frac{|\gamma^{(N)}(s) \cdot x| \varepsilon_0}{2}. \]
But then we see from \eqref{eq:e1o2ej1oie1j2elabel} that the two values $\gamma^{(N-1)}(s \pm \varepsilon_0) \cdot x$ have opposite signs. Since the values $s \pm \varepsilon_0$ both lie in $2I$, by the intermediate value theorem, there exists $\theta(x) \in 2I$ with 
\begin{equation}\label{eq: decomp 2 3}
  \gamma^{(N-1)}\circ \theta (x) \cdot x = 0 \quad\text{and}\quad |\theta(x) - s| \lesssim_A \varepsilon.
\end{equation}
There are no other values $\theta \in 2I$ with $\gamma^{(N-1)}(\theta) \cdot x = 0$, because $s \mapsto \gamma^{(N-1)}(s) \cdot x$ is monotonic on $2I$. Because of this uniqueness, the implicit function theorem then guarantees that $\theta$ is a smooth function on $\supp{x} a^k_N$. Furthermore, by Taylor's theorem, \eqref{eq: G class 1} and \eqref{eq: decomp 2 1}, we have
\begin{equation}\label{eq: decomp 2 4b}
    |\gamma^{(j)}\circ \theta (x) \cdot x| = O_{A,B}(\varepsilon 2^k) \quad \textrm{for all $(s;x) \in \supp{} a_N^k$ and $1 \leq j \leq N-1$.} 
\end{equation}
%
%
%

By \eqref{eq: decomp 2 3} and \eqref{eq: decomp 2 4b}, there exists some $\varepsilon_1 > 0$ with $\varepsilon_1 = O_{A,B}(\varepsilon)$ such that 
\begin{equation}\label{eq: decomp 2 5}
    G_N^k(s; x) \coloneq \sum_{j=1}^{N-2} \Big| \frac{\gamma^{(j)}\circ \theta(x)\cdot x}{\varepsilon_1 2^k} \Big|^{2/(N-j)} + \Big|\frac{s-\theta(x)}{\varepsilon_1}\Big|^2  
\end{equation}
satisfies $|G_N^k(s; x)| \leq 1/4$ for all $(s; x) \in \supp{}a_N^k$.\footnote{The relevance of the parameter $\varepsilon_1 > 0$ will only become apparent in the proof of Lemma~\ref{lem: verify rescaling} below.} For $\eta$ as defined in Step 1, let $\beta \in C_{c}^{\infty}(\R)$ be given by $\beta(r) \coloneq \eta(r) - \eta(4r)$ for $r \in \R$. For all $L \in \N$, we have
\begin{equation*}
 \supp{} \beta  \subseteq \{ r \in \R: 1/4 \leq |r| \leq 2 \} \quad \textrm{and} \quad   \eta(2^{2L} r) + \sum_{\ell =1}^{L-1} \beta(2^{2\ell}r) = 1 \quad \textrm{for $|r| \leq 1/4$}.
\end{equation*} 
Thus, we may write 
\begin{equation*}
    a^N_k = \sum_{\ell = 1}^{\ceil{k/N}} a^N_{k,\ell}
\end{equation*}
where
\begin{equation}\label{eq: decomp 2 6}
  a_N^{k,\ell}(s; x) \coloneq a_N^k(s; x) \cdot
  \begin{cases}
    \eta\big(2^{2\ceil{k/N}} G_N^k(s; x)\big) & \text{if } \ell = \ceil{k/N}, \\
    \beta\big(2^{2\ell} G_N^k(s; x)\big) & \text{if } 1 \leq \ell < \ceil{k/N}.
  \end{cases}
\end{equation}

We further decompose the symbol. Let $\zeta \in C_c^{\infty}(\R)$ be such that $\supp{} \zeta  \subseteq [-1,1]$ and $\sum_{\nu \in \Z}  \zeta(\,\cdot\, - \nu) = 1$. For $1 \leq \ell \leq \ceil{k/N}$ and $\nu \in \Z$, define
\begin{equation}\label{eq: decomp 2 7}
 a^{k, \ell}_{N,\nu}(s; x) \coloneq a^{k, \ell}_N(s; x)\zeta(2^{\ell}s - \nu).
\end{equation}
It is not difficult to check that $a_{N, \nu}^{k, \ell} \in C^{\infty}_t L^{\infty}_x(\R \times \R^n)$ and
\begin{equation}\label{eq: decomp 2 8}
    |\partial_s^j a_{N, \nu}^{k, \ell}(s; x)| \lesssim_{A, B} 2^{\ell j} \qquad \textrm{for all $0 \leq j \leq n$.}
\end{equation}
We then have the full decomposition 
\begin{equation}\label{eq: decomp 2 9}
    a_N^k = \sum_{\ell = 1}^{\ceil{k/N}} a_N^{k, \ell} = \sum_{\ell = 1}^{\ceil{k/N}} \sum_{\nu \in \Z} a_{N, \nu}^{k, \ell}.
\end{equation}

We make some simple observations concerning the symbols $a_{N, \nu}^{k, \ell}$ appearing in~\eqref{eq: decomp 2 9}. For any fixed $1 \leq \ell \leq \ceil{k/N}$ and $x \in \R^n$, if $a_{N, \nu}^{k, \ell}(x; s) \neq 0$ for some $s \in \R$ and $\nu \in \Z$, then we must have $|s -\theta(x)| \lesssim 2^{-\ell}$ and $|s - s_{\nu}^{\ell}| \lesssim 2^{-\ell}$ for $s_{\nu}^{\ell} \coloneq 2^{-\ell}\nu$.  Since the $s_{\nu}^{\ell}$ are $2^{-\ell}$-separated, it follows that there are only $O(1)$ choices of $\nu \in \Z$ for which $s \mapsto a_{N, \nu}^{k, \ell}(x; s)$ is not identically zero.\medskip




\noindent \textit{Step 3: On the singular cone.} We consider the case $\ell = \ceil{k/N}$ in the decomposition~\eqref{eq: decomp 2 9}. Recalling Definition~\ref{dfn: indices 2}, the set $\cL(N, N;k) = \{\,\ceil{k/N}\,\}$ is a singleton, and so we may write
\begin{equation}\label{eq: decomp 3 1}
a_N^{k, \ceil{k/N}} = \sum_{\bdell \in \cL(N, N;k)} a^{k,\bdell} \quad \textrm{where} \quad    a^{k, \ceil{k/N}} \coloneq a_N^{k, \ceil{k/N}}  = \sum_{\nu \in \Z} a_{N, \nu}^{k, \ceil{k/N}}.   
\end{equation}
It remains to show that the symbol defined above satisfies the required kernel estimate and support condition. In this case we argue directly, using the highly localised support of the symbol, rather than appealing to the induction hypothesis. \smallskip

\noindent \textit{Kernel estimate.} It follows from the size of the $s$-support of the symbol that
\begin{equation*}
    \big\|\cK_{\gamma}\big[a^{k, \ceil{k/N}}\big]\big\|_{L^{\infty}(\R^n)}\lesssim_A 2^{-\ceil{k/N}},
\end{equation*}
which is the desired kernel estimate in this case.\smallskip

\noindent \textit{Support condition.} For each $\nu \in \Z$, the definition of the symbol implies that 
\begin{equation*}
    |\gamma^{(i)}\circ \theta (x) \cdot x| \lesssim 2^{i k/N} \quad \textrm{for $1 \leq i \leq N-1$} \quad \textrm{and} \quad |s_{\nu}^{\ell} - \theta(x)| \lesssim 2^{-k/N} 
\end{equation*}
for $(s; x) \in \supp{} a_{N, \nu}^{k, \ceil{k/N}}$. Thus, for $x \in \mathrm{supp}_x\, a_{N, \nu}^{k, \ceil{k/N}}$, by Taylor's theorem,
\begin{align}
\nonumber
   |\gamma^{(j)}(s_{\nu}^{\ell}) \cdot x| &\leq \sum_{i=j}^{N-1} | \gamma^{(i)}\circ \theta (x) \cdot x|\frac{|s_{\nu}^{\ell} - \theta(x)|^{i-j}}{(i-j)!} + B|x||s_{\nu}^{\ell} - \theta(x)|^{N - j} \\
   \label{eq: decomp 3 2}
   &\lesssim_{A,B} 2^{jk/N} 
\end{align}
for $1 \leq j \leq N$. 

We can use \eqref{eq: decomp 3 2} to verify the support condition \eqref{eq: decomp S}. Indeed, again recall that $\cL(N, N; k) = \{\,\ceil{k/N}\,\}$. By Definition~\ref{dfn: plate family}, we have
\begin{equation*}
    D_j\big(\ceil{k/N}\big) \coloneq \sum_{i = \max\{j+1, N\}}^N (i-j)\ell_i =
        (N-j)\ceil{k/N} \quad \textrm{if $1 \leq j \leq N-1$}
\end{equation*}
and $D_j\big(\ceil{k/N}\big) \coloneq 0$ if $N \leq j \leq n$. Furthermore,
\begin{equation*}
    \fkN\big(\ceil{k/N}\big) = \big\{ \nu \in \Z : |\nu| \leq 2^{\ceil{k/N}} \big\} \quad \textrm{and} \quad s_{\nu}^{\ceil{k/N}} =  2^{-\ceil{k/N}} \nu 
\end{equation*}
for $\nu \in \fkN\big(\ceil{k/N}\big)$. Provided $C_N \geq 1$ is chosen sufficiently large, \eqref{eq: decomp 3 2} implies that for all $x \in \supp{x} a^{k, \ceil{k/N}}$ there exists some $\nu \in \fkN(\ceil{k/N})$ such that
\begin{equation*}
    |\gamma^{(j)}(s_{\nu}^{\ceil{k/N}}) \cdot x| \leq C_N 2^{k - D_j(\ceil{k/N})} \qquad \textrm{for $1 \leq j \leq n$.}
\end{equation*}

On the other hand, if $(s; x) \in \supp{} a_{N, \nu}^{k, \ceil{k/N}}$, then 
\begin{equation*}
  |s - s_{\nu}^{\ceil{k/N}}| \leq 2^{-\ceil{k/N}} \leq 2^{-K_N/N} \leq 1 
\end{equation*}
and therefore $s$, $s_{\nu}^{\ceil{k/N}} \in [-2,2]$. Thus, \eqref{eq: decomp 2 2}, property (i) of Definition~\ref{dfn: symb depth N} and ~\eqref{eq: G class 1} ensure that
\begin{align*}
    |\gamma^{(N)}(s_{\nu}^{\ceil{k/N}}) \cdot x| &\geq |\gamma^{(N)}(s) \cdot x| - \sup_{t \in [-2,2]}|\gamma^{(N+1)}(t) \cdot x||s-s_{\nu}^{\ceil{k/N}}| \\
    &\geq \big(A^{-1} - 2^{-K_N/N} B\big)|x| \geq (2A)^{-1}(2A)^{-1}2^k \geq C_N^{-1}2^k,
\end{align*}
where the final inequalities hold provided both $K_N \in \N$ and $C_N \geq 1$ are chosen sufficiently large, depending only on $A$ and $B$. Thus, $a^{k, \ceil{k/N}}$ satisfies \eqref{eq: decomp S}.\medskip




\noindent \textit{Step 4: Rescaling and applying the induction hypothesis.} Fix $1 \leq \ell < \ceil{k/N}$. Recalling the definitions from \S\ref{subsec: rescaling}, to lighten notation for what follows, set
\begin{equation}\label{eq: decomp 4 1}
    \Dil_{N, \nu}^{\ell} \coloneq \Dil_{\gamma}(N; s_{\nu}^{\ell}, 2^{-\ell})  \qquad \textrm{and} \qquad \gamma_{N, \nu}^{\ell} \coloneq \gamma_N^{s_{\nu}^{\ell}, 2^{-\ell}}.
\end{equation}
By Lemma~\ref{lem: rescale}, provided $B'$ is chosen sufficiently large, $\gamma_{N, \nu}^{\ell} \in \fkG(B'; N-1)$. Consider the rescaled symbols
\begin{equation}\label{eq: decomp 4 2}
    b_{N, \nu}^{k, \ell}(\tilde{s}; \tilde{x}) \coloneq a_{N, \nu}^{k, \ell}\big(s_{\nu}^{\ell} + 2^{-\ell} \tilde{s}; \big(\Dil_{N, \nu}^{\ell}\big)^{-\top}\tilde{x}\big).
\end{equation}
The following claim is the crucial step in the inductive argument.

\begin{claim} For $1 \leq \ell < \ceil{k/N}$ and $\nu \in \Z$, provided $A'$ is sufficiently large, we have
\begin{equation*}
    b_{N, \nu}^{k, \ell} \in \cS_{\gamma_{N, \nu}^{\ell}}(A'; N-1; 2^{k- N\ell}).
\end{equation*}
\end{claim}

We verify this claim in Lemma~\ref{lem: verify rescaling} below. Assuming it holds, provided we choose $K_N \geq K_{N-1}$, we may apply the induction hypothesis to write
\begin{equation}\label{eq: decomp 4 3}
    b_{N, \nu}^{k, \ell} = \sum_{\bdell' \in \cL(N-1; k - N\ell)}b_{N, \nu}^{k, \bdell', \ell}
\end{equation}
for non-negative symbols $b_{N, \nu}^{k, \bdell', \ell} \in C^{\infty}_tL^{\infty}_x(\R \times \R^n)$ satisfying
\begin{equation}\label{eq: decomp 4 4}
  \big\|\cK_{\gamma_{N, \nu}^{\ell}}[b_{N, \nu}^{k, \bdell', \ell}]\big\|_{L^{\infty}(\R^n)} \lesssim_{A, B} 2^{-\|\bdell'\|}
\end{equation}
and
\begin{equation}\label{eq: decomp 4 5}
    \supp{} b_{N, \nu}^{k, \bdell', \ell} \subseteq [-1,1] \times \Lambda_{\gamma_{N, \nu}^{\ell}}(k - N \ell; \bdell'; C_{N-1}).
\end{equation} 

For $1 \leq M \leq N-1$, note that
\begin{equation*}
  \cL(M, N; k) = \big\{ (\bdell', \ell_N) \in \N_0^{N-M} \times \N_0 : 0 \leq \ell_N < \ceil{k/N}, \, \bdell' \in \cL(M, N-1; k-N \ell_N) \big\}.  
\end{equation*}
Consequently, if we define
\begin{equation}\label{eq: decomp 4 6}
    a^{k, \bdell}(s; x) \coloneq \sum_{\nu \in \Z} b_{N, \nu}^{k, \bdell}\big(2^{\ell}(s- s_{\nu}^{\ell}); \big(\Dil_{N, \nu}^{\ell}\big)^{\top}x\big)
\end{equation}
for $\bdell \in \cL(M, N;k)$ with $1 \leq M \leq N-1$ and $\ell_N = \ell > 0$, then by combining \eqref{eq: decomp 2 9}, \eqref{eq: decomp 3 1}, \eqref{eq: decomp 4 2} and \eqref{eq: decomp 4 3}, we see that
\begin{equation}\label{eq: decomp 4 7}
    a_N^k = \sum_{M=1}^{N-1}\sum_{\substack{\bdell \in \cL(M, N; k) \\ \ell_N > 0}} a^{k,\bdell} + \sum_{\bdell \in \cL(N, N; k)} a^{k,\bdell}. 
\end{equation}
Furthermore, we claim that the above symbols satisfy the desired conditions.\smallskip

\noindent \textit{Kernel estimate.} Fix $\bdell \in \cL(M, N; k)$ with $1 \leq M \leq N-1$ and $\ell_N = \ell > 0$. To prove the kernel estimate, since $\Dil_{N, \nu}^{\ell} \colon \R^n \to \R^n$ is a bijection, it suffices to show
\begin{equation*}
    \sup_{\tilde{x} \in \R^n}\big|\cK_{\gamma}[a^{k, \bdell}]\big(\big(\Dil_{N, \nu}^{\ell}\big)^{-\top} \tilde{x}\big)\big| \lesssim_{A, B} 2^{-\|\bdell\|}.
\end{equation*}
For $\tilde{x} \in \R^n$, by applying the definition \eqref{eq: decomp 4 6} and a change of variable of integration,
\begin{equation}\label{eq: decomp 4 8}
  \cK_{\gamma}[a^{k, \bdell}](x) = 2^{-\ell} \sum_{\nu \in \Z} e^{2\pi i x \cdot  \gamma(s_{\nu}^{\ell})}  \cK_{\gamma_{N, \nu}^{\ell}}[b_{N, \nu}^{k, \bdell}](\tilde{x}) \quad \textrm{for $x = \big(\Dil_{N, \nu}^{\ell}\big)^{-\top}\tilde{x}$.}
\end{equation}
From the discussion below \eqref{eq: decomp 2 9}, for any fixed $x \in \R^n$, only $O(1)$ terms of the right-hand sum in \eqref{eq: decomp 4 8} are nonzero. Hence, writing $\bdell = (\bdell', \ell) \in \N_0^{N-M} \times \N_0$ and applying \eqref{eq: decomp 4 4}, we may bound
\begin{equation*}
  \big|\cK_{\gamma}[a^{k, \bdell}](x)\big| \lesssim 2^{-\ell} \max_{\nu \in \Z} \big| \cK_{\gamma_{N, \nu}^{\ell}}[b_{N, \nu}^{k, \bdell', \ell}](\tilde{x}) \big| \lesssim_{A,B} 2^{-\|\bdell'\| - \ell} = 2^{-\|\bdell\|}.  
\end{equation*}
 Thus, we have the kernel estimate \eqref{eq: decomp K}.\smallskip

\noindent \textit{Support condition.} Unpacking the definitions in \eqref{eq: decomp 4 5}, given $\tilde{x} \in \supp{\tilde{x}}  b_{N, \nu}^{k, \bdell}$, there exists some $\bdnu' \in \fkN(\bdell')$ such that
\begin{equation}\label{eq: decomp 4 9}
    \big|\big(\gamma_{N, \nu}^{\ell}\big)^{(j)}(s_{\bdnu'}^{\bdell'}) \cdot \tilde{x}\big| \leq C_{N-1} 2^{k - N\ell - D_j(\bdell')} \qquad \textrm{for $1 \leq j \leq n$.}
\end{equation}
and
\begin{equation}\label{eq: decomp 4 10}
    \big|\big(\gamma_{N, \nu}^{\ell}\big)^{(j)}(s_{\bdnu'}^{\bdell'}) \cdot \tilde{x}\big| \geq C_{N-1}^{-1} 2^{k - N\ell - D_j(\bdell')} \qquad \textrm{for $M \leq j \leq N-1$.}
\end{equation}

From the definitions \eqref{eq: decomp 4 1} and \eqref{eq: rescaled curve}, for $j \in \N$ we have
\begin{align}
\nonumber
    \big(\gamma_{N, \nu}^{\ell}\big)^{(j)}(s_{\bdnu'}^{\bdell'}) &= \frac{\ud^j}{\ud s^j} \Big[ \big(\Dil_{N, \nu}^{\ell}\big)^{-1} \big(\gamma(s_{\nu}^{\ell} + 2^{-\ell} s) - \gamma(s_{\nu}^{\ell})\big) \Big] \Big|_{s = s_{\bdnu'}^{\bdell'}} \\
    \label{eq: decomp 4 11}
    &= 2^{- j \ell}\big(\Dil_{N, \nu}^{\ell}\big)^{-1} \gamma^{(j)}(s_{\nu}^{\ell} + 2^{-\ell} s_{\bdnu'}^{\bdell'}).
\end{align}
On the other hand, by the definition \eqref{eq: s nu ell centres}, it follows that $s_{\nu}^{\ell} + 2^{-\ell} s_{\bdnu'}^{\bdell'} = s_{\bdnu}^{\bdell}$ for $\bdell = (\bdell', \ell) \in \cL(M, N; k)$ and $\bdnu = (\bdnu', \nu) \in \fkN(\bdell)$. Combining these observations, 
\begin{equation}\label{eq: decomp 4 12}
  \big|\big(\gamma_{N, \nu}^{\ell}\big)^{(j)}(s_{\bdnu'}^{\bdell'}) \cdot \tilde{x}\big| = 2^{-j\ell}|\gamma^{(j)}(s_{\bdnu}^{\bdell}) \cdot x| \qquad \textrm{for} \qquad x = \big(\Dil_{N, \nu}^{\ell}\big)^{-\top} \tilde{x}.
\end{equation}
Thus, if $x \in \supp{x} a^{k, \bdell}$ as defined in \eqref{eq: decomp 4 6}, then \eqref{eq: decomp 4 9}, \eqref{eq: decomp 4 10} and \eqref{eq: decomp 4 12} together imply that there exists some $\bdnu \in \fkN(\bdell)$ such that
\begin{equation}\label{eq: decomp 4 13}
|\gamma^{(j)}(s_{\bdnu}^{\bdell}) \cdot x| \leq C_{N-1} 2^{k - D_j(\bdell') - (N-j)\ell} \leq C_N 2^{k - D_j(\bdell)} \qquad \textrm{for $1 \leq j \leq N-1$}
\end{equation}
and
\begin{equation}\label{eq: decomp 4 14}
|\gamma^{(j)}(s_{\bdnu}^{\bdell}) \cdot x| \geq C_N^{-1} 2^{k - D_j(\bdell)} \qquad \textrm{for $M \leq j \leq N-1$,}
\end{equation}
where we have used the definition \eqref{eq: ej exp} and have assumed $C_N \geq C_{N-1}$. Furthermore, by \eqref{eq: G class 1} and property (i) of Definition~\ref{dfn: symb depth N}, provided $C_N \geq 1$ is chosen sufficiently large, depending only on $A$ and $B$, we have
\begin{equation}\label{eq: decomp 4 15}
    |\gamma^{(j)}(s_{\bdnu}^{\bdell}) \cdot x| \leq |\gamma^{(j)}(s_{\bdnu}^{\bdell})||x| \leq AB 2^{k+1} \leq C_N 2^k \qquad \textrm{for $N \leq j \leq n$.}
\end{equation}
Finally, it follows from \eqref{eq: decomp 2 2} that
\begin{equation}\label{eq: decomp 4 16}
   |\gamma^{(N)}(s_{\bdnu}^{\bdell}) \cdot x| \geq A^{-1}|x| \geq A^{-1} (2A)^{-1}2^k \geq C_N^{-1} 2^k, 
\end{equation}
again provided $C_N \geq 1$ is chosen sufficiently large. Combining \eqref{eq: decomp 4 13}, \eqref{eq: decomp 4 14}, \eqref{eq: decomp 4 15} and \eqref{eq: decomp 4 16}, we see that $x \in \Theta_{\bdnu}^{\bdell}$. Consequently, $ \supp{x} a^{k, \bdell} \subseteq \Lambda_{\gamma}(k, \bdell; C_N)$. This concludes the proof of the support condition \eqref{eq: decomp S}.\medskip




\noindent \textit{Step 5: Conclusion.} By combining \eqref{eq: decomp 1 1}, \eqref{eq: decomp 1 3} and \eqref{eq: decomp 4 7}, we conclude that
\begin{equation*}
    a^k = \sum_{M=1}^N\sum_{\bdell \in \cL(M, N; k)} a^{k,\bdell} = \sum_{\bdell \in \cL(N; k)} a^{k,\bdell},
\end{equation*}
where each of the symbols $a^{k,\bdell}$ satisfies the desired kernel bound \eqref{eq: decomp K} and support condition \eqref{eq: decomp S}. This closes the induction and completes the proof. 
\end{proof}




\subsection{Properties of the rescaled symbols} It remains to prove the claim from Step 4 of the proof of Proposition~\ref{prop: fine decomp}. This is a consequence of the following lemma.

    \begin{lemma}\label{lem: verify rescaling}
    Let $1 \leq \ell < \ceil{k/N}$ and $\nu \in \Z$ and consider $b_{N, \nu}^{k, \ell}$ as defined in \eqref{eq: decomp 4 2} in Step 4 of the proof of Proposition~\ref{prop: fine decomp}. Then: \vspace{2pt}
\begin{enumerate}[\normalfont(i)]
    \item $\supp{} b_{N, \nu}^{k, \ell} \subseteq [-1,1] \times \{ \tilde{x} \in \R^n : |\tilde{x}| \sim_{A,B} 2^{k-N\ell} \}$; \vspace{2pt}
    \item $\| \partial_s^j   b_{N, \nu}^{k, \ell}  \|_{L^\infty(\R \times \R^n)} \lesssim_{A,B} 1$ for $0 \leq j \leq n$; \vspace{2pt}
    \item $\displaystyle \max_{1 \leq j \leq N-1} | (\gamma_{N, \nu}^{\ell})^{(j)}(\tilde{s}) \cdot \tilde{x} | \gtrsim_{A,B} 2^{k-N\ell}$ for $(\tilde{s}; \tilde{x}) \in \supp{}  b_{N, \nu}^{k, \ell} $.
\end{enumerate}
\end{lemma}

\begin{proof} Fix $1 \leq \ell < \ceil{k/N}$ and $\nu \in \Z$ and let $(\tilde{s};\tilde{x}) \in \supp{} b_{N, \nu}^{k, \ell}$, so that 
\begin{equation}\label{eq: verify 1}
 (s; x) \coloneq \big(s_{\nu}^{\ell} + 2^{-\ell}s; \big(\Dil_{N, \nu}^{\ell}\big)^{-\top}\tilde{x} \big)   \quad \textrm{satisfies} \quad (s; x) \in \supp{} a_{N, \nu}^{k, \ell}.
\end{equation}
By definition of the symbol from \eqref{eq: decomp 2 5} and \eqref{eq: decomp 2 6}, we have 
\begin{equation}\label{eq: verify 2a}
  \sum_{j=1}^{N-1} 2^{(N-j) \ell} | \gamma^{(j)}(s_{\nu}^{\ell}) \cdot x| \lesssim_{A, B} 2^k \quad \textrm{and} \quad   \sum_{j=1}^{N-1} 2^{(N-j) \ell} | \gamma^{(j)}(s) \cdot x| \gtrsim_{A, B} 2^k.
\end{equation}
Indeed, \eqref{eq: verify 2a} is essentially the content of \cite[Lemma 11]{GovindanSheri2024} (see also \cite{KLO2023}). For completeness, presently, we nevertheless detail the proof of \eqref{eq: verify 2a} via a Taylor expansion.

Since $(s;x) \in \supp{} a_N^{k, \ell}$, we have $2^{-2 \ell - 2} \leq |G_N^k(s; x)| \leq 2^{-2 \ell + 1}$. In particular,  
\begin{equation} \label{eq: verify 3}
    |\gamma^{(j)} \circ \theta(x) \cdot x| \lesssim \varepsilon_1 2^{k-(N-j)\ell}, \quad 1 \leq j \leq N-1, \quad \textrm{and} \quad |s - \theta(x)| \leq \varepsilon_1 2^{-\ell + 1}.
\end{equation}
Furthermore, $|s_{\nu}^{\ell} - \theta(x)| \lesssim 2^{-\ell}$. Taylor's theorem, using \eqref{eq: verify 3}, \eqref{eq: decomp 2 3} and \eqref{eq: G class 1}, implies that
\begin{equation*}
    |\gamma^{(j)}(s_{\nu}^{\ell}) \cdot x| \lesssim_B \sum_{i = j}^{N-1} 2^{k-(N-i)\ell} 2^{-(i -j)\ell} + 2^{k - (N-j)\ell} \lesssim 2^{k-(N-j)\ell}, \quad 1 \leq j \leq N-1,
\end{equation*}
establishing the upper bound in \eqref{eq: verify 2a}.

Turning to the lower bound, we consider the following two cases.\smallskip

\noindent \textit{Case 1: $|s - \theta(x)| \geq \varepsilon_1 2^{-\ell - 2} $}. By Taylor's theorem and \eqref{eq: decomp 2 3}, we have,
\begin{align*}
    |\gamma^{(N-1)}(s) \cdot x| &= \big|\big(\gamma^{(N-1)}(s) - \gamma^{(N-1)}(\theta(x)) \big) \cdot x\big|\\
    &\geq \big(A^{-1} - 2BA\varepsilon_1\big)2^k|s - \theta(x)|,
\end{align*}
where we have used \eqref{eq: decomp 2 2}, \eqref{eq: G class 1} and \eqref{eq: verify 3}. Since $\varepsilon_1 = O_{A,B}(\varepsilon)$, provided $\varepsilon > 0$ is chosen sufficiently small, depending only on $A$ and $B$, we have 
\begin{equation*}
   |\gamma^{(N-1)}(s) \cdot x| \gtrsim_{A, B} \varepsilon_1 2^{k-\ell},
\end{equation*}
using the hypothesis of the case.\smallskip

\noindent \textit{Case 2: $|s - \theta(x)| < \varepsilon_1 2^{-\ell-2}$.} The lower bound $|G_N^k(s,x)| \geq 2^{-2\ell-2}$ implies that
\begin{equation}\label{eq: verify 4}
    |\gamma^{(m)} \circ \theta(x) \cdot x| \sim \varepsilon_1 2^{-\ell(N-m)} \qquad \textrm{for some $1 \leq m \leq N-2$. }
\end{equation}
 By Taylor's theorem, again using \eqref{eq: verify 3} and \eqref{eq: G class 1}, we have
\begin{equation}\label{eq: verify 5}
    |\gamma^{(m)}(s) \cdot x| \geq |\gamma^{(m)}\circ \theta(x) \cdot x| - O_B\big(\varepsilon_1^2 2^{k - (N - m) \ell} \big)
\end{equation}
Here, when bounding the remainder, we used that $N-m \geq 2$, so that $\varepsilon_1^{N-m} \leq \varepsilon_1^2$. Thus, provided $\varepsilon > 0$ is chosen sufficiently small depending only on $n$, $A$ and $B$, the above bounds \eqref{eq: verify 4} and \eqref{eq: verify 5} combine to give
\begin{equation*}
    |\gamma^{(m)}(s) \cdot x|  \gtrsim \varepsilon_1 2^{k-(N-m)\ell}.
\end{equation*} 

In either case, we see that the lower bound in \eqref{eq: verify 2a} holds.\medskip

\noindent i) Comparing the definitions \eqref{eq: decomp 4 2} and \eqref{eq: decomp 2 7}, we see $\supp{\tilde{s}} b_{N, \nu}^{k, \ell} \subseteq \supp{} \zeta \subseteq [-1,1]$.

There exists some unit vector $v \in V_N(s_{\nu}^{\ell})^{\perp}$ such that
\begin{equation}\label{eq: verify 6}
    |\tilde{x}| \lesssim \sum_{j=1}^N  | \gamma^{(j)}(s_{\nu}^{\ell}) \cdot \tilde{x}| + |v \cdot \tilde{x}|.
\end{equation}
By unpacking the definition of $\Dil_{N, \nu}^{\ell}$ from \eqref{eq: decomp 4 1} and \eqref{eq: scaling map 1}, for $1 \leq j \leq N$ we see that
\begin{equation}\label{eq: verify 7}
\gamma^{(j)}(s_{\nu}^{\ell}) \cdot \tilde{x} = 2^{-j \ell} \gamma^{(j)}(s_{\nu}^{\ell}) \cdot x  \quad \textrm{and} \quad  v \cdot \tilde{x} = 2^{-N\ell} v \cdot x. 
\end{equation}
Combining \eqref{eq: verify 6}, \eqref{eq: verify 7} and \eqref{eq: verify 2a}, and recalling that $|x| \sim_A 2^k$ by property i) of Definition~\ref{dfn: symb depth N}, we obtain
\begin{equation*}
    2^{N\ell}|\tilde{x}| \lesssim 2^{N\ell}\sum_{j=1}^{N-1}  | \gamma^{(j)}(s_{\nu}^{\ell}) \cdot \tilde{x}| + 2^{N\ell}| \gamma^{(j)}(s_{\nu}^{\ell}) \cdot \tilde{x}| + 2^{N\ell}|v \cdot \tilde{x}| \lesssim_A |x|.
\end{equation*}
However, recalling the operator norm bound \eqref{eq: scaling map 2}, the reverse form of the above inequality also holds. In particular,
\begin{equation*}
2^{k-N\ell} \lesssim_A  2^{-N\ell}|x| \lesssim |\tilde{x}| \lesssim  2^{-N \ell} |x| \lesssim_A 2^{k- N \ell}.
\end{equation*}
This is the desired estimate. \medskip

\noindent ii) The derivative bounds are an immediate consequence of \eqref{eq: decomp 2 8} and the definition of the scaling in \eqref{eq: decomp 4 2}.\medskip

\noindent iii) Once again, let $(\tilde{s};\tilde{x}) \in \supp{} b_{N, \nu}^{k, \ell}$ and defined $(s; x) \in \supp{} a_{N, \nu}^{k, \ell}$ as in \eqref{eq: verify 1}. Arguing as in \eqref{eq: decomp 4 11}, for $j \in \N$ we have
\begin{equation}\label{eq: verify 8}
    (\gamma_{N, \nu}^{\ell})^{(j)}(\tilde{s}) \cdot \tilde{x} = 2^{-j\ell}\big(\Dil_{N, \nu}^{\ell}\big)^{-1} \gamma^{(j)}(s) \cdot \tilde{x} = 2^{-j\ell} \gamma^{(j)}(s) \cdot x.
\end{equation}
If we substitute \eqref{eq: verify 8} into \eqref{eq: verify 2a}, then we obtain
\begin{equation*}
2^{N\ell} \sum_{j = 1}^{N-1} | (\gamma_{N, \nu}^{\ell})^{(j)}(\tilde{s}) \cdot \tilde{x}| \gtrsim_{A, B} 2^k,
\end{equation*}
as required.    
\end{proof}




\section{Proof of \texorpdfstring{$L^p$-improving}{Lp improving} III: A high-low argument}\label{sec: Lp improv III}




\subsection{Geometric convolution operators} 

In \S\ref{sec: Lp improv IV} below, we shall use the decomposition described in \S\ref{sec: Lp improv II} to reduce Theorem~\ref{thm: symbolconvolutionbound} to $L^p$-improving estimates for certain geometric averaging operators, which take averages over the sets $\Lambda_{\gamma}(\bdell; C)$ introduced in Definition~\ref{dfn: plate family}. In this section, we introduce the geometric averaging operators and establish their $L^p$-improving properties. 

Given $\bdell \in \N_0^{N-M-1}$ for $1 \leq M \leq N \leq n$, recall the definition of the exponents $D_j(\bdell)$ from Definition~\ref{dfn: plate family}. Let
\begin{equation}\label{eq: D and E}
    D(\bdell) \coloneq \sum_{j=1}^n D_j(\bdell) = \sum_{j=M}^N \frac{j(j-1)}{2} \ell_j \quad \textrm{and} \quad E(\bdell) \coloneq \sum_{j = \max\{3, M\}}^N \frac{j(j-1)}{2} \ell_j,
\end{equation}
so that $D(\bdell) \geq E(\bdell)$ and $D(\bdell) = E(\bdell)$ when $M \geq 3$. In this section, we establish the following $L^p$-improving estimate.

\begin{proposition}\label{prop: Lambda conv}  Let $\gamma \colon [-2,2] \to \R^n$ be a smooth, non-degenerate curve and $\bdell \in \N_0^{N-M+1}$ for $1 \leq M \leq N \leq n$ and $C \geq 1$. For all $1 \leq p < p_n$, we have
    \begin{equation*}
        \|\bdone_{\Lambda_{\gamma}(\bdell; C)} * f \|_{L^{p'}(\R^n)} \lesssim_{\gamma, C} 2^{-2E(\bdell)/p'} \| f \|_{L^p(\R^n)}.
    \end{equation*} 
\end{proposition}

For minor technical reasons,\footnote{Mainly so that Lemma \ref{lem: Lambda contain} below holds.} in order to prove Proposition~\ref{prop: Lambda conv}, we slightly modify the definition of the sets $\Lambda_{\gamma}(\bdell; C)$. For $B \geq 1$, let $\kappa_B > 0$ be a constant, chosen sufficiently small depending only on $n$ and $B$ so as to satisfy the forthcoming requirements of our arguments.

\begin{definition}\label{dfn: mod plates} Let $\gamma \colon [-2,2] \to \R^n$ be a smooth, non-degenerate curve and $B \geq 1$ be the minimal value such that $\gamma \in \fkG(B)$. Given $C \geq 1$, let $c > 0$ be dyadic and such that $ c \leq \kappa_B C^{-2} < 2c$. Letting $\bdell \in \N_0^{N-M+1}$ for $1 \leq M \leq N \leq n$, for $\bdnu = (\nu_M, \dots, \nu_N) \in \Z^{N-M+1}$ we redefine 
    \[ s_{\bdnu}^{\bdell} \coloneq  c\sum_{j = M}^N 2^{-\sum_{i = j}^N \ell_i} \nu_j \]
    and
    \begin{equation*}
    \fkN(\bdell) \coloneq  \left\{ (\nu_M, \dots, \nu_N) \in \Z^{N-M+1} : |\nu_j| \leq c^{-1}2^{\ell_j}, \,M \leq j \leq N;\; \left|s_{\bdnu}^{\bdell}\right| \leq 2\right\}.
\end{equation*}
Given $\bdnu = (\nu_M, \dots, \nu_N) \in \fkN(\bdell)$, we let $\Theta_{\bdnu}^{\bdell} \coloneq \Theta_{M,N}(s_{\bdnu}^{\bdell}, \bdrho(\bdell), C)$  be as in Definition~\ref{dfn: plate family}, but now defined with respect to $s_{\bdnu}^{\bdell}$ as above. Similarly, for $k \in \N$ we let $\Lambda_{\gamma}(\bdell; C)$ and $\Lambda_{\gamma}(k, \bdell; C)$ be as in Definition~\ref{dfn: plate family}, but now defined with respect to $\fkN(\bdell)$ and $\Theta_{\bdnu}^{\bdell}$ as above. 
\end{definition}

Note that, although $\fkN(\bdell)$ and $s_{\bdnu}^{\bdell}$ now depend on $\gamma$ and the technical parameter $C$, we suppress this in the notation. 

The redefined sets are supersets of the original definitions: we are simply including plates from a finer mesh of centres. In particular, if we establish Proposition~\ref{prop: Lambda conv} under the new definitions, then it remains valid under the original definitions. For the remainder of this section, we work under the above new definition.

Proposition~\ref{prop: Lambda conv} is a consequence of a smooth variant, which we describe presently.

\begin{definition}\label{dfn: geometric kernel} Let $\gamma \colon [-2,2] \to \R^n$ be a smooth, non-degenerate curve, $\bdell \in \N_0^{N-M+1}$ for $1 \leq M \leq N$, and $C \geq 1$. 

\begin{enumerate}[\normalfont(a)]
    \item Given $\bdnu \in \fkN(\bdell)$, we let $\chi_{\bdnu}^{\bdell} \in C^{\infty}_c(\R^n)$ be a non-negative smooth function which is adapted to $\Theta_{\bdnu}^{\bdell} \coloneq \Theta_{M,N}(s_{\bdnu}^{\bdell}, \bdrho(\bdell), C)$ in the following sense:
\begin{enumerate}[\normalfont(i)]
    \item $ \bdone_{\Theta_{\bdnu}^{\bdell}}(x) \leq \chi_{\bdnu}^{\bdell}(x) \leq  \bdone_{\overtheta_{\bdnu}^{\bdell}}(x)$ for all $x \in \R^n$, where $\overtheta_{\bdnu}^{\bdell} \coloneq \Theta_{M,N}(s_{\bdnu}^{\bdell}, \bdrho(\bdell), 2C)$;
    \item For all $1 \leq j \leq n$ and $x \in \R^n$, we have 
    \[  | ( v_j(s_{\bdnu}^{\bdell}) \cdot \nabla )^d \chi_{\bdnu}^{\bdell}(x) | \lesssim_{\gamma} C^d \rho_j(\bdell)^{-d} \quad \text{for all $0 \leq d \leq 2$}. \]
\end{enumerate}
\item Furthermore, we define 
\begin{equation*}
   \chi_{(C)}^{\bdell} \coloneq \sum_{\bdnu \in \fkN(\bdell)} \chi_{\bdnu}^{\bdell} \quad \textrm{so that} \quad \bdone_{\Lambda_{\gamma}(\bdell; C)} \leq \chi_{(C)}^{\bdell} \quad \textrm{and} \quad \supp{} \chi_{(C)}^{\bdell} \subseteq \Lambda_{\gamma}(\bdell; 2C).
\end{equation*}
\end{enumerate} 
\end{definition}

We make two observations:
\begin{enumerate}
    \item The definition of the functions $\chi_{\bdnu}^{\bdell}$ depends on the technical parameter $C$, but we suppress this in the notation.

    \item Since (i) implies the support of $\chi_{\bdnu}^{\bdell}$ has volume at most $|\overtheta_{\bdnu}^{\bdell}| \lesssim C^{O(1)} 2^{-D(\bdell)}$, integration by parts and the bounds (ii) yields that
    \begin{equation} \label{eq: high 2}
    \left|\left(\chi_{\bdnu}^{\bdell}\right)\widehat{\vphantom{A}}\,\,(\xi)\right| \lesssim_{\gamma} C^{O(1)} 2^{-D(\bdell)} \left(1 + \sum_{j=1}^n \rho_j(\bdell)| v_j(s_{\bdnu}^{\bdell}) \cdot \xi|\right)^{-2}.
    \end{equation}
    Geometrically, this bound tells us that the Fourier transform of $\chi_{\bdnu}^{\bdell}$ decays rapidly away from the dual plate $(\Theta_{\bdnu}^{\bdell})^\circ$.
\end{enumerate}

The smooth variant of Proposition~\ref{prop: Lambda conv} reads thus.

\begin{restatable}{proposition}{propchiconv} \label{prop: chi conv}
    Let $\gamma \colon [-2,2] \to \R^n$ be a smooth, non-degenerate curve and $\bdell \in \N_0^{N-M+1}$ for $3 \leq M \leq N \leq n$ and $C \geq 1$. For all $1 \leq p < p_n$, we have
    \begin{equation*}
        \|\chi_{(C)}^{\bl} * f \|_{L^{p'}(\R^n)} \lesssim_\gamma C^{O(1)} 2^{-2D(\bdell)/p'} \| f \|_{L^p(\R^n)}.
    \end{equation*} 
\end{restatable}

Here the exponent $p_n$ is as in \eqref{eq: U gamma}. In practice, we shall take $C = O_{\gamma}(1)$ and so the dependence on this parameter is harmless; however, it is useful to loosely track the dependence on $C$ in view of certain technicalities in the method used to prove Proposition~\ref{prop: chi conv} below; namely, the inductive argument used to prove that proposition for a certain value of $C$ requires using the induction hypothesis with a slightly larger value of $C$.

We shall show Proposition~\ref{prop: chi conv} implies Proposition~\ref{prop: Lambda conv} at the end of this section. By combining the symbol decomposition from \S\ref{sec: Lp improv II} with Proposition~\ref{prop: Lambda conv}, it is a straightforward matter to conclude the proof of Theorem~\ref{thm: symbolconvolutionbound}. We describe this reduction in \S\ref{sec: Lp improv IV} below.

We briefly describe the strategy used to prove Proposition~\ref{prop: chi conv}. We use a variant of the \textit{high-low} method, featured in, for instance, \cite{GWZ2020}. We break the kernel \smash{$\chi^{\bdell}_{(C)}$} into Littlewood--Paley pieces. At high frequencies, the dual plates $\Theta_{\bdnu}^{\bdell, \circ}$ are disjoint, which yields a gain via orthogonality. At low frequencies, the plates are blended out, which gives a favourable gain which we exploit using Young's convolution inequality. An induction-on-scale then handles the intermediate frequency scales, essentially by reduction to the previous two cases. 




\subsection{Littlewood--Paley projectors} We use a variant of the standard Littlewood--Paley decomposition, often referred to as a \textit{Calder\'on decomposition}, which the authors learned from \cite{Christ_notes}.

Fix $\psi \in C^{\infty}_c(\R^n)$ a real-valued, radial function which is not identically zero and satisfies
\begin{equation*}
  \supp{} \psi \subseteq \{ x \in \R^n : 1/2 \leq |x| \leq 1 \} \quad \textrm{and} \quad \int_{\R^n} \psi(x)\,\ud x = 0. 
\end{equation*}
It follows that the Fourier transform $\widehat{\psi} \in \cS(\R^n)$ is a real-valued, radial function which is not identically zero and satisfies $\widehat{\psi}(0)  = 0$. For $R > 1$, define
\begin{equation*}
    m_R(\xi) \coloneq \int_{R^{-1}}^R |\widehat{\psi}( {t^{-1}} \xi)|^2 \frac{\ud t}{t} \quad \textrm{and} \quad  m(\xi) \coloneq \int_0^{\infty} |\widehat{\psi}( {t^{-1}} \xi)|^2 \frac{\ud t}{t} \qquad \textrm{for all $\xi \in \widehat{\R}^n$.}
\end{equation*}
It is a straightforward matter to verify the integral defining $m$ converges and, moreover, $m$ is bounded. Thus, $m$ defines a real-valued, radial function and, by a change of variable of integration, satisfies $m(r\xi) = m(\xi)$ and $r > 0$. Consequently, $m$ is a non-zero constant function. Replacing $\psi$ with an appropriate scalar multiple of itself, we may assume $m(\xi) = 1$, and thus $m_R(\xi) \to 1$ as $R \to \infty$ for all $\xi \in \widehat{\R}^n$.

Given  $R > 0$, define $\psi_R(x) \coloneq R^n \psi(Rx)$ and the Littlewood--Paley projector
\begin{equation*}
    P_R f \coloneq f \ast \psi_R \qquad \textrm{for $f \in \cS(\R^n)$.}
\end{equation*}
By the Fourier inversion formula and Fubini's theorem,
\begin{equation*}
  f(x) - \int_{R^{-1}}^R (P^2_t f)(x)\frac{\ud t}{t} = \int_{\widehat{\R}^n}e^{2 \pi i x \cdot \xi} (1 - m_R(\xi)) \widehat{f}(\xi)\,\ud \xi.
\end{equation*}
Thus, for all $2 \leq q \leq \infty$, by the Hausdorff--Young inequality and monotone convergence theorem,
\begin{equation*}
  \lim_{R \to \infty}\left\|f - \int_{R^{-1}}^R (P_{t}^2 f)(\,\cdot\,)\frac{\ud t}{t}\right\|_{L^{q}(\R^n)} = \lim_{R \to \infty} \left\|(1 - m_R) \widehat{f}\,\right\|_{L^{q'}(\R^n)} = 0.
\end{equation*}
By decomposing the integral and applying a change of variable, 
\begin{equation}\label{eq: LP 1}
f = \sum_{k \in \Z} \int_{2^{-k-1}}^{2^{-k}} P_t^2 f(\,\cdot\,)\frac{\ud t}{t} = \sum_{k \in \Z} \int_1^2 P^2_{t 2^k} f(\,\cdot\,)\frac{\ud t}{t},
\end{equation}
where the convergence holds in the $L^q$ sense for all $2 \leq q \leq \infty$.

The main advantage of this decomposition is that we is expresses $f$ as a superposition of functions of the form $P_s^2 f$, for $s \in \R$, where, most importantly, the convolution kernel $\psi_s$ of $P_s$ has compact support.




\subsection{Low frequency estimate} We begin by proving the following estimate at an individual frequency bandwidth.

\begin{proposition} \label{prop: low} Let $\bdell \in \N_0^{N-M+1}$ for some $3 \leq M \leq N \leq n$ and $C \geq 1$. For all $1 \leq p \leq p_n$ and $0 < R \leq 2^{D_1(\ell)}$ we have,
    \begin{equation} \label{eq: low a}
        \left\|\left(P_R \chi_{(C)}^{\bdell}\right) * f\right\|_{L^{p'}(\R^n)} \lesssim_{\gamma} C^{O(1)} (1 + R 2^{-D_2(\bdell)})^{2/p-1}2^{-2D(\bdell)/p'} \|f\|_{L^p(\R^n)}. 
    \end{equation}
\end{proposition}

In particular, under the hypotheses of Proposition~\ref{prop: low}, for all $K \in \N_0$, we have 
    \begin{equation*}
        \left\|\left(P_R \chi_{(C)}^{\bdell}\right) * f\right\|_{L^{p'}(\R^n)} \lesssim_{\gamma} C^{O(1)} 2^{(2/p-1)K}2^{-2D(\bdell)/p'} \|f\|_{L^p(\R^n)} 
    \end{equation*}
for $0 < R \leq \min\{2^{D_2(\bdell) + K}, 2^{D_1(\bdell)}\}$.

To prove Proposition~\ref{prop: low}, we establish a preliminary lemma. Given $1 \leq r \leq \infty$ and $\bdell \in \N_0^{N-M+1}$ for $1 \leq M  \leq N \leq n$, and $R \geq 1$, define 
    \begin{equation}\label{eq: M}
     M(r; \bdell; R)  \coloneq 2^{\|\bl\|/r}  \left( 1 + R^{-1}2^{D_{M-1}(\bdell)}\right) \left(R^n \prod_{j = 1}^n \min\{2^{-D_j(\bdell)}, R^{-1}\}\right)^{1/r'},
    \end{equation}
    where $D_0(\bdell) \coloneq \sum_{j=1}^N j \ell_j$ in the $M = 1$ case. Proposition~\ref{prop: low} is then a fairly direct consequence of the following kernel estimate. 

\begin{lemma} \label{lem: low} Let $n \geq 3$, $\bdell \in \N_0^{N - M +1}$ for some $1 \leq M \leq N \leq n$ and $C \geq 1$. 

\begin{enumerate}[\normalfont(a)]
    \item For all $1 \leq r \leq \infty$ and $R \geq 1$, we have
    \begin{equation} \label{eq: low ker a}
        \| P_R \chi_{(C)}^{\bdell} \|_{L^r(\R^n)} \lesssim_\gamma  C^{O(1)} M(r; \bdell; R) 2^{-D(\bdell)/r}. 
    \end{equation}
    \item For all $1 \leq r \leq \infty$ and $0 < R < 1$, we have
    \begin{equation}\label{eq: low ker b}
    \| P_R \chi_{(C)}^{\bdell} \|_{L^r(\R^n)} \lesssim_\gamma C^{O(1)} R^{n/r'}  2^{\|\bdell\|-D(\bdell)}.
     \end{equation}
\end{enumerate} 
\end{lemma}

Using the above lemma, it is a straightforward matter to deduce Proposition~\ref{prop: low}.

\begin{proof}[Proof (of Proposition~\ref{prop: low}, assuming Lemma~\ref{lem: low})] Fix $\bdell \in \N_0^{N- M +1}$ for $3 \leq M \leq N \leq n$ and $1 \leq p \leq p_n$. Let $r \coloneq p'/2$, noting that $p'/2 \geq p_n'/2 = n(n+1)/4 \geq 3 > 2$, where we applied the definition of $p_n$ from \eqref{eq: U gamma}.\medskip

\noindent \textit{Case 1: $1 \leq R \leq 2^{D_1(\bdell)}$.} The function $M(r; \bdell; R)$ is easily seen to be continuous and nondecreasing over the range $1 \leq R \leq 2^{D_2(\bdell)}$. Since $M \geq 3$, we also have $2^{D_{M-1}(\bdell)} \leq 2^{D_2(\bdell)}$.  Thus, fixing $1 \leq R \leq 2^{D_2(\bdell)}$, we may bound
    \begin{align}
    \nonumber
       M(r; \bdell; R) \leq M\left(r; \bdell; 2^{D_2(\bdell)}\right) &\lesssim 2^{\|\bdell\|/r} 2^{-(D_1(\bdell) - D_2(\bdell))/r'}\\
       \label{eq: low 1}
       &= 2^{-\|\bdell\|(1/r' - 1/r)} \leq 1.
    \end{align}
    %
    %
    %
    %
    %
    %
    %
    For $2^{D_2(\bdell)} < R \leq 2^{D_1(\bdell)}$, we have $R^{-1}2^{D_{M-1}(\bdell)} \leq 1$. Exploiting that $D_1(\bdell) = D_2(\bdell) + \| \bdell \|$ and $r \geq 2$, we conclude
    \begin{equation}\label{eq: low 2}
    \begin{split}
        M(r; \bdell; R) &\lesssim 2^{\|\bdell\|/r} \left( R 2^{-D_1(\bdell)} \right)^{1/r'}\\
        &= 2^{\|\bdell\|(1/r - 1/r')} (R2^{-D_2(\bdell)})^{1/r'} \leq (R2^{-D_2(\bdell)})^{1/r'}. 
    \end{split}
    \end{equation}
    By combining \eqref{eq: low ker a} with either \eqref{eq: low 1} or \eqref{eq: low 2} and applying Young's inequality, for all $1 \leq R \leq 2^{D_1(\bdell)}$ we have
    \begin{align*}
        \left\|\left(P_R \chi^{\bdell}\right) * f\right\|_{L^{p'}(\R^n)} &\leq \|P_R \chi^{\bdell}\|_{L^{p'/2}(\R^n)}  \|f\|_{L^p(\R^n)} \\
        &\lesssim_{\gamma} C^{O(1)} (1 + R 2^{-D_2(\bdell)})^{2/p-1}2^{-2D(\bdell)/p'} \|f\|_{L^p(\R^n)},
    \end{align*}
    which is the required estimate \eqref{eq: low a} in this case. \medskip

\noindent \textit{Case 2: $0 < R < 1$.} Observe that
\begin{equation*}
   \frac{j(j-1)}{2} \geq \frac{M(M-1)}{2} \geq 3 \quad \textrm{for} \quad M \leq j \leq N\quad \textrm{and} \quad 1 - \frac{2}{p'} \geq 1 - \frac{1}{3} = \frac{2}{3}. 
\end{equation*} Consequently, recalling the formula for $D(\bdell)$ from \eqref{eq: D and E}, we have
\begin{equation}\label{eq: low 3}
    D(\bdell) - \|\bdell\| = \frac{2D(\bdell)}{p'} + \sum_{j=M}^N \left( \frac{j(j-1)}{2}\left(1 - \frac{2}{p'}\right) - 1\right)\ell_j \geq \frac{2D(\bdell)}{p'}.
\end{equation}
Let $0 < R < 1$. By combining \eqref{eq: low ker b} and \eqref{eq: low 3} and applying Young's inequality,
    \begin{equation*}
        \left\|\left(P_R \chi^{\bdell}\right) * f\right\|_{L^{p'}(\R^n)} \leq \|P_R \chi^{\bdell}\|_{L^{p'/2}(\R^n)}  \|f\|_{L^p(\R^n)}\lesssim_{\gamma} C^{O(1)} 2^{-2D(\bdell)/p'} \|f\|_{L^p(\R^n)},
    \end{equation*}
    which is the required estimate \eqref{eq: low a} in this case. 
\end{proof}

For future reference,  we also highlight a further corollary of Lemma~\ref{lem: low}.

\begin{corollary}\label{cor: kernel est} If $\bdell \in \N_0^{N - M +1}$ for some $3 \leq M \leq N \leq n$, $C \geq 1$ and $R \geq 2^{D_2(\bdell)}$, then
\begin{equation*}
    \left\|\left(P_R \chi^{\bdell}_{(C)}\right)\ast f \right\|_{L^{\infty}(\R^n)} \lesssim_{\gamma} C^{O(1)} \|f\|_{L^1(\R^n)}.
\end{equation*}
\end{corollary}

\begin{proof}  Recalling the definition in \eqref{eq: M}, we have
\begin{equation*}
    M(\infty; \bdell; R) \leq ( 1 + R^{-1} 2^{D_{M-1}(\bdell)}) \lesssim  1,
\end{equation*}
where we have used $2^{D_{M-1}(\bdell)} \leq 2^{D_2(\bdell)}$. Thus, $ \|P_R \chi^{\bdell}\|_{L^{\infty}(\R^n)} \lesssim_{\gamma} C^{O(1)}$ by Lemma~\ref{lem: low}, which immediately implies the desired estimate. 
\end{proof}

We now turn to the proof of the key kernel estimates in Lemma~\ref{lem: low}.

\begin{proof}[Proof (of Lemma~\ref{lem: low})] Let $\bdell \in \N_0^{N - M +1}$ for some $1 \leq M \leq N \leq n$ and $1 \leq r \leq \infty$.\medskip

\noindent (a) Fix $R \geq 1$ and recall $\psi_R$ denotes the convolution kernel of the frequency projection $P_R$. Letting $\rho \coloneq R^{-1}$, we have
\begin{equation*}
    |\psi_R(x)| \lesssim \frac{1}{|B(0,\rho)|} \bdone_{B(0, \rho)}(x) \quad \textrm{and} \quad |\chi_{\bdnu}^{\bdell}(x)| \leq \bdone_{\overtheta_{\bdnu}^{\bdell}}(x) \quad \textrm{for all $\bdnu \in \fkN(\bdell)$.}
\end{equation*}
Thus, 
\begin{equation}\label{eq: low ker 1}
    |P_R \chi_{\bdnu}^{\bdell}(x)| \leq |\psi_R| \ast |\chi_{\bdnu}^{\bdell}|(x) \lesssim \sup_{x \in \R^n} \frac{|B(x,\rho) \cap \overtheta_{\bdnu}^{\bdell}|}{|B(x,\rho)|}. 
\end{equation}
Let $N(\overtheta{\vphantom{\Theta}}_{\bdnu}^{\bdell}, \rho)$ denote the Euclidean $\rho$-neighbourhood of $\overtheta_{\bdnu}^{\bdell}$ and define
\begin{equation*}
    \fkN(\bdell; x) \coloneq \big\{  \bdnu \in \fkN(\bdell) : x \in N(\Theta_{\bdnu}^{\bdell}, \rho) \}.
\end{equation*}
Note that the above observations also show that
\begin{equation}\label{eq: low ker 2}
    x \in \supp{} P_R \chi_{\bdnu}^{\bdell} \quad \Rightarrow \quad \bdnu \in \fkN(\bdell, x).
\end{equation}

By comparing the principal radii of $\overtheta_{\bdnu}^{\bdell}$ with the radius $\rho$, we see that
\begin{equation}\label{eq: low ker 3}
   |B(x, \rho) \cap \overtheta_{\bdnu}^{\bdell}| \lesssim_{\gamma} C^{O(1)} \prod_{j = 1}^n \min\{2^{-D_j(\bdell)},\rho\} \qquad \textrm{for all $x \in \R^n$.}
\end{equation}
Combining \eqref{eq: low ker 1}, \eqref{eq: low ker 2} and \eqref{eq: low ker 3}, and recalling $\rho \coloneq R^{-1}$, we obtain 
\begin{equation} \label{eq: low ker 4}
\begin{split}
    |P_R \chi^{\bdell}(x)| &\leq \sum_{\bdnu \in \fkN(\bdell; x)} |P_R \chi_{\bdnu}^{\bdell}(x)|\\
    &\lesssim_\gamma C^{O(1)} R^n \left(\prod_{j=1}^n \min\{ 2^{-D_j(\bdell)}, R^{-1}\}\right) \#\fkN(\bdell; x).
\end{split}
\end{equation}

We claim that
\begin{equation} \label{eq: low ker 5}
        \# \fkN(\bl;x) \lesssim_{\gamma} C^{O(1)} \big( 1 +  R^{-1}2^{D_{M-1}(\bdell)}\big).
    \end{equation}
We shall show below that this follows from Lemma~\ref{lem: plate sep}. Temporarily assuming \eqref{eq: low ker 5} holds, we may combine it with \eqref{eq: low ker 4} to give
    \begin{equation}\label{eq: low ker 6}
        \| P_R \chi^{\bdell} \|_{L^\infty(\R^n)} \lesssim_\gamma C^{O(1)} \big( 1 +  R^{-1}2^{D_{M-1}(\bdell)}) R^n \left(\prod_{j=1}^n \min\{ 2^{-D_j(\bdell)}, R^{-1}\}\right),
    \end{equation}
    which is the $r = \infty$ case of the required bound \eqref{eq: low ker a}. On the other hand, again by comparing the principal radii of $\overtheta_{\bdnu}^{\bdell}$ with the radius $\rho$, we see that
    \begin{equation*}
        |N(\overtheta_{\bdnu}^{\bdell},\rho)| \lesssim_{\gamma} C^{O(1)} \prod_{j = 1}^n \max\{2^{-D_j(\bdell)}, \rho\}.
    \end{equation*}
    From this, and the bound $\#\fkN(\bdell) \lesssim_{\gamma} C^{O(1)} 2^{\|\ell\|}$, we deduce that
    \begin{equation} \label{eq: low ker 7}
        \left|\supp{} \left(P_R \chi^{\bdell}\right)\right| \lesssim_\gamma   C^{O(1)} \left( \prod_{j = 1}^n \max\{2^{-D_j(\bdell)}, R^{-1}\}\right) 2^{\|\bl\|}.
    \end{equation}
    Combining \eqref{eq: low ker 6} and \eqref{eq: low ker 7} then gives 
    \begin{equation} \label{eq: low ker 8}
         \| P_R \chi^{\bdell} \|_{L^1(\R^n)} \lesssim_\gamma C^{O(1)} \big( 1 +  R^{-1}2^{D_{M-1}(\bdell)}) 2^{\|\bl\| - D(\bdell)}.
    \end{equation}
    Interpolating \eqref{eq: low ker 6} and \eqref{eq: low ker 8} via logarithmic convexity of $L^r$, we obtain \eqref{eq: low ker a} for all $1 \leq r \leq \infty$, as required.\smallskip

It remains to prove the claimed estimate \eqref{eq: low ker 5}. For this, we wish to invoke Lemma~\ref{lem: plate sep}. Observe,
\begin{enumerate}[\normalfont(i)]
    \item $0 < \rho_1(\bdell) \leq \dots \leq \rho_n(\bdell) \leq 1$;
\end{enumerate}
this will ensure that hypothesis (i) of Lemma~\ref{lem: plate sep} holds for our application. To ensure the remaining hypotheses also hold, we decompose the set $\fkN(\bdell)$.

Let $B \geq 1$ be the minimal value such that $\gamma \in \fkG(B)$ and $\varepsilon \coloneq \varepsilon_{B,1} > 0$ be the constant appearing in the statement of Lemma~\ref{lem: plate sep}. Partition $\{ s_{\bdnu}^{\bdell} : \nu \in \fkN(\bdell) \}$ into $O_B(C^{O(1)})$ subsets, each $\varepsilon^{-1}2^{-\|\bdell\|}$-separated and of diameter at most $\varepsilon$. We apply Lemma~\ref{lem: force sep} to each of these subsets, with parameters $\delta \coloneq \varepsilon$ and $\varepsilon_j \coloneq \rho_j(\bdell)^{-1} \rho_{j-1}(\bdell)$ for $1 \leq j \leq n$. From this, using the identity $2^{-\|\bdell\|} = \rho_M(\bdell)^{-1} \rho_{M-1}(\bdell)$, we conclude that 
\begin{equation*}
    \fkN(\bdell) = \bigcup_{\alpha =1}^A \fkN_\alpha(\bdell) \qquad \textrm{where} \qquad A \lesssim_B C^{O(1)}
\end{equation*}
and where, for all $1 \leq \alpha \leq A$, for $\bdnu_1$, $\bdnu_2 \in \fkN_\alpha(\bdell)$ distinct we have
\begin{enumerate}[\normalfont(i)]
\addtocounter{enumi}{1}
    \item $|s_{\bdnu_1}^{\bdell} - s_{\bdnu_2}^{\bdell}| \leq \varepsilon$ and $|s_{\bdnu_1}^{\bdell} - s_{\bdnu_2}^{\bdell}| \geq \varepsilon^{-1} \rho_M(\bdell)^{-1} \rho_{M-1}(\bdell)$;
    \item For $M \leq i \leq N-1$, either 
    \begin{equation*}
        |s_{\bdnu_1}^{\bdell} - s_{\bdnu_2}^{\bdell}| \geq \varepsilon^{-1} \rho_{i+1}(\bdell)^{-1} \rho_i(\bdell) \quad \textrm{or} \quad|s_{\bdnu_1}^{\bdell} - s_{\bdnu_2}^{\bdell}| \leq \varepsilon \rho_{i+1}(\bdell)^{-1} \rho_i(\bdell).
    \end{equation*}
\end{enumerate}

Fix $x \in \R^n$ and define
\begin{equation*}
  \fkN_{\alpha}(\bl; x) \coloneq \left\{ \nu \in \fkN_\alpha(\bdell) : x \in N(\overtheta_{\bdnu}^{\bdell}, \rho) \right\}. 
\end{equation*}
Suppose $\fkN_{\alpha}(\bl; x) \neq \varnothing$ and let $\bdnu_1 \in \fkN_{\alpha}(\bl; x)$. If $\bdnu_2 \in \fkN_{\alpha}(\bl; x)$, then it follows that $\dist(\overtheta_{\bdnu_1}^{\bdell}, \overtheta_{\bdnu_2}^{\bdell}) \lesssim \rho$. Thus, applying Lemma~\ref{lem: plate sep}, noting that (i), (ii) and (iii) above verify the required hypotheses,
\begin{equation*}
    |s_{\bdnu_2}^{\bdell} - s_{\bdnu_1}^{\bdell}| \lesssim C^{O(1)} \rho \rho_M(\bdell)^{-1}.
\end{equation*}
However, by the above reductions the set $\{s_{\bdnu}^{\bdell} : \nu \in \fkN_{\alpha}(\bdell)\}$ is $2^{-\|\bdell\|}$-separated. Thus, it follows that 
\begin{equation*}
    \fkN_{\alpha}(\bl;x) \lesssim_{\gamma} C^{O(1)} \big(1+  2^{\|\bl\|} R^{-1} 2^{D_M(\bdell)}\big) \qquad \textrm{for all $1 \leq \alpha \leq A$.}
\end{equation*}
Summing these estimates, and using the bound $\|\bl\| + D_M(\bdell) = D_{M-1}(\bdell)$, we obtain
    \begin{equation*}
        \# \fkN(\bl;x) \leq \sum_{\alpha = 1}^A \# \fkN_{\alpha}(\bl;x)  \lesssim_B C^{O(1)} \big(1 +  R^{-1}2^{D_{M-1}(\bdell)}),
    \end{equation*}
    which is the claimed bound \eqref{eq: low ker 5}.\medskip

    \noindent (b) We use a minor modification of the above argument. For $0 < R < 1$, we have
    \begin{equation*}
        |P_R \chi_{\bdnu}^{\bdell}(x)| \leq \|\psi_R\|_{L^{\infty}(\R^n)} \|\chi_{\bdnu}^{\bdell}\|_{L^1(\R^n)} \lesssim R^n|\overtheta_{\bdnu}^{\bdell}| \lesssim C^{O(1)}  R^n 2^{-D(\bdell)}.
    \end{equation*}
    Thus, by summing over all $\bdnu \in \fkN(\bdell)$ and trivially bounding $\# \fkN(\bdell) \lesssim 2^{\|\bdell\|}$, we have
       \begin{equation} \label{eq: low ker 9}
    |P_R \chi^{\bdell}(x)| \leq \sum_{\bdnu \in \fkN(\bdell)} |P_R \chi_{\bdnu}^{\bdell}(x)| \lesssim_\gamma C^{O(1)}  R^n 2^{\|\bdell\|-D(\bdell)}.
    \end{equation}
    This is the $r = \infty$ case of the required bound \eqref{eq: low ker b}. 

    Turing to the general $L^r$ bound, here the support of $P_R \chi^{\bdell}$ is contained in a ball of radius $O_B(R^{-1})$ and therefore satisfies $\left|\supp{} \left(P_R \chi^{\bdell}\right)\right| \lesssim_B R^{-n}$. Combining this with \eqref{eq: low ker 9} gives the $L^1$-norm bound in \eqref{eq: low ker b}. Interpolating, via logarithmic convexity, we then obtain \eqref{eq: low ker b} for all $1 \leq r \leq \infty$.   
\end{proof}




\subsection{High frequency estimate} We prove an alternative estimate at an individual frequency bandwidth.

\begin{proposition} \label{prop: high} Let $\bdell \in \N_0^{N - M +1}$ for some $3 \leq M \leq N \leq n$. If $1 \leq p \leq 2$ and $R \geq 2^{D_2(\bdell)}$, then
    \begin{equation}\label{eq: high a}
         \left\| \left(P_R \chi_{(C)}^{\bdell}\right) * f \right\|_{L^{p'}(\R^n)} \lesssim_\gamma C^{O(1)}  \left(1 + R^{-1}2^{D_1(\bdell)}\right)^{2/p'} 2^{-D(\bdell)/p'} \| f \|_{L^p(\R^n)}.
    \end{equation}
\end{proposition}

In particular, under the hypotheses of Proposition~\ref{prop: high}, for all $K \in \N_0$, we have 
    \begin{equation*}
         \left\| \left(P_R \chi_{(C)}^{\bdell}\right) * f \right\|_{L^{p'}(\R^n)} \lesssim_\gamma C^{O(1)} 2^{2K/p'} 2^{-2D(\bdell)/p'} \| f \|_{L^p(\R^n)}
    \end{equation*}
for $R \geq \max\{2^{D_1(\bdell) - K}, 2^{D_2(\bdell)}\}$.
\begin{proof}[Proof (of Proposition~\ref{prop: high})] The $p = 1$ case of \eqref{eq: high a} is precisely Corollary~\ref{cor: kernel est}. Thus, by interpolation, it suffices to only prove the $p = 2$ case of \eqref{eq: high a}. Furthermore, by Plancherel's theorem, matters are reduced to establishing
\begin{equation}\label{eq: high 1}
    \left\| {\left(P_R \chi_{(C)}^{\bdell}\right) \widehat{\vphantom{A}} }\;\, \right\|_{L^{\infty}(\widehat{\R}^n)} \lesssim_\gamma  C^{O(1)} \big( 1 + R^{-1}2^{D_1(\bdell)-D(\bdell)} \big).
\end{equation}

     For $\xi \in \widehat{\R}^n$ we define
    \begin{align*}
    \fkN_0(\bdell; \xi) &\coloneq \left\{\bdnu \in \fkN(\bdell): \xi \in \Theta_{\bdnu}^{\bdell, \circ} \right\}; \\
    \fkN_r(\bdell; \xi) &\coloneq \left\{\bdnu \in \fkN(\bdell): \xi \in 2^r \Theta_{\bdnu}^{\bdell, \circ} \setminus 2^{r-1}\Theta_{\bdnu}^{\bdell, \circ} \right\}, \qquad r \in \N. 
    \end{align*}
    It then follows that 
    \begin{equation}\label{eq: high 3}
        \fkN(\bdell) = \bigcup_{r = 0}^{\infty} \fkN_r(\bdell; \xi) \qquad \textrm{for all $\xi \in \widehat{\R}^n$.}
    \end{equation}
    
    By \eqref{eq: high 2}, for each $\xi \in \widehat{\R}^n$ and $r \in \N_0$ we have
     \begin{equation}\label{eq: high 4} 
     \sum_{\bdnu \in \fkN_r(\bdell; \xi)}\left| \left(\chi_{\bdnu}^{\bdell}\right)\widehat{\vphantom{A}}\,\,(\xi) \right|  \lesssim_\gamma C^{O(1)} 2^{-D(\bdell)} 2^{-2r} \#\fkN_r(\bdell; \xi).
    \end{equation}
We further claim that
\begin{equation}\label{eq: high 5}
    \#\fkN_r(\bdell; \xi) \lesssim_{\gamma} C^{O(1)}\big(1 + 2^r 2^{D_1(\bdell)} |\xi|^{-1}\big). 
\end{equation}
Assuming this is the case, we may combine \eqref{eq: high 4} and \eqref{eq: high 5} and sum over $r \in \N_0$ to deduce that
    \begin{align}
    \nonumber
    \left| \left(\chi_{(C)}^{\bdell} \right) \widehat{\vphantom{A}}\,\, (\xi)\right| &\leq \sum_{r = 0}^{\infty} \sum_{\bdnu \in \fkN_r(\bdell; \xi)}\left|\left(\chi_{\bdnu}^{\bdell}\right) \widehat{\vphantom{A}}\,\, (\xi)\right|\\ 
         \label{eq: high 6}
         &\lesssim_\gamma C^{O(1)} \big(1 + 2^{D_1(\bdell)} |\xi|^{-1}\big)2^{-D(\bdell)} \quad \textrm{for all $\xi \in \widehat{\R}^n$,}
    \end{align}
    where the first inequality is a consequence of the triangle inequality and \eqref{eq: high 3}. Since $\widehat{\psi}$ vanishes at the origin, the mean value theorem and \eqref{eq: high 6} yield
        \begin{equation*}
            \left|\left(P_R \chi_{(C)}^{\bdell}\right) \widehat{\vphantom{A}}\,\, (\xi)\right| = |\widehat{\psi}(R^{-1} \xi )|\left|\left(\chi_{(C)}^{\bdell}\right) \widehat{\vphantom{A}}\,\, (\xi)\right| \lesssim_\gamma C^{O(1)} \left(1 +  R^{-1} 2^{D_1(\bdell)}\right)2^{-D(\bdell)} 
        \end{equation*}
for all $\xi \in \widehat{\R}^n$. This is precisely the desired estimate \eqref{eq: high 1}. 

It remains to verify the claim \eqref{eq: high 5}. Let $\varepsilon = \varepsilon_B > 0$ be the constant appearing in Lemma~\ref{lem: dual overlap}, where $B \geq 1$ is the minimal value such that $\gamma \in \fkG(B)$. For $r \in \N_0$, we consider the following cases.\smallskip

\noindent \textit{Case 1: $|\xi| \geq \varepsilon^{-1} 2^r 2^{D_2(\bdell)}$}. Here we may apply Lemma~\ref{lem: dual overlap} to conclude that
    \begin{equation*}
        \left\{ s \in [-1,1] : \xi \in 2^r \Theta^\circ(s, \bdrho(\bdell)) \right\} 
    \end{equation*} 
    is contained in $O_{\gamma}(1)$ intervals, each of length $O_{\gamma}( 2^r \rho_2(\bdell)^{-1} |\xi|^{-1} )$. Note that $\fkN(\bdell)$ is the union of $O_{\gamma}(C^{O(1)})$ sets which are each $2^{-\|\bl\|}$-separated. Since $2^{-\|\bdell\|} = 2^{-D_1(\bdell)} \rho_2(\bdell)^{-1}$, the desired bound \eqref{eq: high 5} therefore follows.\smallskip

    \noindent \textit{Case 2: $|\xi| < \varepsilon^{-1} 2^r 2^{D_2(\bdell)}$}. Rearranging the hypothesis of the case,
    \begin{equation*}
        2^{\|\bdell\|} = 2^{D_1(\bdell) - D_2(\bdell)} < \varepsilon^{-1} 2^r 2^{D_1(\bdell)} |\xi|^{-1}.
    \end{equation*}
    The desired result follows from the trivial bound $\#\fkN_r(\bdell; \xi) \leq \#\fkN(\bdell) \lesssim_{\gamma} C^{O(1)} 2^{\|\bdell\|}$.\smallskip

Thus, we see \eqref{eq: high 5} holds in either case, which concludes the proof.
\end{proof}




\subsection{Proof of Proposition~\ref{prop: chi conv}: Technical preliminaries}

Before concluding the proof of Proposition~\ref{prop: chi conv}, we present some minor technical lemmas which feature in the argument. 

\begin{lemma}\label{lem: D exp} For $1 \leq p < p_n$, define $\delta(p) \coloneq 2(1/p_n' - 1/p') > 0$.  If $\bdell$, $\tbdell \in \N_0^{N- M +1}$ for $1 \leq M \leq N \leq n$ satisfy $\tilde{\ell}_j \leq \ell_j$ for $M \leq j \leq N$, then
    \begin{equation*}
    D_1(\bdell) - D_1(\tbdell) - \frac{2}{p'}\big(D(\bdell) - D(\tbdell) \big) \geq \delta(p)\big(D_1(\bdell) - D_1(\tbdell) \big).
\end{equation*}
\end{lemma}

\begin{proof} It follows from the formula in \eqref{eq: D and E} that
\begin{equation*}
    D(\bdell) - D(\tbdell) = \sum_{j=M}^N \frac{(j-1)j}{2}(\ell_j - \tilde{\ell}_j) \geq \sum_{j=M}^N (j-1)(\ell_j - \tilde{\ell}_j) = D_1(\bdell) - D_1(\tbdell) \geq 0.
\end{equation*}
Furthermore, recalling the definition of $p_n$ from \eqref{eq: U gamma}, we have 
\begin{equation*}
    D_1(\bdell) - D_1(\tbdell) - \frac{2}{p_n'}\big(D(\bdell) - D(\tbdell) \big) = \sum_{j=M}^N (j-1)\left(1 - \frac{2j}{n(n+1)}\right)(\ell_j - \tilde{\ell}_j) \geq 0. 
\end{equation*}
Combining the two previous displays yields the desired inequality. 
\end{proof}

\begin{lemma}\label{lem: Lambda contain}
    Suppose $\bdell = (\ell_M, \dots, \ell_N) \in \N_0^{N - M + 1}$ and $\bdell' = (\tilde{\ell}_{M'}, \ell_{M'+1}, \dots, \ell_N)  \in \N_0^{N - M' + 1}$ for some $1 \leq M \leq M' \leq N \leq n$ and $1 \leq \tilde{\ell}_{M'} \leq \ell_{M'}$. Then
    \begin{equation*}
        \Lambda_{\gamma}(\bdell; C) \subseteq \Lambda_{\gamma}(\bdell'; 2C).
    \end{equation*} 
\end{lemma}

\begin{proof} Let $B \geq 1$ be the minimum value such that $\gamma \in \fkG(B)$. Let $x \in \Lambda_{\gamma}(\bdell; C)$, so that $x \in \Theta_{M, N}(s_{\bdnu}^{\bdell}, \bdrho(\bdell), C)$ for some $\bdnu \in \fkN(\bdell)$. Note that this implies $|x| \lesssim BC$. Recalling Definition~\ref{dfn: mod plates}, there exists $\bdnu' \in \fkN(\bdell')$ such that $|s_{\bdnu}^{\bdell} - s_{\bdnu'}^{\bdell'}| \leq c 2^{-\|\bdell'\|}$. For $1 \leq i \leq n$, by Taylor's theorem, 
\begin{align}
\nonumber
    \left|\gamma^{(i)}(s_{\bdnu'}^{\bdell'}) \cdot x - \gamma^{(i)}(s_{\bdnu}^{\bdell}) \cdot x\right| &\lesssim_B \sum_{j = i+1}^{n} \frac{|\gamma^{(j)}(s_{\bdnu}^{\bdell}) \cdot x|}{(j-i)!} |s_{\bdnu'}^{\bdell'} - s_{\bdnu}^{\bdell}|^{j-i} + C |s_{\bdnu'}^{\bdell'} - s_{\bdnu}^{\bdell}|^{n+1-i}\\
    \label{eq: Lambda contain 1}
    &\lesssim_B cC\left( \sum_{j = i+1}^n \rho_j(\bdell)2^{- (j-i) \|\bdell'\|} + 2^{-(n+1-i)\|\bdell'\|}\right).
\end{align}
For $i + 1 \leq j \leq n$, observe that
\begin{equation}\label{eq: Lambda contain 2}
   \rho_j(\bdell) 2^{-\|\bdell'\| (j-i)} \leq \rho_j(\bdell') 2^{-(j-i)\|\bdell'\|} \leq \rho_i(\bdell') \quad \textrm{and} \quad  2^{-(n+1-i)\|\bdell'\|} \leq \rho_i(\bdell'). 
\end{equation} 
Recall from Definition~\ref{dfn: mod plates} that $c \sim \kappa_B C^{-2}$. Provided $\kappa_B > 0$ is chosen sufficiently small, depending only on $n$ and $B$, combining \eqref{eq: Lambda contain 2} and \eqref{eq: Lambda contain 1} yields
\begin{equation}\label{eq: Lambda contain 3}
    \left|\gamma^{(i)}(s_{\bdnu'}^{\bdell'}) \cdot x - \gamma^{(i)}(s_{\bdnu}^{\bdell}) \cdot x\right| \leq (4C)^{-1} \rho_i(\bdell) \qquad \textrm{for $1 \leq i \leq n$.}
\end{equation}

On the other hand, $\rho_i(\bdell) = \rho_i(\bdell')$ for $M' \leq i \leq N$ and so
\begin{equation}\label{eq: Lambda contain 4}
   \left|\gamma^{(i)}(s_{\bdnu}^{\bdell}) \cdot x\right| \geq (2C)^{-1} \rho_i(\bdell) = (2C)^{-1} \rho_i(\bdell') \quad \textrm{for $M' \leq i \leq N$.}
\end{equation}
In view of \eqref{eq: Lambda contain 3} and \eqref{eq: Lambda contain 4},  it follows that $x \in \Theta_{M', N}(s_{\bdnu'}^{\bdell'}, \bdrho(\bdell'), 2C)$, and so $x \in \Lambda_{\gamma}(\bdell'; 2C)$, as required. 
\end{proof}




\subsection{Proof of Proposition~\ref{prop: chi conv}: Induction-on-scale}

Proposition~\ref{prop: low} and Proposition~\ref{prop: high} above do not effectively control all bandwidths: there is an intermediate range $2^{D_2(\bdell)} < R < 2^{D_1(\bdell)}$ where the estimates from either proposition are unfavourable. To complete the proof of Proposition~\ref{prop: chi conv}, we use an induction-on-scale mechanism, reminiscent of the \textit{high-low} argument of~\cite{GWZ2020}.

\begin{proof}[Proof (of Proposition~\ref{prop: chi conv})]  For $p = 1$, the result is trivial. Fix $1 < p < p_n$ and let $\bdC_{\gamma} \geq 1$ and $W \geq 1$ be chosen sufficiently large, depending only on $n$, $p$ and $\gamma$, so as to satisfy the forthcoming requirements of the proof. Fixing $3 \leq N \leq n$, our goal is to show 
\begin{equation}\label{eq: induct 1}
    \|\chi_{(C)}^{\bdell} \ast f\|_{L^{p'}(\R^n)} \leq \bdC_{\gamma} C^W 2^{-2D(\bdell)/p'} \|f\|_{L^p(\R^n)} \qquad \textrm{for all $\bdell \in \N_0^{N-M+1}$.}
\end{equation}
We shall prove \eqref{eq: induct 1} by induction on $\|\bdell\|$. \medskip

\noindent \textbf{Base case.} Given $\bdell \in \N_0^{N-M+1}$ for some $3 \leq M \leq N$, observe that
    \begin{equation*}
        \| \chi_{(C)}^{\bdell} \|_{L^{p'/2}(\R^n)} \leq \# \fkN(\bdell) \sup_{\nu \in \fkN(\bdell)} \| \chi_{\bdnu}^{\bdell} \|_{L^{p'/2}(\R^n)} \lesssim_\gamma C^{O(1)} 2^{\|\bl\| - 2D(\bdell)/p'}.
    \end{equation*}
Provided $\bdC_{\gamma} \geq 1$ and $W \geq 1$ are chosen sufficiently large, this, together with Young's convolution inequality, implies that \eqref{eq: induct 1} holds for $\bdell = \boldsymbol{0}$.\medskip

\noindent \textbf{Inductive step.} Fix $\bdell \in \N_0^{N-M+1} \setminus\{\boldsymbol{0}\}$ for some $3 \leq M \leq N$ and assume, as an induction hypothesis, that for all $3 \leq M' \leq N$ the inequality
\begin{equation}\label{eq: induct 2}
    \|\chi_{(C)}^{\bdell'} \ast f\|_{L^{p'}(\R^n)} \leq \bdC_{\gamma}C^W 2^{-2D(\bdell')/p'} \|f\|_{L^p(\R^n)} 
\end{equation}
holds for all $\bdell' \in \N_0^{N-M'+1}$ with $\|\bdell'\| < \|\bdell\|$. Our goal is now to show \eqref{eq: induct 1} holds for our fixed index $\bdell$.

Let $K \in \N$ be a parameter, which we shall later choose to be large, depending only on $n$, $\gamma$ and $p$. Note that $\max\{2/p - 1, 2/p'\} \leq 1$. By Proposition~\ref{prop: low} and Proposition~\ref{prop: high}, there exists some $W_1 >0$ such that
    \begin{equation} \label{eq: induct 3}
        \left\| \left(P_R \chi_{(C)}^{\bdell}\right) * f \right\|_{L^{p'}(\R^n)}
        \lesssim_{\gamma} C^{W_1} 2^K 2^{-2D(\bdell)/p'}  \| f \|_{L^p(\R^n)}
    \end{equation}
holds if either $0 < R \leq \min\{2^{D_2(\bdell) + K}, 2^{D_1(\bdell)}\}$ or $R \geq \max\{2^{D_1(\bdell) - K}, 2^{D_2(\bdell)}\}$. Thus, it remains to consider the intermediate range $2^{D_2(\bdell) + K} < R < 2^{D_1(\bdell) - K}$.\footnote{This range could well be empty, in which case we immediately pass to \eqref{eq: induct 9} below.} For this, we appeal to the induction hypothesis.

Fix $2^{D_2(\bdell) + K} < R < 2^{D_1(\bdell) - K}$. Consider the set
\begin{equation*}
   \cI(\bdell) \coloneq \left\{ \tbdell \in \N_0^{N - M + 1} : \tilde{\ell}_j \leq \ell_j \textrm{ for $M \leq j \leq N$} \right\}. 
\end{equation*}
We pick the largest element $\tbdell \in \cI(\bdell)$ satisfying $R  \geq 2^{D_1(\tbdell)}$, with respect to lexicographical order. It then follows that there exists some $M \leq M' \leq N$ and $0 \leq \tilde{\ell}_{M'} \leq \ell_{M'}$ such that $\tbdell = (0, \dots, 0, \tilde{\ell}_{M'}, \ell_{M'+1}, \dots, \ell_N)$. Furthermore, $D_1(\tbdell) < D_1(\bdell) - K$ and so
\begin{equation}\label{eq: induct 4}
    D_1(\bdell) - D_1(\tbdell) > K.
\end{equation}

Write $\bdell' \coloneq (\tilde{\ell}_{M'}, \ell_{M'+1}, \dots, \ell_N)$. By Lemma~\ref{lem: Lambda contain}, we have $\Lambda_{\gamma}(\bdell; C) \subseteq \Lambda_{\gamma}(\bdell'; 2C)$. Provided $2^K \geq 8BC$ where $B = O_{\gamma}(1)$ is chosen minimally so that $\gamma \in \fkG(B)$, we have 
\begin{equation}\label{eq: induct 5}
\rho \coloneq R^{-1} \leq 2^{-K} 2^{-D_2(\bdell)} \leq (8BC)^{-1} 2^{-D_j(\bdell')}  \quad \textrm{for $M' \leq j \leq N$.}
\end{equation}
Furthermore, $\rho \leq 2^{-D_1(\tbdell)} = 2^{D_1(\bdell')} \leq 2^{-D_j(\bdell')}$ for $1 \leq j \leq n$. It follows that 
\begin{equation*}
    \supp{} P_R \chi^{\bdell} \subseteq N(\Lambda_{\gamma}(\bdell'; 2C), \rho) \subseteq \Lambda_{\gamma}(\bdell'; C') \qquad \textrm{for $C' = O_{\gamma}(C)$.}
\end{equation*}
 Since $R > 2^{D_2(\bdell)}$, recalling the definition in \eqref{eq: M}, we have
\begin{equation*}
    M(\infty; \bdell; R)  = ( 1 + R^{-1} 2^{D_{M-1}(\bdell)}) R 2^{-D_1(\bdell)} \lesssim  R 2^{-D_1(\bdell)},
\end{equation*}
where we have used the condition $M \geq 3$ to bound $2^{D_{M-1}(\bdell)} \leq 2^{D_2(\bdell)} \leq R$. Combining this with Lemma~\ref{lem: low}, there exists some $W_2 > 0$ depending only on $n$ such that
    \begin{equation*}
        |P_R \chi_{(C)}^{\bdell}(x)| \lesssim_\gamma C^{W_2} R 2^{-D_1(\bdell)} \bdone_{\Lambda_{\gamma}(\bdell'; C')}(x).
    \end{equation*} 
    By the maximality of $\tbdell$, we have $R \lesssim 2^{D_1(\tbdell)} = 2^{D_1(\bdell')}$. Thus,
    \begin{equation}\label{eq: induct 6}
        \| P_R \chi_{(C)}^{\bdell} * f \|_{L^{p'}(\R^n)} \lesssim_{\gamma}  C^{W_2} 2^{-(D_1(\bdell) - D_1(\bdell'))} \| \bdone_{\Lambda_{\gamma}(\bdell'; C')} * |f| \|_{L^{p'}(\R^n)}.
    \end{equation}
    Since $\|\bdell'\| < \|\bdell\|$, we may apply our induction hypothesis \eqref{eq: induct 2} to deduce that
    \begin{equation}\label{eq: induct 7}
       \| \chi_{(C')}^{\bdell'} * |f| \|_{L^{p'}(\R^n)} \leq  \bdC_{\gamma}(C')^W 2^{-2D(\bdell')/p'} \|f\|_{L^p(\R^n)}.
    \end{equation}
    Combining \eqref{eq: induct 6} and \eqref{eq: induct 7} with the pointwise bound $\bdone_{\Lambda_{\gamma}(\bdell'; C')} \leq \chi_{(C')}^{\bdell'}$, we obtain
    \begin{equation*}
        \| P_R \chi_{(C)}^{\bdell} * f \|_{L^{p'}(\R^n)} \lesssim_{\gamma}  C^{W_2} \bdC_{\gamma} (C')^W  2^{-(D_1(\bdell) - D_1(\bdell'))}  2^{-2D(\bdell')/p'} \|f\|_{L^p(\R^n)}.
    \end{equation*}
    Note that $D_j(\bdell') = D_j(\tbdell)$ for $1 \leq j \leq n$. By Lemma~\ref{lem: D exp} and \eqref{eq: induct 4}, we have
    \begin{equation*}
     2^{-(D_1(\bdell) - D_1(\bdell'))}  2^{-2D(\bdell')/p'} =   2^{-(D_1(\bdell) - D_1(\tbdell))}  2^{-2D(\tbdell)/p'} \leq 2^{-\delta(p)K} 2^{-2D(\bdell)/p'}
    \end{equation*}
    and so
    \begin{equation}\label{eq: induct 8}
        \| P_R \chi_{(C)}^{\bdell} * f \|_{L^{p'}(\R^n)} \lesssim_{\gamma}  C^{W_2}  2^{-\delta(p)K} \bdC_{\gamma} (C')^W 2^{-2D(\bdell)/p'} \|f\|_{L^p(\R^n)}.
    \end{equation}

Since $2 \leq p' < \infty$, the Littlewood--Paley formula \eqref{eq: LP 1} holds in the sense of $L^{p'}$ convergence. This, together with Minkowski's inequality, gives
\begin{equation*}
    \left\|\chi_{(C)}^{\bdell} \ast f \right\|_{L^{p'}(\R^n)} \leq \int_1^2\left\| \sum_{k \in \Z} \left(P_{t2^k}\chi^{\bdell}\right) \ast \big( P_{t2^k}f \big) \right\|_{L^{p'}(\R^n)}\frac{\ud t}{t}.
\end{equation*}
We may therefore use the classical Littlewood--Paley inequality and a second application of Minkowski's inequality (again using $p' \geq 2$) to deduce that
\begin{equation}\label{eq: induct 9}
    \left\|\chi_{(C)}^{\bdell} \ast f \right\|_{L^{p'}(\R^n)} \lesssim_p  \int_1^2\left(\sum_{k \in \Z}\left\|  \left(P_{t2^k}\chi^{\bdell}\right) \ast \big( P_{t2^k}f \big)\right\|_{L^{p'}(\R^n)}^2\right)^{1/2}\frac{\ud t}{t}.
\end{equation}
Similarly, since $1 < p \leq 2$, for all $1 \leq t \leq 2$ we have
\begin{equation}\label{eq: induct 10}
   \left(\sum_{k \in \Z}\left\| P_{t2^k}f\right\|_{L^{p}(\R^n)}^2\right)^{1/2} \leq \left\| \left(\sum_{k \in \Z} |P_{t2^k}f|^2 \right)^{1/2}\right\|_{L^{p}(\R^n)} \lesssim \|f\|_{L^p(\R^n)}.
\end{equation}
 Comparing the constants in \eqref{eq: induct 3} and \eqref{eq: induct 8}, let
\begin{equation*}
    M_{\gamma, p} \coloneq \max\left\{2^K C^{W_1 - W}\bdC_{\gamma}^{-1}, 2^{-\delta(p)K} (C'/C)^W C^{W_2} \right\}. 
\end{equation*}
Thus, combining \eqref{eq: induct 9} and \eqref{eq: induct 10} with our frequency localised estimates,
\begin{equation}\label{eq: induct 11}
  \|\chi_{(C)}^{\bdell} \ast f \|_{L^{p'}(\R^n)} \leq C_{\gamma, p}  M_{\gamma, p} \bdC_{\gamma}C^W 2^{-2D(\bdell)/p'} \|f\|_{L^p(\R^n)},
\end{equation}
where here $C_{\gamma, p} \geq 1$ is a constant depending only on $n$, $\gamma$  and $p$, which arises as an amalgamation of the various implied constants featured above. 

We are free to choose $K$ large, depending on $p$, $n$, $\gamma$ and $C$, so as to ensure both \eqref{eq: induct 5} and
\begin{equation*}
    1/2 \leq C_{\gamma, p}2^{-\delta(p)K} C^{W_2}(C'/C)^W \leq 1
\end{equation*}
hold. We now choose $W \coloneq W_1 + W_2/\delta(p)$ so that
\begin{equation*}
    2^K C^{W_1 - W}\bdC_{\gamma}^{-1} \lesssim_{\gamma} (C'/C)^{W/\delta(p)}\bdC_{\gamma}^{-1}.
\end{equation*}
Finally, since $C'/C = O_{\gamma}(1)$, we are free to choose $\bdC_{\gamma} \geq 1$ large, depending only on $n$, $p$ and $\gamma$, so as to ensure 
 \begin{equation*}
    C_{\gamma, p}2^KC^{W_1 - W}\bdC_{\gamma}^{-1} \leq 1. 
 \end{equation*}
With these choices, $C_{\gamma, p} M_{\gamma, p}  \leq 1$ and \eqref{eq: induct 11} implies that \eqref{eq: induct 1} holds for our choice of $\bdell$. This closes the induction and concludes the proof. 
\end{proof}




\subsection{Proof of Proposition~\ref{prop: Lambda conv}} Here we derive Proposition~\ref{prop: Lambda conv} as a consequence of Proposition~\ref{prop: chi conv}. 

\begin{proof}[Proof (of Proposition~\ref{prop: Lambda conv})] Let $\bdell = (\ell_M, \dots, \ell_N) \in \N_0^{N - M + 1}$ for $1 \leq M \leq N \leq n$. Recall from Definition~\ref{dfn: geometric kernel} that $\bdone_{\Lambda_{\gamma}(\bdell; C)} \leq \chi^{\bdell}_{(C)}$. Furthermore, if $M \geq 3$, then recall from \eqref{eq: D and E} that $D(\bdell) = E(\bdell)$. Thus, in the case $M \geq 3$, it immediately follows from positivity and Proposition~\ref{prop: chi conv} that
\begin{equation*}
        \|\bdone_{\Lambda_{\gamma}(\bdell; C)} * f \|_{L^{p'}(\R^n)} \leq \|\chi_{(C)}^{\bdell} * |f| \|_{L^{p'}(\R^n)} \lesssim_{\gamma, C} 2^{-2E(\bdell)/p'} \| f \|_{L^p(\R^n)}.
\end{equation*}

It remains to establish the case $M \in \{1,2\}$, which is a consequence of the above result. Assuming $M \in \{ 1, 2 \}$, define $\bdell' = (\ell_3, \dots, \ell_N) \in \N_0^{N - 3 + 1}$. By Lemma~\ref{lem: Lambda contain}, we have $\Lambda_{\gamma}(\bdell; C) \subseteq \Lambda_{\gamma}(\bdell'; 2C)$. Using positivity and applying the previous case of the proposition,
\begin{equation*}
    \|\bdone_{\Lambda_{\gamma}(\bdell; C)} * f \|_{L^{p'}(\R^n)} \leq \|\bdone_{\Lambda_{\gamma}(\bdell'; 2C)} * |f| \|_{L^{p'}(\R^n)}\lesssim_{\gamma, C} 2^{-2E(\bdell')/p'} \| f \|_{L^p(\R^n)}.
\end{equation*}
Recalling \eqref{eq: D and E}, we see that $E(\bdell') = E(\bdell)$, concluding the proof.
\end{proof}




\section{Proof of \texorpdfstring{$L^p$-improving}{Lp improving} IV: Conclusion}\label{sec: Lp improv IV}

We now draw together results from \S\S\ref{sec: Lp improv I}--\ref{sec: Lp improv III} to conclude the proof of Theorem~\ref{thm: symbolconvolutionbound}. For the reader's convenience, we reproduce the statement of this theorem below.

\symbolconvolutionbound*

Recall, in \S\ref{subsec: duality and interpolation}, we showed that Theorem~\ref{thm: symbolconvolutionbound} implies Theorem~\ref{thm: main}. 

\begin{proof}[Proof (of Theorem~\ref{thm: symbolconvolutionbound})] Suppose $a \in \cS(A)$ and $\gamma \in \fkG(B)$ for some $A$, $B \geq 1$. We break the argument into steps.\medskip

\noindent \textit{Step 1: Symbol decomposition}. We apply Lemma~\ref{lem: initial decomp} to write $a = \sum_{k=0}^{\infty} a^k$ where 
\begin{enumerate}[\normalfont(a)]
    \item $a^0 \in C_t^{\infty}L_x^{\infty}(\R \times \R^n)$ satisfies
    \begin{equation*}
        \|\cK_{\gamma}[a^0]\|_{L^{\infty}(\R^n)} \lesssim_A 1 \qquad \textrm{and} \qquad \supp{} a^0 \subseteq [-1,1] \times \{x \in \R^n : |x| \leq 2\};
    \end{equation*}
    \item $a^k \in \cS_{\gamma}(O(A); n; 2^k)$ for all $k \in \N$. 
\end{enumerate}
By Proposition~\ref{prop: fine decomp}, there exists some $C \geq 1$ and $K \in \N$ with $C = O_{\gamma, A}(1)$ and $K = O_{\gamma, A}(1)$ such that the following holds. For $\cL(k)$ as defined in \eqref{eq: indicies}, we have
\begin{equation*}
    \cK_{\gamma}[a] = \sum_{k=0}^{K-1} \cK_{\gamma}[a^k] + \sum_{k = K}^{\infty} \sum_{\bdell \in \cL(k)} \cK_{\gamma}[a^{k, \bdell}]
\end{equation*}
where, for $k \geq K$ and $\bdell \in \cL(k)$, the symbols $a^{k, \bdell}$ satisfy\footnote{Here we may revert back to the original definition of the $\Lambda_{\gamma}(k; \bdell; C)$ sets as in Definition~\ref{dfn: plate family}.} 
\begin{enumerate}[\normalfont(i)]
    \item Kernel estimate: $\|\cK_{\gamma}[a^{k, \bdell}]\|_{L^{\infty}(\R^n)} \lesssim_{\gamma, A} 2^{-\|\bdell\|}$;\vspace{2pt}
 \item Support condition: $ \supp{} \cK_{\gamma}[a^{k, \bdell}] \subseteq \Lambda_{\gamma}(k; \bdell; C)$.
\end{enumerate}

Since the supports of the symbols $a^k$ are finitely-overlapping for $k \in \N_0$, by the Cauchy--Schwarz inequality,
\begin{equation}\label{eq: concluding 1}
    |\cK_{\gamma}[a]|^2 \lesssim_{\gamma} \sum_{k=0}^{K-1} |\cK_{\gamma}[a^k]|^2 + \sum_{k = K}^{\infty} k^n \sum_{\bdell \in \cL(k)} |\cK_{\gamma}[a^{k, \bdell}]|^2,
\end{equation}
where we have also noted that $\#  \cL(k) \lesssim k^n$ for each $k \geq K$.\medskip

\noindent \textit{Step 2: Domination by geometric averaging operators}. The kernel estimate and support conditions for $\cK_{\gamma}[a^{k, \bdell}]$, together with \eqref{eq: concluding 1}, imply
\begin{equation} \label{eq: concluding 2}
    |T_{\gamma}[a]f| \lesssim_{\gamma, A} \sum_{k = 0}^{K-1} T_{\gamma}[a^k] |f| + \sum_{k = K}^{\infty} k^n \sum_{\bdell \in \cL(k)} 2^{-2 \|\bl\|} \, \bdone_{\Lambda_{\gamma}(k; \bdell; C)} * |f|.
\end{equation}
We individually estimate the operators corresponding to the terms of this sum. \medskip

\noindent \textit{Step 3: Estimating the geometric averaging operators}. For each $1 \leq k \leq K-1$, the kernel $\cK_{\gamma}[a^k]$ satisfies $\| \cK_{\gamma}[a^k] \|_{L^\infty(\R^n)} \lesssim_A 1$ and is supported in a ball of radius $O_{\gamma, A}(1)$. Consequently, we may trivially estimate
\begin{equation}\label{eq: concluding 3}
    \|T_{\gamma}[a^k]|f|\|_{L^{p'}(\R^n)} \lesssim_{\gamma, A} \|f\|_{L^p(\R^n)} \qquad \textrm{for $0 \leq k \leq K-1$.}
\end{equation}
In remains to analyse the terms in \eqref{eq: concluding 2} corresponding to $k \geq K$.

If $\bdell \in \N_0^{n-M+1}$ for $1 \leq M \leq n$, then Proposition \ref{prop: Lambda conv} implies that
    \begin{equation}\label{eq: concluding 4}
        \| \bdone_{\Lambda_{\gamma}(\bdell; C)} * f \|_{L^{p'}(\R^n)} \lesssim_{\gamma} 2^{- 2 E(\bdell)/p'} \| f \|_{L^{p}(\R^n)}.
    \end{equation} 
    Let $\bdell \in \cL(k)$ and note that
    \begin{equation*}
      \bdone_{\Lambda_{\gamma}(k; \bdell; C)} * f(x) = 2^{nk}\, ( \bdone_{\Lambda_{\gamma}(\bdell; C)} * f_{-k} )(2^{-k}x) \quad \textrm{where} \quad f_{-k}(y) \coloneq f(2^ky).
    \end{equation*}
    Applying this to \eqref{eq: concluding 4} together with a change of variables,
    \begin{equation} \label{eq: concluding 5}
        \| \bdone_{\Lambda_{\gamma}(k; \bdell; C)} * f \|_{L^{p'}(\R^n)} \lesssim_{\gamma} 2^{2kn/p'} 2^{-2 E(\bdell)/p'} \| f \|_{L^p(\R^n)}.
    \end{equation}

    Given $\bdell \in \cL(k)$, we claim that 
    \begin{equation}\label{eq: concluding 6}
        \frac{2(kn - E(\bl))}{p'} \leq 2 \|\bl\| - \delta(p) k + O(1) \quad \textrm{for} \quad \delta(p) \coloneq \frac{2}{p_n'} - \frac{2}{p'} > 0.
    \end{equation} 
    Assuming this is the case, it then follows from \eqref{eq: concluding 5} that
    \begin{equation} \label{eq: concluding 7}
        \| \bdone_{\Lambda_{\gamma}(k; \bdell; C)} * f \|_{L^{p'}(\R^n)} \lesssim_{\gamma, A} 2^{-k \delta(p)} 2^{2 \|\bl\|} \| f \|_{L^{p}(\R^n)}
    \end{equation}
    Taking $L^p$-norms of both sides of \eqref{eq: concluding 2} and apply \eqref{eq: concluding 3} and \eqref{eq: concluding 7}, we conclude that
    \begin{align*}
        \| T_{\gamma}[a] f \|_{L^{p'}(\R^n)} &\lesssim_{\gamma, A} \|f\|_{L^p(\R^n)} +  \sum_{k = K}^\infty k^n \sum_{\bdell \in \cL(k)} 2^{-2 \|\bl\|} \| \bdone_{\Lambda_{\gamma}(k; \bdell; C)} * f \|_{L^{p'}(\R^n)}\\
        &\lesssim \left( 1 + \sum_{k = K}^\infty k^n 2^{- \delta(p) k} \right) \| f \|_{L^{p}(\R^n)} \lesssim_p \| f \|_{L^{p}(\R^n)},
    \end{align*}
which is the desired estimate.\medskip

\noindent \textit{Step 4: Bounding the exponents.} It remains to prove \eqref{eq: concluding 6}. For this, we write 
\begin{equation}\label{eq: concluding 8}
    \frac{2(kn - E(\bl))}{p'} = \frac{2(kn - E(\bl))}{p_n'} - \delta(p)(kn - E(\bl))
\end{equation}
and estimate both terms on the right-hand side individually. 

Given $1 \leq M \leq n$ and $\bdell = (\ell_M, \dots, \ell_n) \in \cL(M, n; k)$, by Definition~\ref{dfn: indices 2} () we have 
\begin{equation}\label{eq: concluding 9}
k - \sum_{j=M+1}^N j \ell_j \leq M\ell_M = ML_M(N; k, \ell_{M+1}, \dots, \ell_N) <
k - \sum_{j=M+1}^N j \ell_j + M,   
\end{equation}
where the $M$ term on the right-hand side corresponds to a potential `rounding error' from the ceiling function. Recalling the definition \eqref{eq: D and E} and applying the first inequality in \eqref{eq: concluding 9} to bound $k$, we have
  \begin{equation*}
        kn - E(\bl) \leq \sum_{j=M}^2 n j \ell_j + \sum_{j = \max\{M,3\}}^n \frac{j(2n-j+1)}{2} \ell_j \leq \frac{n(n+1)}{2}\|\bdell\| = p_n'\|\bdell\|;
    \end{equation*}
here the second inequality follows by noting that $n < 2n \leq n(n+1)/2$, since $n \geq 3$, and that $x(2n-x+1)$ is increasing on the domain $3 \leq x \leq n$. Consequently, 
\begin{equation}\label{eq: concluding 10}
    \frac{2(kn - E(\bl))}{p_n'} \leq 2 \|\bdell\|.
\end{equation}

On the other hand, by the second inequality in \eqref{eq: concluding 9}, we have
\begin{equation*}
    E(\bdell) = \sum_{j=\max\{3, M\}}^n \frac{j(j-1)}{2} \ell_j \leq \frac{n-1}{2}\sum_{j=M}^n j \ell_j < (n-1)(k + M). 
\end{equation*}
Thus, $kn - E(\bdell) > k - (n-1)M$ and so 
\begin{equation}\label{eq: concluding 11}
    \delta(p)(kn - E(\bdell)) > k\delta(p) - O(1).
\end{equation}

Substituting \eqref{eq: concluding 10} and \eqref{eq: concluding 11} into \eqref{eq: concluding 8} gives \eqref{eq: concluding 6}, as required.
\end{proof}




\section{Necessary conditions} \label{sec: nec}

In this section, we shall prove the following proposition, which establishes necessary conditions for orthonormal restriction bounds for curves.

\begin{proposition}\label{prop: nec cond}
    Suppose $2 \leq q \leq \infty$ and $1 \leq r \leq \infty$, and that the restriction estimate 
    \begin{equation} \label{eq: nec cond}
        \left\| \sum_{j \in J} |\cE_{\gamma} g_j|^2 \right\|_{L^{q/2}(\R^n)} \lesssim_\gamma \big(\#J\big)^{1/r}
    \end{equation}
    holds whenever $(g_j)_{j \in J}$ are orthonormal functions in $L^2(\R)$ indexed by a finite set~$J$. Then $q \geq n(n+1)$ and $r \leq q/2n$.
\end{proposition}

This result shows that Theorem~\ref{thm: main} is sharp up to the endpoint for $n = 3$, but admits the possibility of improved estimates for $n \geq 4$. To prove the proposition, we show orthonormal restriction estimates imply certain Kakeya multiplicity bounds. In a companion paper \cite{CDH}, we study such Kakeya bounds, and thereby provide further evidence for Conjecture~\ref{conj: CDH}.

Fix a smooth, non-degenerate curve $\gamma \colon [-2,2] \to \R^n$. For $\delta ~\in~(0,1)$, a \textit{$\delta$-plate associated to $\gamma$} is a set of the form
\begin{equation*}
    T = \big\{ x \in \R^n : \big|(x-x_T) \cdot \gamma^{(r)}(s_T)\big| < \delta^{(n-r)/(n-1)} \textrm{ for $1 \leq r \leq n$} \big\}
\end{equation*}
for some $x_T \in \R^n$ and $s_T \in [-1/2,1/2]$. We refer to $x_T$ as the \textit{centre} $T$ and $s_T$ as the \textit{direction} of $T$. If $\T$ is a family of $\delta$-plates associated to $\gamma$, we say $\T$ is \textit{direction-separated} if
\[
|s_{T_1} - s_{T_2}| \geq \delta^{1/(n-1)}\quad \mbox{whenever } T_1, T_2 \in \T \mbox{ are distinct}.
\]
Note that in this case, we must have $\#\T \lesssim \delta^{-1/(n-1)}$. 

\begin{lemma}\label{lem: rest vs Kakeya}
    Let $2 \leq q \leq \infty$ and $1 \leq r \leq \infty$. Suppose the restriction estimate 
    \begin{equation}\label{eq: ON rest}
        \left\| \sum_{j \in J}|\cE_{\gamma} g_j|^2 \right\|_{L^{q/2}(\R^n)} \lesssim_\gamma \big(\#J\big)^{1/r}
    \end{equation}
    holds whenever $(g_j)_{j \in J}$ are orthonormal functions in $L^2(\R)$. Define
    \begin{equation*}
        s \coloneq \frac{q}{2} \qquad \textrm{and} \qquad e(n, s) \coloneq \frac{n^2}{n-1} \cdot \frac{1}{s} - \frac{1}{n-1}. 
    \end{equation*}
    Then for all $\delta\in(0,1)$ the multiplicity bound
    \begin{equation}\label{eq: multiplicity I}
        \left\| \sum_{T \in \T} \bdone_T \right\|_{L^s(\R^n)} \lesssim_\gamma \delta^{e(n, s)} (\# \T)^{1/r}
    \end{equation}
    holds whenever $\T$ is a direction-separated family of $\delta$-plates associated to $\gamma$. 
\end{lemma}

Assuming this lemma, it is a straightforward matter to prove Proposition~\ref{prop: nec cond}. 

\begin{proof}[Proof of Proposition~\ref{prop: nec cond}]
    Suppose \eqref{eq: multiplicity I} holds for some choice of $1 \leq r$, $s < \infty$.  

    Let $\T = \{T\}$ consist of a single $\delta$-plate associated to~$\gamma$. In this case, \eqref{eq: multiplicity I} implies that
    \begin{equation*}
        \delta^{n/(2s)} \lesssim_{\gamma} \delta^{e(n, s)}. 
    \end{equation*}
    Taking $\delta \to 0$, we conclude that $q \coloneq 2s$ must satisfy $q \geq n(n+1)$. In light of Lemma~\ref{lem: rest vs Kakeya}, this implies that \eqref{eq: nec cond} can only hold for $q \geq n(n+1)$

  Now $\T$ consist of maximal direction-separated family of $\delta$-plates associated to $\gamma$ satisfying $x_T = 0$ for all $T \in \T$. It follows that there exists some constant $ 0 < c_{\gamma} < 1$ such that $B(0,c_{\gamma} \delta) \subseteq T$ for all $T \in \T$. So
    \begin{equation*}
        \sum_{T \in \T} \bdone_T (x) \geq \#\T \qquad \textrm{for all $x \in B(0, c_{\gamma}\delta)$.}
    \end{equation*}
    Consequently, \eqref{eq: multiplicity I} implies that
    \begin{equation*}
        \delta^{n/s} \#\T \lesssim_{\gamma} \delta^{e(n, s)} (\#\T)^{1/r} \qquad \textrm{where} \qquad \#\T \sim \delta^{-1/(n-1)}. 
    \end{equation*}
    Taking $\delta \to 0$, we conclude that $q \coloneq 2s$ must satisfy $r \leq q/2n$. In light of Lemma~\ref{lem: rest vs Kakeya}, this implies that \eqref{eq: nec cond} can only hold for $r \leq q/2n$.
\end{proof}

\begin{proof}[Proof of Lemma~\ref{lem: rest vs Kakeya}] We assume that the cutoff $\alpha \in C^{\infty}_c(\R)$ is choses so that $0 \leq \alpha \leq 1$, with $\alpha(s) = 1$ for $|s| \leq 3/4$. Let $0 < c_\gamma < 1/4$ be chosen small enough for the purposes of the forthcoming argument, depending only on $n$ and $\gamma$. Fix a family of direction-separated $\delta$-plates $\T$ associated to $\gamma$. Let $\zeta \in C^{\infty}_c(\R)$ be a non-negative bump function supported on $[-c_\gamma, c_\gamma]$ with $\int_{\R} \zeta(t)^2\, \ud t = 1$. Define $R \coloneq \delta^{-n/(n-1)}$ and let
\begin{equation*}
    g_T(s) \coloneq R^{1/(2n)}e^{-2\pi i R x_T \cdot \gamma(s)} \zeta(R^{1/n}(s - s_T)) \qquad \textrm{for all $T \in \T$.}
\end{equation*}
Then $g_T$ is normalised in $L^2(\R)$. Moreover, since $\T$ is direction-separated, it follows that the functions $(g_T)_{T \in \T}$ have pairwise disjoint supports. Consequently, $(g_T)_{T \in \T}$ is an orthonormal fmaily. Thus, assuming \eqref{eq: ON rest} holds, we have
\begin{equation}\label{eq: rest vs Kak 1}
    \left\| \sum_{T \in \T}|\cE_{\gamma} g_T|^2 \right\|_{L^{q/2}(\R^n)} \lesssim \big(\#\T\big)^{1/r}.
\end{equation}
Fixing $T \in \T$, by a change of variables we have
\begin{equation*}
    |\cE g_T(R x)| = R^{-1/(2n)} \left|\int_{\R} e^{2 \pi i R (x - x_T)\cdot (\gamma(s_T + R^{-1/n} s) - \gamma(s_T))} \zeta(s)\,\ud s\right|. 
\end{equation*}
If $x \in T$, then Taylor's theorem implies that
\begin{equation*}
    \big|(x - x_T)\cdot (\gamma(s_T + R^{-1/n} s) - \gamma(s_T))\big| \lesssim_{\gamma} R^{-1}|s|.
\end{equation*}
Provided $c_\gamma > 0$ is chosen sufficienlty small, 
\begin{equation*}
    \cos\Big(2\pi R(x - x_T)\cdot (\gamma(s_T + R^{-1/n} s) - \gamma(s_T)) \Big) \gtrsim 1 \qquad \textrm{for $x \in T$}.
\end{equation*}
Therefore
\begin{equation*}
    |\cE g_T(R x)| \gtrsim R^{-1/(2n)} \int_{\R}  \zeta(s)\,\ud s \gtrsim_{\gamma} R^{-1/(2n)} \qquad \textrm{for $x \in T$.}
\end{equation*}
Thus, by a change of variables, we deduce that
\begin{equation} \label{eq: rest vs Kak 2}
    \left\| \sum_{T \in \T} \bdone_T \right\|_{L^s(\R^n)} \lesssim_{\gamma}  \delta^{e(n, s)}
  \left\|\sum_{T \in \T}|\cE_{\gamma} g_T|^2\right\|_{L^{q/2}(\R^n)}.
\end{equation}
The desired result now follows by combining \eqref{eq: rest vs Kak 1} and \eqref{eq: rest vs Kak 2}. 
\end{proof}




\bibliography{Reference}
\bibliographystyle{amsplain}

\end{document}